\documentclass[reqno,oneside,final]{amsart}

\usepackage{mathtools, etoolbox} 
\usepackage{esint} 
\usepackage{bbm} 
\usepackage{hyperref}
\usepackage{graphicx}
\usepackage{xargs} 
\usepackage{enumitem}
\usepackage[dvipsnames]{xcolor}
\usepackage{tikz}
\usepackage[textwidth=16cm]{geometry}
\usepackage{subcaption}

\newcommand{\zz}{\mathbb{Z}}
\newcommand{\sph}{\mathbb{S}}
\newcommand{\rr}{\mathbb{R}}
\newcommand{\cc}{\mathbb{C}}
\newcommand{\nn}{\mathbb{N}}
\newcommand{\eps}{\varepsilon}

\newcommand{\vphi}{\varphi}
\newcommand{\loc}{\mathrm{loc}}
\newcommand{\borel}{\mathcal{B}}
\newcommand{\radon}{\mathcal{M}}

\newcommand{\weakstarto}{\overset{*}{\rightharpoonup}}
\newcommand{\weakto}{\rightharpoonup}

\newcommand{\leps}{\abs{\log \eps}}

\newcommand{\afac}{\sqrt{\abs{g}}}
\newcommand{\trunc}{\mathrm{Tr}_{\frac{1}{2}}}

\newcommand{\echar}{\chi(S)}
\newcommand{\power}{p^{(m)}_{\CCC \tau} \BBB}
\newcommand{\powerc}{\bar p^{(m)}_{\CCC \tau}}
\newcommand{\balls}{\mathcal{B}}
\newcommand{\transp}{\mathcal{T}}
\newcommand{\mrestr}{%
  \,\raisebox{-.127ex}{\reflectbox{\rotatebox[origin=br]{-90}{$\lnot$}}}\,%
}
\newcommand{\as}{\mathcal{AS}}
\newcommand{\asm}{\mathcal{AS}^{(m)}}

\newcommand{\lsm}{\mathcal{LS}^{(m)}}
\newcommand{\renm}{\mathcal{W}^{(m)}}
\newcommand{\ren}{\mathcal{W}}
\newcommand{\defas}{\coloneqq}
\newcommand{\gl}{GL_\eps}
\newcommand{\glm}{GL_\eps^{(m)}}
\DeclareMathOperator{\haus}{\mathcal{H}}
\newcommand{\hodge}{\star}

\newcommand{\glmflat}{\overline{GL}_\eps^{(m)}}
\newcommand{\drot}{\mathrm{d}_{\bar{\mathcal H}}}
\newcommand{\rots}{\bar{\mathcal H}}

\newcommand*{\iprod}[1]{{#1}\llcorner}
\newcommand*{\del}[1]{\frac{\partial}{\partial x^{#1}}}
\newcommand*{\polden}[1]{\sigma_{#1}}
\DeclareMathOperator{\BigO}{\mathcal{O}}
\DeclareMathOperator{\littleo}{o}
\DeclareMathOperator{\dirac}{\delta}
\newcommand*{\indic}[1]{\mathbbm{1}_{#1}}

\newcommand{\NNN}{\color{black}}
\newcommand{\EEE}{\color{black}}
\newcommand{\CCC}{\color{black}} 
\newcommand{\BBB}{\color{black}}

\newcommand*{\di}{\mathop{}\!\mathrm{d}}

\DeclareMathOperator{\var}{var}
\DeclareMathOperator{\diver}{div}
\DeclareMathOperator{\spt}{spt}

\DeclareMathOperator{\vol}{vol}
\DeclareMathOperator{\der}{d}

\DeclareMathOperator{\dist}{dist}

\DeclareMathOperator{\jac}{j}

\DeclareMathOperator{\vort}{\omega}

\DeclareMathOperator{\dg}{dg}

\DeclarePairedDelimiterX\abs[1]||{#1}
\DeclarePairedDelimiterX\norm[1]\lVert\rVert{%
    \ifblank{#1}{\:\cdot\:}{#1}
}
\DeclarePairedDelimiterX\setof[1]\{\}{#1}
\DeclarePairedDelimiterX\sprod[2]\langle\rangle{#1, #2}
\DeclarePairedDelimiterX\floor[1]\lfloor\rfloor{#1}
\DeclarePairedDelimiterX\ceil[1]\lceil\rceil{#1}
\DeclarePairedDelimiterX\round[1]\lfloor\rceil{#1}

\theoremstyle{plain}
  \newtheorem{theorem}{Theorem}
  \newtheorem{proposition}{Proposition}
  \newtheorem{lemma}{Lemma}
  
\theoremstyle{definition}
  \newtheorem{definition}{Definition}
\theoremstyle{remark}
  \newtheorem{remark}{Remark}

\author[R.~Badal]{Rufat Badal}
\address[Rufat Badal]{
  Department of Mathematics \\
  Friedrich-Alexander Universit\"at Erlangen-N\"urnberg \\
  Cauerstr.~11, D-91058 Erlangen, Germany
}
\email{rufat.badal@fau.de}

\author[M.~Cicalese]{Marco Cicalese}
\address[Marco Cicalese]{
  Zentrum Mathematik \\
  Technical University of Munich \\
  Boltzmannstr.~3, D-85747 Garching bei M\"unchen, Germany
}
\email{cicalese@ma.tum.de}

\begin{document}

\title[Renormalized energy between fractional vortices]{Renormalized energy between fractional vortices with topologically induced free discontinuities on 2-dimensional Riemannian manifolds}

\begin{abstract}
On a two-dimensional Riemannian manifold without boundary we consider the variational limit of a family of functionals given by the sum of two terms: a Ginzburg-Landau and a perimeter term. Our scaling allows low-energy states to be described by an order parameter which can have finitely many point singularities (vortex-like defects) of (possibly) fractional-degree connected by line discontinuities (string defects) of finite length. Our main result is a compactness and $\Gamma$-convergence theorem which shows how the coarse grained limit energy depends on the geometry of the manifold in driving the interaction between vortices and string defects.
\end{abstract}

\maketitle

\tableofcontents

\section{Introduction}
In many physical and biological systems low energy states form complex patterns. The latter result from the necessary coexistence of different and often incompatible geometries that characterize the ground states of those systems. Explaining the emergence of such a complexity is a fascinating task which in the last decades has attracted the attention of the mathematical community. The variational methods, combined with ad-hoc rigorous coarse graining procedures, have proved to be very successful tools to obtain detailed information in several cases of interest. They have lead, for instance, to a satisfactory understanding of the energetic mechanism at the basis of phase coexistence. It is worth mentioning the formation of microstructures in austenite-martensite mechanical transformations, micromagnetics, the theory of liquid crystals, fracture mechanics or plasticity theory, to cite a few examples. \\

In the present paper we are interested in the variational analysis of some energy functionals that can drive the emergence and the coexistence of point and line singularities of a vector-valued order parameter defined on a two-dimensional manifold. In their simplest form this type of functionals consist of the sum of a Ginzburg-Landau term, which penalizes point defects, and a perimeter term which penalizes line defects. Functionals of this kind have been rigorously investigated from a mathematical point of view in connection with models of ripple phases coexistence in biological systems in \cite{goldman_glflat}. Slightly different energy functionals leading to point and line defects have been recently investigated in the context of discrete systems to model chirality phase separations in geometrically frustrated spin systems \cite{BCKO1, BCKO2}, the dependence of the energy concentration phenomenon on the rate of divergence of $n$ in the $n$-clock model \cite{COR1, COR2, COR3, COR-S2} or the formation of partial vortices and line defects in modified xy models \cite{badal_genxy}. The latter analysis is also connected to the investigation of orientability issues of the director field of some liquid crystals model as first discussed in \cite{BB,BZ1}. In all such cases the analysis has been carried out in a Euclidean setting. Here, instead, we aim at beginning the extension of the analysis done in the flat setting to the general two-dimensional Riemannian setting. We start this program with the investigation of the energetic model for ripple phases in biological matter (see for instance \cite{BFL, LS, RS, Sack}) and extend some of the results first obtained in \cite{goldman_glflat}. It is our opinion that some of the results obtained here will help advancing the variational theory of spin systems on planar networks recently investigated in \cite{ABC, ACR, BCR, BP} (see also \cite{ABCS} and the references therein) to the case of discrete systems on manifolds. \\

Let $(S,g)$ denote a $2$-dimensional, compact Riemannian manifold without boundary endowed with a metric tensor $g$.
We denote by $SBV(TS)$ the space of special sections of bounded variation of the tangent bundle $TS$, i.e., those vector fields which are tangent to $S$, have bounded variation, and vanishing Cantor part of the distributional derivative.
For $m \in \NNN \nn$ we define the space of admissible vector fields $\asm(S)$ as those $u \in SBV(TS)$ with square integrable approximate gradient $\nabla u$, jump set $\mathcal J_u$ of finite length, and such that for a.e.~point on $\mathcal J_u$ the traces satisfy $(u^+)^m = (u^-)^m$ (here the product is taken in the sense of complex numbers, see Section \ref{sec:vorticity}).
Roughly speaking, the latter condition can be understood as the angle between $u^+$ and $u^-$ being equal to $\frac{2\pi}{m} \mod 2\pi$\NNN. A \BBB rigorous definition of the above spaces can be found in Subsection \ref{sec:comp} and Appendix \ref{sec:proofdecomp}\NNN. \BBB
Given $\eps > 0$ we consider the generalized Ginzburg-Landau functional $\gl \colon \asm(S) \to [0,+\infty)$ defined as
\begin{equation}\label{def_gengl_functional}
  \gl(u) \defas \frac{1}{2} \int_S \abs{\nabla u}^2 + \frac{1}{2\eps^2} (1 - \abs{u}^2)^2 \vol +\, \haus^1_g(\mathcal J_u).
\end{equation}
Here $\vol$ denotes the volume form on $S$ and $\haus^1_g$ the one-dimensional Hausdorff measure induced by the geodesic distance on $S$.
The main goal of this paper is to analyse the asymptotic behavior as $\eps \to 0$ of a renormalization of the above functionals in the spirit of the first order $\Gamma$-convergence (see e.g.~ Theorem 6.1 in \cite{marcello_gl}).
In the Euclidean setting such an analysis has been carried out in \cite{goldman_glflat} for a slightly modified version of (\ref{def_gengl_functional}) and in \cite{badal_genxy} for a related lattice spin model.

Further\NNN more\BBB, notice that for $m = 1$ the space of admissible spins $\as^{(1)}(S)$ simplifies to $W^{1, 2}(TS)$ and that the functional $\gl$ restricted to $\as^{(1)}(S)$ coincides with the one considered in \cite{jerrard_glman}, where a similar asymptotic variational analysis is one of the core results of the paper. An analogous lattice spin version on a Riemannian manifold was investigated in \cite{canevari_xyman}.

In order to explain the main result of this paper we start with some heuristic arguments.
Let us fix $m = 1$ and let $(u_\eps)$ be a sequence of admissible fields with equi-bounded energy, i.e.,  such that $\sup_\eps \gl(u_\eps) \leq C$.
On one hand, for small $\eps$, by the definition of $\gl$, the penalization term in (\ref{def_gengl_functional}) forces the vector field $u_\eps$ to have length close to one while having square integrable gradient.
On the other hand, it is a standard result (see \cite{BZ}) that the space of maps in $W^{1, 2}(TS)$ with unit length is empty if the Euler characteristic $\echar$ of $S$ is different from zero.
In such a case the minimal Ginzburg-Landau energy diverges as $\leps$ as $\eps \to 0$.
Consequently, it is natural to assume a logarithmic energy bound for $(u_\eps)$, namely $\gl(u_\eps) \leq C \leps$ for all $\eps$.
Under this bound it has been proved (see \cite{BBH, jerrard_ballcon, sandier_ballcon, SS}) that the family of vector fields $(u_\eps)$ can have $K \in \nn$ many vortex-like singularities around which $u_\eps$ winds an integer amount of times. 
The winding numbers $d_1, \dots, d_K$ of the $K$ singularities are related to the topology of $S$ by the following formula
\begin{equation}\label{degree_rel_nonfrac}
  d_1 + \dots + d_K = \echar.
\end{equation}
The case of general $m \in \nn$ is similar.
Under the same logarithmic energy bound as above the vector field $u_\eps$ creates $K \in \nn$ fractional vortices at locations $x_1, \, \dots, \, x_K \in S$ with degree $d_1, \, \dots, \, d_K \in \frac{\zz}{m}$, respectively.
Roughly speaking, a vortex of fractional degree $\frac{k}{m}$ for some $k \in \zz$ is a point around which the vector field rotates by an angle of $\frac{2\pi k}{m}$.
As before, the formula (\ref{degree_rel_nonfrac}) remains true.
Furthermore, for any vortex center $x_k$ and $r > 0$ small enough the following energy lower-bound holds true (see Lemma \ref{lem:localized_liminf_frac}):
\begin{equation}\label{localized_liminf_frac_intro}
  \liminf_{\eps \to 0} \left(\gl(u_\eps, B_r(x_k)) - \frac{\abs{d_k}}{m^{\NNN 2}} \NNN \pi \BBB \log \NNN\left(\BBB\frac{r}{\eps}\NNN\right)\BBB\right) \geq \tilde C,
\end{equation}
where $B_r(x_k)$ is the geodesic ball around $x_k$ with radius $r$ and $\tilde C > -\infty$.

In this paper we are interested in the energetic behavior of sequences $(u_\eps) \subset \asm(S)$ whose Ginzburg-Landau energy satisfies
\begin{equation*}
  \gl(u_\eps) \leq \frac{N}{m} \pi \leps + C,
\end{equation*}
for some fixed $N \in \nn$.
By (\ref{localized_liminf_frac_intro}) the energy bound above allows for the creation of vortices of degrees satisfying the \NNN bound \BBB $\abs{d_1} + \dots + \abs{d_K} \leq N$.
This heuristic picture is stated and proved in the first part of the compactness result in Theorem \ref{thm:gamma_convergence}. More precisely, as it is customary in the framework of Ginzburg-Landau energies, to every $u_\eps$ we associate a two-form $\vort(u_\eps)$ (see \eqref{ibp_vorticity}) keeping track of the energy concentration of $u_\eps$ around vortices.
\NNN In the planar setting this agrees with a multiple of the distributional jacobian of $u_\eps$. \BBB
At this stage, in Theorem \ref{thm:gamma_convergence} the compactness properties of $u_\eps$ can only be described via the ones of $\vort(u_\eps)$.
We show that, up to subsequences, $\vort(u_\eps)$ converges in the flat sense towards a finite sum of weighted Dirac deltas $\mu = \sum_k^K d_k \dirac_{x_k}$ with $d_k \in \frac{\zz}{m}$ and satisfying the same degree bound.
In the same theorem we also prove a more refined compactness result when the energy of the vector fields $(u_\eps)$ diverges like $\frac{N}{m} \leps$.
In this case we can find a vector field $u \in SBV(TS)$ such that $u_\eps \to u$ weakly in $SBV^2(\NNN S \setminus \setof{x_1, \dots, x_K}; TS\BBB)$. The limit vector field $u$ has unitary length, $Nm$ fractional vortices, each of degree $\pm\frac{1}{m}$, and it is such that $(u^+)^m = (u^-)^m$ at a.e.~point of the jump set.
Under the same assumptions as in the refined part of the compactness result the following asymptotic lower and upper bound are shown in Theorem \ref{thm:gamma_convergence}:
\begin{equation}\label{gamma_limit_intro}
\begin{aligned}
		\liminf_{\eps \to 0} \gl(u_\eps) - \frac{N}{m} \pi \leps &\geq \renm(u) + \haus^1_g(\mathcal J_u) + N m \gamma_m, \\
    \limsup_{\eps \to 0} \gl(u_\eps) - \frac{N}{m} \pi \leps &\leq \renm(u) + \haus^1_g(\mathcal J_u) + N m \gamma_m,
\end{aligned}
\end{equation}
where $\gamma_m \in \rr$ is the so called core energy defined in \eqref{core_convergence_flat}. This is to be understood as the local energy contribution due to the presence of a vortex of degree $\pm \frac{1}{m}$.
In Lemma \ref{lem:core_convergence_flat} it is show that $\gamma_m$ does not depend on the geometry of $S$.
In fact, it coincide\NNN s \BBB with the Euclidean analog obtained in \cite{goldman_glflat}.
The other two terms in the limit energy depend on the geometry of $S$. The term $\haus^1_g(\mathcal J_u)$ penalizes the length of the jump set of $u$ in terms of its one-dimensional Hausdorff measure induced by the geodesic distance on $S$. The term $\renm(u)$ stands for the renormalized energy that weights the interaction between vortices and, according to its definition, pairs of vortices attract or repel each other (as usual depending on their sign) with a force which scales linearly with the inverse of their geodesic distance.\\

In what follows we highlight the major obstacles we need to overcome in order to prove our result.
The main source of difficulties arises from the possible nontrivial topology as well as geometry of $S$.
A first sign of a nontrivial interplay between the topology and the energy concentration phenomenon appears in the constraint (\ref{degree_rel_nonfrac}). A consequence of this condition is that, in the case $\echar \neq 0$, there exists no global smooth orthonormal frame on $S$. This induces many technical difficulty in most of our proofs. We show here that the 'power map tool' described below cannot be easily exported from the Euclidean to the Riemannian setting.

\begin{figure}[t]
  \begin{subfigure}[t]{.49\textwidth}
    \centering
    \begin{tikzpicture}
    \draw[step=0.7,black] (-3.8499999999999996,-1.75) grid (3.8499999999999996,2.4499999999999997);
    \draw[line width = 0.4mm,ForestGreen] (-1.7499999999999996,0.35) -- (1.7499999999999996,0.35);
    \filldraw[red] (-1.7499999999999996,0.35) circle (0.0875);
    \filldraw[red] (1.7499999999999996,0.35) circle (0.0875);
    \draw[thick,->] (-3.5,-1.4) -- +(-58.28252558853899:0.385);
    \draw[thick,->] (-3.5,-0.7) -- +(-68.86315549695314:0.385);
    \draw[thick,->] (-3.5,0.0) -- +(-82.43799634584472:0.385);
    \draw[thick,->] (-3.5,0.7) -- +(262.4379963458447:0.385);
    \draw[thick,->] (-3.5,1.4) -- +(248.86315549695314:0.385);
    \draw[thick,->] (-3.5,2.0999999999999996) -- +(238.282525588539:0.385);
    \draw[thick,->] (-2.8,-1.4) -- +(-49.96312275332586:0.385);
    \draw[thick,->] (-2.8,-0.7) -- +(-61.00269160404176:0.385);
    \draw[thick,->] (-2.8,0.0) -- +(-78.58317291104123:0.385);
    \draw[thick,->] (-2.8,0.7) -- +(258.58317291104123:0.385);
    \draw[thick,->] (-2.8,1.4) -- +(241.00269160404173:0.385);
    \draw[thick,->] (-2.8,2.0999999999999996) -- +(229.96312275332588:0.385);
    \draw[thick,->] (-2.0999999999999996,-1.4) -- +(-38.43298884680184:0.385);
    \draw[thick,->] (-2.0999999999999996,-0.7) -- +(-46.58991505993214:0.385);
    \draw[thick,->] (-2.0999999999999996,0.0) -- +(-64.90278554613259:0.385);
    \draw[thick,->] (-2.0999999999999996,0.7) -- +(244.90278554613258:0.385);
    \draw[thick,->] (-2.0999999999999996,1.4) -- +(226.58991505993214:0.385);
    \draw[thick,->] (-2.0999999999999996,2.0999999999999996) -- +(218.43298884680183:0.385);
    \draw[thick,->] (-1.4,-1.4) -- +(-24.817731713451344:0.385);
    \draw[thick,->] (-1.4,-0.7) -- +(-26.56505117707799:0.385);
    \draw[thick,->] (-1.4,0.0) -- +(-19.32990412704507:0.385);
    \draw[thick,->] (-1.4,0.7) -- +(199.32990412704507:0.385);
    \draw[thick,->] (-1.4,1.4) -- +(206.56505117707798:0.385);
    \draw[thick,->] (-1.4,2.0999999999999996) -- +(204.81773171345134:0.385);
    \draw[thick,->] (-0.7,-1.4) -- +(-11.749282837976061:0.385);
    \draw[thick,->] (-0.7,-0.7) -- +(-10.900704743175908:0.385);
    \draw[thick,->] (-0.7,0.0) -- +(-5.152423234383022:0.385);
    \draw[thick,->] (-0.7,0.7) -- +(185.15242323438304:0.385);
    \draw[thick,->] (-0.7,1.4) -- +(190.9007047431759:0.385);
    \draw[thick,->] (-0.7,2.0999999999999996) -- +(191.74928283797607:0.385);
    \draw[thick,->] (0.0,-1.4) -- +(0.0:0.385);
    \draw[thick,->] (0.0,-0.7) -- +(-1.4210854715202004e-14:0.385);
    \draw[thick,->] (0.0,0.0) -- +(-1.4210854715202004e-14:0.385);
    \draw[thick,->] (0.0,0.7) -- +(180.0:0.385);
    \draw[thick,->] (0.0,1.4) -- +(180.0:0.385);
    \draw[thick,->] (0.0,2.0999999999999996) -- +(180.0:0.385);
    \draw[thick,->] (0.7,-1.4) -- +(11.749282837976054:0.385);
    \draw[thick,->] (0.7,-0.7) -- +(10.900704743175908:0.385);
    \draw[thick,->] (0.7,0.0) -- +(5.152423234383022:0.385);
    \draw[thick,->] (0.7,0.7) -- +(174.84757676561696:0.385);
    \draw[thick,->] (0.7,1.4) -- +(169.0992952568241:0.385);
    \draw[thick,->] (0.7,2.0999999999999996) -- +(168.25071716202393:0.385);
    \draw[thick,->] (1.4,-1.4) -- +(24.817731713451316:0.385);
    \draw[thick,->] (1.4,-0.7) -- +(26.565051177077983:0.385);
    \draw[thick,->] (1.4,0.0) -- +(19.329904127045054:0.385);
    \draw[thick,->] (1.4,0.7) -- +(160.67009587295496:0.385);
    \draw[thick,->] (1.4,1.4) -- +(153.43494882292202:0.385);
    \draw[thick,->] (1.4,2.0999999999999996) -- +(155.1822682865487:0.385);
    \draw[thick,->] (2.0999999999999996,-1.4) -- +(38.43298884680184:0.385);
    \draw[thick,->] (2.0999999999999996,-0.7) -- +(46.58991505993212:0.385);
    \draw[thick,->] (2.0999999999999996,0.0) -- +(64.90278554613259:0.385);
    \draw[thick,->] (2.0999999999999996,0.7) -- +(115.09721445386741:0.385);
    \draw[thick,->] (2.0999999999999996,1.4) -- +(133.4100849400679:0.385);
    \draw[thick,->] (2.0999999999999996,2.0999999999999996) -- +(141.56701115319817:0.385);
    \draw[thick,->] (2.8,-1.4) -- +(49.96312275332585:0.385);
    \draw[thick,->] (2.8,-0.7) -- +(61.00269160404175:0.385);
    \draw[thick,->] (2.8,0.0) -- +(78.58317291104123:0.385);
    \draw[thick,->] (2.8,0.7) -- +(101.41682708895877:0.385);
    \draw[thick,->] (2.8,1.4) -- +(118.99730839595824:0.385);
    \draw[thick,->] (2.8,2.0999999999999996) -- +(130.03687724667415:0.385);
    \draw[thick,->] (3.5,-1.4) -- +(58.28252558853899:0.385);
    \draw[thick,->] (3.5,-0.7) -- +(68.86315549695314:0.385);
    \draw[thick,->] (3.5,0.0) -- +(82.43799634584472:0.385);
    \draw[thick,->] (3.5,0.7) -- +(97.56200365415528:0.385);
    \draw[thick,->] (3.5,1.4) -- +(111.13684450304686:0.385);
    \draw[thick,->] (3.5,2.0999999999999996) -- +(121.71747441146098:0.385);
    \end{tikzpicture}
  \end{subfigure}
  \begin{subfigure}[t]{.49\textwidth}
    \centering
    \begin{tikzpicture}
    \draw[step=0.7,black] (-3.8499999999999996,-1.75) grid (3.8499999999999996,2.4499999999999997);
    \filldraw[red] (-1.7499999999999996,0.35) circle (0.0875);
    \filldraw[red] (1.7499999999999996,0.35) circle (0.0875);
    \draw[thick,->] (-3.5,-1.4) -- +(-116.56505117707798:0.385);
    \draw[thick,->] (-3.5,-0.7) -- +(-137.72631099390628:0.385);
    \draw[thick,->] (-3.5,0.0) -- +(-164.87599269168945:0.385);
    \draw[thick,->] (-3.5,0.7) -- +(524.8759926916894:0.385);
    \draw[thick,->] (-3.5,1.4) -- +(497.7263109939063:0.385);
    \draw[thick,->] (-3.5,2.0999999999999996) -- +(476.565051177078:0.385);
    \draw[thick,->] (-2.8,-1.4) -- +(-99.92624550665172:0.385);
    \draw[thick,->] (-2.8,-0.7) -- +(-122.00538320808352:0.385);
    \draw[thick,->] (-2.8,0.0) -- +(-157.16634582208246:0.385);
    \draw[thick,->] (-2.8,0.7) -- +(517.1663458220825:0.385);
    \draw[thick,->] (-2.8,1.4) -- +(482.00538320808346:0.385);
    \draw[thick,->] (-2.8,2.0999999999999996) -- +(459.92624550665175:0.385);
    \draw[thick,->] (-2.0999999999999996,-1.4) -- +(-76.86597769360368:0.385);
    \draw[thick,->] (-2.0999999999999996,-0.7) -- +(-93.17983011986428:0.385);
    \draw[thick,->] (-2.0999999999999996,0.0) -- +(-129.80557109226518:0.385);
    \draw[thick,->] (-2.0999999999999996,0.7) -- +(489.80557109226515:0.385);
    \draw[thick,->] (-2.0999999999999996,1.4) -- +(453.1798301198643:0.385);
    \draw[thick,->] (-2.0999999999999996,2.0999999999999996) -- +(436.86597769360367:0.385);
    \draw[thick,->] (-1.4,-1.4) -- +(-49.63546342690269:0.385);
    \draw[thick,->] (-1.4,-0.7) -- +(-53.13010235415598:0.385);
    \draw[thick,->] (-1.4,0.0) -- +(-38.65980825409014:0.385);
    \draw[thick,->] (-1.4,0.7) -- +(398.65980825409014:0.385);
    \draw[thick,->] (-1.4,1.4) -- +(413.13010235415595:0.385);
    \draw[thick,->] (-1.4,2.0999999999999996) -- +(409.6354634269027:0.385);
    \draw[thick,->] (-0.7,-1.4) -- +(-23.498565675952122:0.385);
    \draw[thick,->] (-0.7,-0.7) -- +(-21.801409486351815:0.385);
    \draw[thick,->] (-0.7,0.0) -- +(-10.304846468766044:0.385);
    \draw[thick,->] (-0.7,0.7) -- +(370.3048464687661:0.385);
    \draw[thick,->] (-0.7,1.4) -- +(381.8014094863518:0.385);
    \draw[thick,->] (-0.7,2.0999999999999996) -- +(383.49856567595214:0.385);
    \draw[thick,->] (0.0,-1.4) -- +(0.0:0.385);
    \draw[thick,->] (0.0,-0.7) -- +(-2.842170943040401e-14:0.385);
    \draw[thick,->] (0.0,0.0) -- +(-2.842170943040401e-14:0.385);
    \draw[thick,->] (0.0,0.7) -- +(360.0:0.385);
    \draw[thick,->] (0.0,1.4) -- +(360.0:0.385);
    \draw[thick,->] (0.0,2.0999999999999996) -- +(360.0:0.385);
    \draw[thick,->] (0.7,-1.4) -- +(23.498565675952108:0.385);
    \draw[thick,->] (0.7,-0.7) -- +(21.801409486351815:0.385);
    \draw[thick,->] (0.7,0.0) -- +(10.304846468766044:0.385);
    \draw[thick,->] (0.7,0.7) -- +(349.6951535312339:0.385);
    \draw[thick,->] (0.7,1.4) -- +(338.1985905136482:0.385);
    \draw[thick,->] (0.7,2.0999999999999996) -- +(336.50143432404786:0.385);
    \draw[thick,->] (1.4,-1.4) -- +(49.63546342690263:0.385);
    \draw[thick,->] (1.4,-0.7) -- +(53.130102354155966:0.385);
    \draw[thick,->] (1.4,0.0) -- +(38.65980825409011:0.385);
    \draw[thick,->] (1.4,0.7) -- +(321.3401917459099:0.385);
    \draw[thick,->] (1.4,1.4) -- +(306.86989764584405:0.385);
    \draw[thick,->] (1.4,2.0999999999999996) -- +(310.3645365730974:0.385);
    \draw[thick,->] (2.0999999999999996,-1.4) -- +(76.86597769360368:0.385);
    \draw[thick,->] (2.0999999999999996,-0.7) -- +(93.17983011986423:0.385);
    \draw[thick,->] (2.0999999999999996,0.0) -- +(129.80557109226518:0.385);
    \draw[thick,->] (2.0999999999999996,0.7) -- +(230.19442890773482:0.385);
    \draw[thick,->] (2.0999999999999996,1.4) -- +(266.8201698801358:0.385);
    \draw[thick,->] (2.0999999999999996,2.0999999999999996) -- +(283.13402230639633:0.385);
    \draw[thick,->] (2.8,-1.4) -- +(99.9262455066517:0.385);
    \draw[thick,->] (2.8,-0.7) -- +(122.0053832080835:0.385);
    \draw[thick,->] (2.8,0.0) -- +(157.16634582208246:0.385);
    \draw[thick,->] (2.8,0.7) -- +(202.83365417791754:0.385);
    \draw[thick,->] (2.8,1.4) -- +(237.99461679191648:0.385);
    \draw[thick,->] (2.8,2.0999999999999996) -- +(260.0737544933483:0.385);
    \draw[thick,->] (3.5,-1.4) -- +(116.56505117707798:0.385);
    \draw[thick,->] (3.5,-0.7) -- +(137.72631099390628:0.385);
    \draw[thick,->] (3.5,0.0) -- +(164.87599269168945:0.385);
    \draw[thick,->] (3.5,0.7) -- +(195.12400730831055:0.385);
    \draw[thick,->] (3.5,1.4) -- +(222.27368900609372:0.385);
    \draw[thick,->] (3.5,2.0999999999999996) -- +(243.43494882292197:0.385);
    \end{tikzpicture}
  \end{subfigure}
  \caption{Doubling angles of a stadium configuration. On the left the vector field $u$ which has $2$ fractional vortices of degrees $1/2$ and jumps by $\pi$ along a segment. On the right the vector field $v=u^2$ which has $2$ integer vortices of degrees $1$ and does not jump.}
  \label{fig:stadium}
\end{figure}
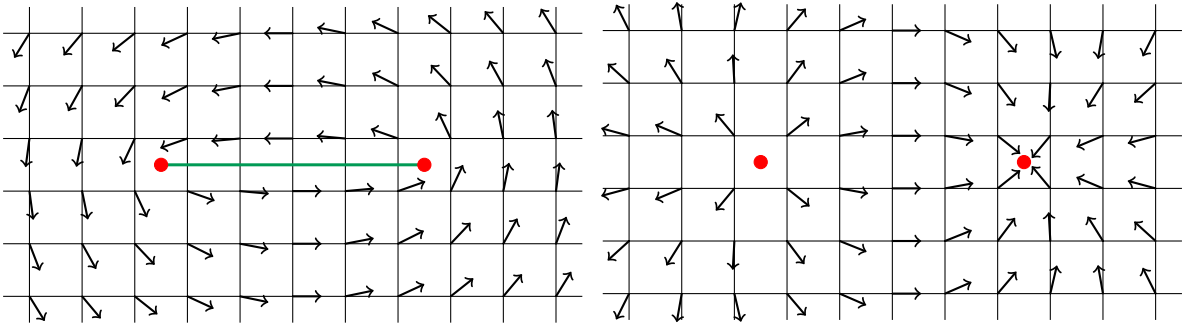

Roughly speaking, in the flat setting, in which a global frame is available, in \cite{goldman_glflat} and \cite{badal_genxy} the authors could (in part) simplify some of the arguments of their proofs by exploiting the power map $p: \NNN\rr^2\BBB\to\rr^2$, $p(x) \NNN\defas\BBB x^m$ (here we identify $\rr^2$ with $\cc$) as a simple tool to show in the fractional setting the analog version of some of the results proved in the classical Ginzburg-Landau theory. To understand how the power map is used one can look at the action of $p$ on a prototype vector field $u$ describing the energy concentration in the flat model (see e.g.~Figure \ref{fig:stadium}). Such a vector field $u$ has $K$ fractional vortices of degree $d_1, \, \dots, \, d_K \in \frac{\zz}{m}$ and it jumps on a segment with its traces on both sides of the jump set having an angular difference of $\frac{2\pi}{m} \zz$. The vector $v \defas p(u)$ does not jump and has vortices at same locations as $u$ but of degrees $m d_1, \, \dots, \, m d_K \in \zz$. Moreover, given $N \in \nn$ and a sequence of vector fields $(u_\eps)$ such that $\gl(u_\eps) \leq \frac{N}{m}\pi \leps + C$ the corresponding transformed sequence $(v_\eps) = (p(u_\eps))$ satisfies the energy bound $\gl(v_\eps) \leq N \pi \leps + C$. Of course these two properties enable one to transfer to the fractional setting many of the results developed in the framework of the integer-degree Ginzburg-Landau theory.
The generalization of such an idea to the manifold is not straightforward\NNN. \BBB
\NNN In order to generalize the map $p$ to the manifold setting one needs a global choice of frame $\{\tau, i\tau\}$ on $TS$\NNN, where $i\tau$ is the rotation of $\tau$ by $\frac{\pi}{2}$\BBB.
But any such frame is forced to develop singularities if and only if $S$ has a non-trivial Euler characteristic.
\NNN Globally applying the map $p$ defined according to such \BBB a singular frame would induce additional ``spurious'' singularities, further complicating the analysis. \BBB
\CCC For instance, in Figure \ref{fig:doubling_man} doubling the angles of the vector field in \ref{fig:start} with respect to the singular frame in \ref{fig:frame} results not only in changing half-vortices into integer vortices, but also in creating an additional vorticity at the blue point in \ref{fig:end}. \BBB
\NNN (We refer to \cite{trivial_connect} for a numerical computation of such global frames.) \BBB
On the one hand, this forces us to localize many of the Euclidean results and eventually to apply a partition of unity argument. On the other hand, even in a coordinate neighborhood $O$ where $(u_\eps)$ is such that $\gl(u_\eps) \leq \frac{\pi}{m} \leps + C$, it is a non-trivial task to show that the sequence of vector fields $v_\eps \defas p(u_\eps)$ satisfies $\gl(v_\eps) \leq \pi \leps + \tilde C$. The main difficulties in pursuing this task are already evident when one considers how the gradient term in the Ginzburg-Landau energy transforms under the map $p$. By the chain rule formula (see Proposition \ref{prop:after_doubling}) it holds that
\begin{equation*}
  \abs{\nabla v_\eps}^2
  = m^2 \abs{\nabla u_\eps}^2 + (1 - m^2) \abs{\di\abs{u_\eps}}^2 + (m - 1)^2 \abs{u_\eps}^2 \abs{\jac(\NNN\tau\BBB)}^2 - 2m (m - 1) \sprod{\jac(u_\eps)}{\jac(\NNN\tau\BBB)}.
\end{equation*}
In the equation above \NNN $\{\tau, i\tau\}$ is a smooth (up to the boundary) frame in $O$ and \BBB $\jac$ denotes the pre-jacobian, which for any vector field $w$ satisfies $\jac(w) \defas \sprod{\nabla w}{i w}$\NNN. \BBB
Notice the last two terms above are peculiar of the manifold setting and need to be uniformly controlled in order to derive the needed energy upper bound for $(v_\eps)$.
In particular, it is non-trivial to show the boundedness of the last term since $\jac(\NNN\tau\BBB) \neq 0$ as the manifold in general is curved and $(\nabla u_\eps)$ is not a priori bounded in $L^1$.
To tackle this problem we need to combine a ball construction argument together with a specific choice of frame $\NNN\tau\BBB$, namely a frame having least Dirichlet energy $\int_O \abs{\nabla \NNN\tau\BBB}^2 \vol$.

\begin{figure}[t]
  \begin{subfigure}[t]{.32\textwidth}
    \centering
    \includegraphics[height=4.75cm]{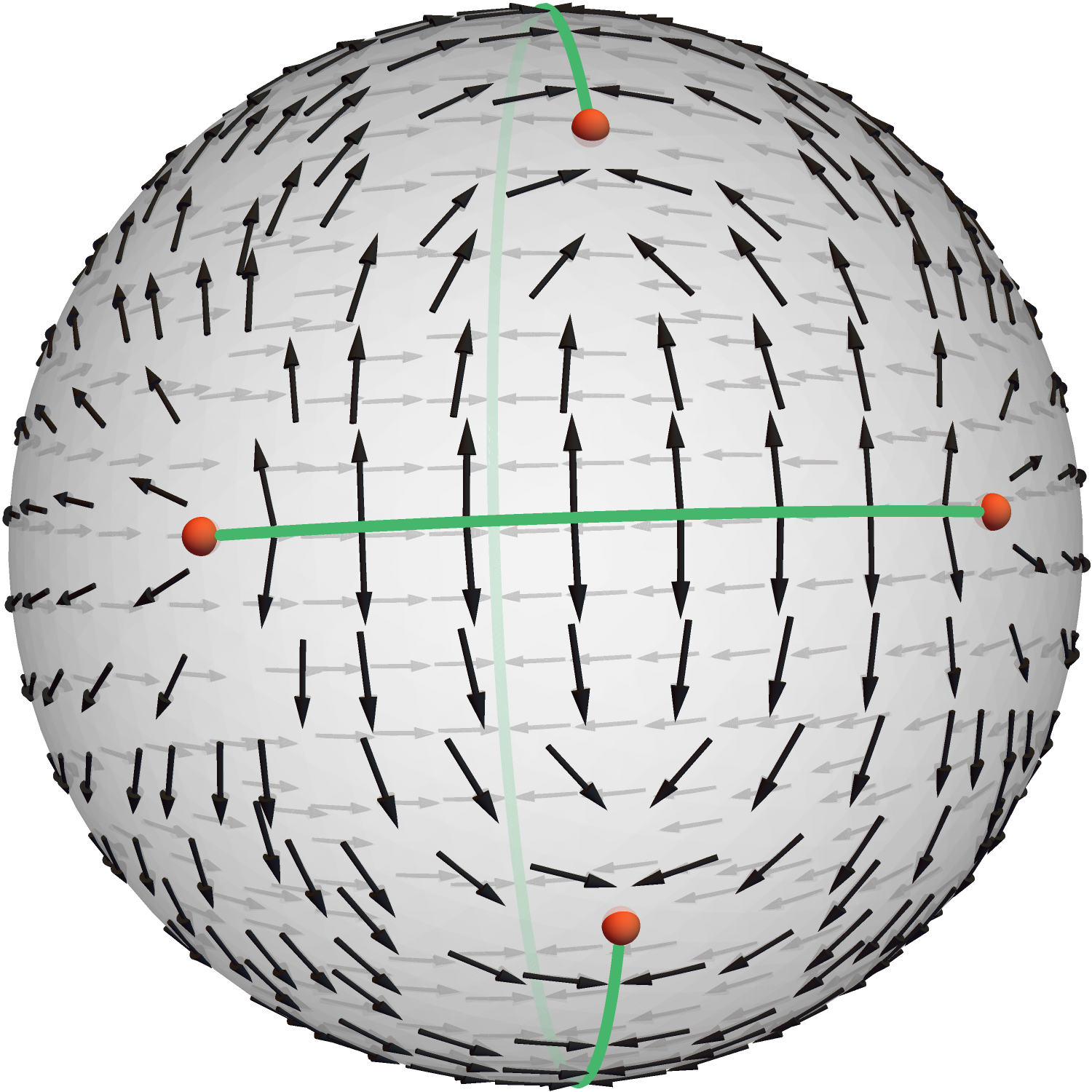}
    \caption{Initial vector field \NNN with $4$ half-vortices (red dots) jumping along geodesic segments (green lines)\BBB}
    \label{fig:start}
  \end{subfigure}
  \begin{subfigure}[t]{.32\textwidth}
    \centering
    \includegraphics[height=4.75cm]{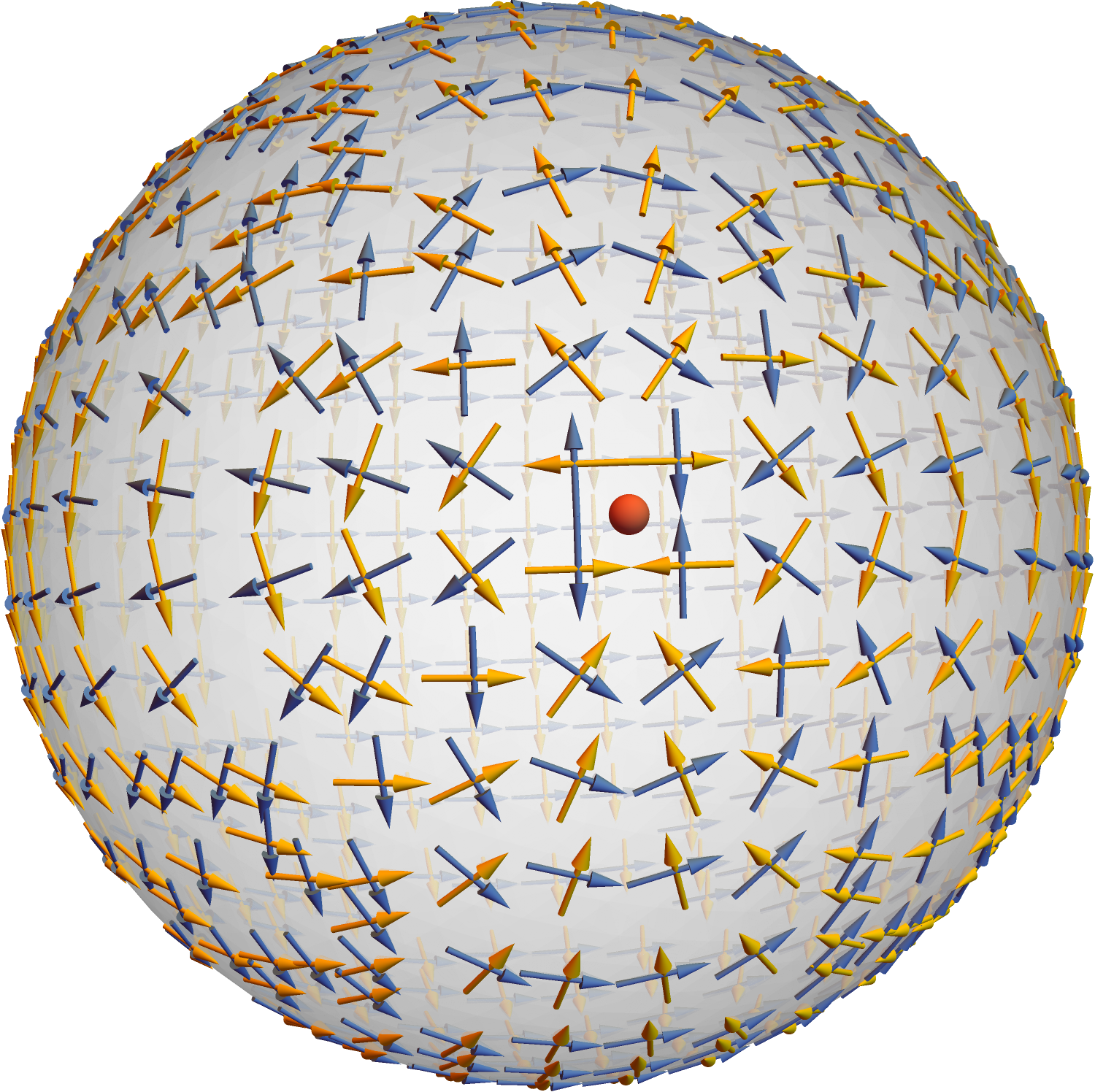}
    \caption{Singular frame \NNN with a double-vortex singularity (red dot)\BBB}
    \label{fig:frame}
  \end{subfigure}
  \begin{subfigure}[t]{.32\textwidth}
    \centering
    \includegraphics[height=4.75cm]{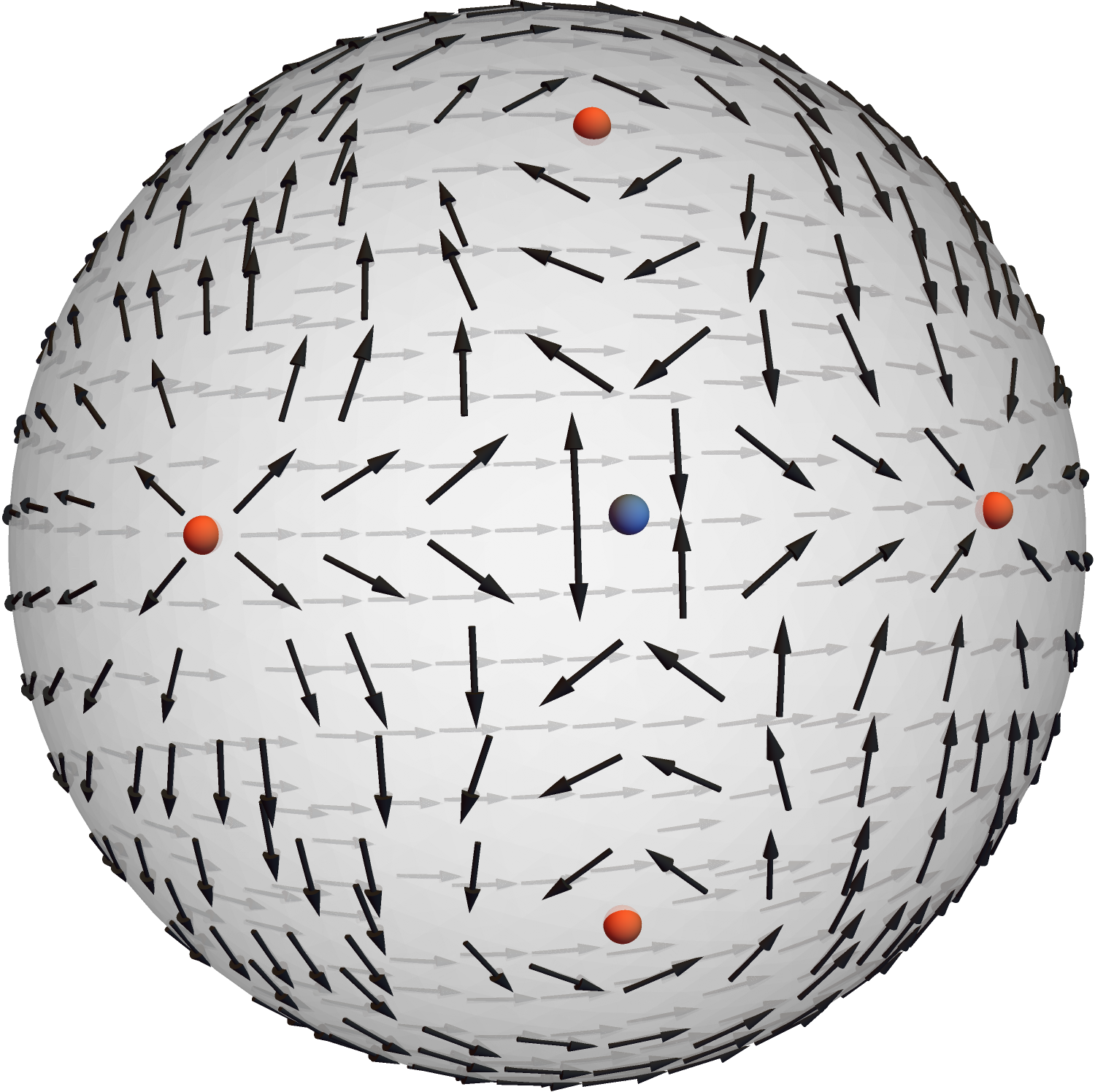}
    \caption{Vector field after doubling \NNN with single-vortices (red dots) and a spurious $(-2)$-vortex (blue dot)}
    \label{fig:end}
  \end{subfigure}
  \caption{Globally doubling angles on a sphere.}
  \label{fig:doubling_man}
\end{figure}

Another source of difficulties is the characterization of the core energy $\gamma_m$ in (\ref{gamma_limit_intro}).
Let us recall the definition of $\gamma_m$ in the Euclidean setting (see for instance \cite{goldman_glflat}).
For any $r > 0$ we have that
\begin{equation*}
  \gamma_m \defas \lim_{\eps \to 0} \left(\bar \gamma_\eps^{(m)}(r) - \frac{\pi}{m^2} \log\NNN\left(\BBB\frac{r}{\eps}\NNN\right)\BBB \right).
\end{equation*}
Here $\gamma_\eps^{(m)}(r)$ is given by the following minimization problem:
\begin{equation*}
  \bar \gamma_\eps^{(m)}(r) \defas \min\setof*{\glmflat(v, B_r(0)) \colon v \in W^{1, 2}(B_r(0); \rr^2), \, v = \frac{x}{\abs{x}} \text{ on } \partial B_r(0)},
\end{equation*}
where
\begin{equation*}
	\glmflat(v, B_r(0))
	\defas \frac{1}{2 m^2} \int_{B_r(0)} \abs{\nabla{v}}^2 
		+ (m^2 - 1) \abs{\nabla \abs{v}}^2
		+ \frac{m^2}{2\eps^2} (1 - \abs{v}^2)^2 \di x.
\end{equation*}
The core energy in the manifold setting can be obtained by comparing $\glmflat$ with its manifold analog $\glm$ defined as
\begin{equation*}
  \glm(v, B_r(x_0)) \defas \frac{1}{2m^2} \int_{B_r(x_0)} \abs{\nabla{v}}^2 
		+ (m^2 - 1) \abs{\der \abs{v}}^2 
		+ \frac{m^2}{2\eps^2} (1 - \abs{v}^2)^2 \vol
\end{equation*}
for any $v \in W^{1, 2}(TB_r(x_0))$.
More precisely, the two functionals can be compared by choosing centered normal coordinates at $x_0$ and an orthonormal frame which results in a coordinate representation $\bar v$ of $v$ satisfying
\begin{equation}\label{glm_compare_intro}
  \abs*{\glm(v, B_r(x_0)) - \glmflat(\bar v, B_r(0))} \leq \NNN C r \BBB \left(1 + \glmflat(\bar v, B_r(0))\right).
\end{equation}
Notice that for a sequence of minimizers (in the Euclidean setting) and fixed $r > 0$ the right-hand side can diverge as $\leps$, making our comparison strategy inefficient.
A possible way out already appeared in \cite{jerrard_glman} and consists in considering a properly chosen sequence of radii $(r_\eps)$ such that $r_\eps \leps \to 0$ \NNN for which the error in \eqref{glm_compare_intro} vanishes\BBB.
\NNN For each $\eps > 0$ one then needs to find $Nm$ disjoint balls of radius $r_\eps$ (cores) such that the Ginzburg-Landau energy in each ball scales like $\frac{\pi}{m^2} \log\left(\frac{r_\eps}{\eps}\right)$.
In the Euclidean setting $r_\eps = r > 0$ is an admissible choice and the task is rather straightforward as one can simply take balls around the limit vortex centers.
On the manifold, instead, one must resort once again to the ball construction in order to find an appropriate choice of ball centers around which a sufficient amount of energy concentrates. Moreover, in contrast to \cite{jerrard_glman} we need to take the $\Gamma$-limit with respect to the $L^1$-convergence which does not only track the defects but also the asymptotic behavior of the vector fields realizing them. Because of that, compared to \cite{jerrard_glman}, further work is invested to assure that the strategy in the Euclidean setting (see e.g.~\cite{goldman_glflat}) generalizes to the case of small cores (see in particular Step 3-5 in the proof of our $\Gamma$-$\liminf$). \BBB

To finish we would like to mention that even the definition of special sections of bounded variation that we roughly introduced at the beginning of this introduction requires some care. More precisely, most of the notions in the statements of this paper benefit from an intrinsic definition of $SBV(TS)$, whose most important properties are stated in Section \ref{sec:bvsecs}, while their proofs are contained in Appendix \ref{sec:proofdecomp}.
Although the notion and the main properties of $BV$ functions are well-understood even in the more abstract setting of metric spaces \NNN (see \cite{A_metric_1, A_metric_2,AMP_metric, pointwise_porps_metric, M_metric}) \BBB and the translation of their finer properties to the Riemannian manifold setting is possible, we have found it more convenient for the reader to derive these results directly. It is worth mentioning that the definition and some of the properties of $BV$ functions on manifolds can be found in other papers as for instance in \NNN \cite{pallara_bvfuncs,pallara_heat}\BBB. To the best of our knowledge however, the derivation of the finer properties of $BV$ functions in the manifold setting was lacking. It is now contained in the Appendix \ref{sec:proofdecomp} where in particular we prove the decomposition theorem \ref{thm:decomp_secs}.
\NNN The key ideas behind its proof is an intrinsic definition of blow-up quantities, the investigation of their relation to the Euclidean ones, and a partition of unity argument. \BBB
For the purposes of the variational analysis contained in this paper it would be sufficient to prove the decomposition theorem only for tangent vector fields or scalar maps on $S$.
Instead, \NNN since the main argument remains unchanged, \BBB we have decided to extend the theorem to the general case of sections of an arbitrary Riemannian vector bundle.

\section{Preliminaries}
\subsection{\NNN Tangent vector fields \BBB of bounded variation}\label{sec:bvsecs}
We wish to provide a definition of \BBB \textit{(special) functions} and \BBB \textit{(special) tangent vector fields} \NNN of bounded variation on \BBB a manifold.
Furthermore, we will state several important results concerning this function spaces.
The missing \CCC proofs can \BBB be found in Appendix \ref{sec:proofdecomp}.
(In the same appendix we will also deal with case of a general vector bundle.)

\NNN Let $n \in \nn \defas \setof{1, 2, \ldots}$. \BBB
We denote by $M$ an $n$-dimensional, oriented Riemannian manifold (with or without boundary) with metric tensor $g$.
The \CCC open \BBB geodesic ball of radius $r$ centered at $x \in M$ will be written as \CCC $B_r(x)$\BBB.
If no confusion is possible, given $x \in \rr^n$ we use the same notation \CCC $B_r(x)$ \BBB to denote the Euclidean open ball.
We will write $r^* = r^*(M)$ for the injectivity radius of $M$ and $\exp_x \colon T_x M \to M$\CCC, where $T_x M$ denotes the tangent space to $M$ at $x$, \BBB for the exponential map at $x \in M$.
The volume form on $M$ will be \CCC denoted by \BBB $\vol$.
We will write $TM$ for the tangent bundle and $T^*\!M$ for the co-tangent bundle of $M$, respectively.
\CCC For t\BBB he covariant derivative \CCC we use the symbol \BBB $\nabla$.
Whenever possible, Einstein summation convention will be used.
Herewith, we will implicitly assume that any index such as $i, \, j, \, \dots$ we encounter is ranging in $\setof{1, \dots, n}$.

Functions spaces $\CCC X\BBB(M; TM)$ of tangent vector fields $u \colon M \to TM$ with regularity prescribed by \CCC $X$ \BBB will be shortly written as $\CCC X\BBB(TM)$.
We further use the notation $\CCC X\BBB_\loc(TM)$ to denote those tangent vector fields $u$ that belong to $\CCC X\BBB(TK)$ for any compact $K \subset M$, \CCC where given any set $A \subset M$ we denote by $TA \defas \bigsqcup_{x \in A} T_xM$. \BBB
For example, $C^\infty_c(TM)$ will be the space of smooth compactly supported sections of $TM$.
For any $\alpha \in [0, n]$, we will denote by $\haus_g^\alpha$ the $\alpha$-dimensional Hausdorff measure on $M$ defined through the metric distance induced by $g$.
Note that for $\alpha = n$ we recover the usual notion of integration on the manifold $M$.
More precisely, for any $f \in C^\infty_c(M)$ we have
\begin{equation*}
	\int_M f \vol = \int_M f \di \haus^n_g.
\end{equation*}
From now on, we will shortly write a.e.~in place of $\haus_g^n$-a.e., when no confusion is possible.
Let $u$ be a $\haus_g^n$-measurable tangent vector field and $p \in [1, \infty]$.
\CCC If \BBB $p \in [1, \infty)$, we define \CCC the space $L^p(TM)$ of those $\haus_g^n$-measurable functions $u \colon M \to TM$ \NNN with finite $\norm{u}_{L^p}$\CCC, namely \BBB
\begin{equation*}
	\norm{u}_{L^p} \defas \int_M \abs{u}^p \NNN \vol \BBB \CCC < \infty \BBB,
\end{equation*}
where $\abs{\cdot}$ denotes the norm induced by $g$.
\CCC For $p = \infty$ we require the following $L^\infty$-norm to be finite:\BBB
\begin{equation*}
	\norm{u}_{L^\infty} \defas \inf \setof*{C \in \rr \colon \abs{u(x)} \leq C \text{ for \NNN a.e.~\BBB} x \in M} \CCC < \infty\NNN.\BBB
\end{equation*}
\NNN The spaces $L^p(M)$ for $p \in [1, \infty]$ are defined similarly. \BBB
\NNN G\BBB iven $\Omega \subset \rr^n$ and $O \subset M$ open sets, we denote by $\Psi \colon \Omega \times \rr^n \to TO$ a local trivialization of $TM$.
We remark that any such $\Psi$ induces a unique coordinate chart as well as a unique frame \NNN which, without further mention, will be denoted by $\Phi \colon \Omega \to O$ and $\setof{\CCC \tau_1, \dots, \tau_n\BBB}$, respectively\BBB.
\NNN $\Psi$, $\Phi$, and $\setof{\tau_1, \ldots, \tau_n}$ are implicitly assumed to be smooth up to the boundary. \BBB

We single out tangent vector fields of \textit{bounded variation} as precisely those elements of $L^1(TM)$ whose \textit{total variation} is finite.
In order to define the total variation in our present setting we will \NNN first need \BBB to introduce a classical object from differential geometry: the \textit{adjoint covariant derivative} $\nabla^*$.
\CCC The latter \BBB is the unique operator $\nabla^* \colon C^\infty(TM \otimes T^*\!M) \to C^\infty(TM)$ such that for all $\CCC u \BBB \in C^\infty_c(TM)$ and $v \in C^\infty_c(TM \otimes T^*\!M)$ the following integration-by-parts formula holds true (see also Proposition 10.1.30 in \cite{nicolaescu_book} for further details):
\begin{equation}\label{ibp_adjoint_grad}
	\int_M \sprod{u}{\nabla^* v} \vol = \int_M \sprod{\nabla u}{v} \vol.
\end{equation}
\CCC Here, $\otimes$ denotes the tensor product. \BBB
The total variation of a section $u \in L^1(TM)$ is then defined as the supremum of the left-hand side of (\ref{ibp_adjoint_grad}) over test-functions bounded by $1$:
\begin{equation}\label{total_variation}
  \var(u) \defas \sup\setof*{\int_M \sprod{u}{\nabla^* v} \, \vol \colon v \in C^\infty_c(TM \otimes T^*\!M), \, \norm{v}_{L^\infty} \leq 1}.
\end{equation}
Notice that in the special case of $M = \CCC \Omega\BBB$\NNN, where $\Omega \subset \rr^n$ is an open set, \BBB equipped with the Euclidean distance\CCC, \BBB the adjoint gradient $\nabla^*$ satisfies for any $i \in \setof{1, \dots, n}$
\begin{equation*}
	-(\nabla^* v)_i = (\diver v)_i \defas \sum_{j = 1}^n \frac{\partial v_{i j}}{\partial x^i},
\end{equation*}
where $v \in C^\infty_c(\Omega; \rr^{n \times n})$.
Hence, the general definition in (\ref{total_variation}) of total variation in the manifold setting agrees with the usual one in Euclidean space (see also (3.4) in \cite{fusco_bible}).
We now introduce a local definition of total variation.
Given an $O \subset M$ open we define the total variation of $u$ in $O$ as
\begin{equation} \label{total_variation_local}
	\var(u, O) \defas \sup\setof*{\int_M \sprod{u}{\nabla^* v} \vol \colon v \in C^\infty_c(TO \otimes T^* O), \, \norm{v}_{L^\infty} \leq 1},
\end{equation}
where $T^*O \defas \CCC\bigsqcup\BBB_{x \in O} T_x^* M$.

We are now ready to define \textit{tangent vector fields of bounded variation}.
\begin{definition}[Tangent vector fields of bounded variation]
	A section $u \in L^1(TM)$ is of bounded variation if and only if $\var(u) < \infty$.
	The set of all such sections wll be denoted by $BV(TM)$.
	It is equipped with the norm
	\begin{equation*}
		\norm{u}_{BV} \defas \norm{u}_{L^1} + \var(u)\CCC. \BBB
	\end{equation*}
	With this norm $BV(TM)$ turns out to be a Banach space.
\end{definition}

We will now introduce the Riesz representation theorem and Radon-Nikodym theorem in our setting.
The former provides a representation of linear bounded functionals on the space of compactly supported continuous sections of $TM$ via $TM\NNN\otimes T^*\!M$-valued Radon measures.
These measures are defined as follows:
\begin{definition}[$TM\NNN\otimes T^*\!M\BBB$-valued Radon measures]\label{def:vector_measures}
	Let $\radon_+(M)$ denote the set of positive (finite) Radon measures on $M$.
	Given $\mu \in \radon_+(M)$ we define as $L^1(TM\NNN\otimes T^*\!M\BBB; \mu)$ the set of measurable sections $\sigma$ of $TM\NNN\otimes T^*\!M\BBB$ such that
	\begin{equation*}
		\int_M \abs{\sigma} \di \mu < \infty. 
	\end{equation*}
	Then\CCC, \BBB the set $\radon(TM\NNN\otimes T^*\!M\BBB)$ of $TM\NNN\otimes T^*\!M\BBB$-valued Radon-measures is \CCC defined as \BBB
	\begin{equation*}
		\radon(TM\NNN\otimes T^*\!M\BBB) \defas \setof{(\sigma, \mu) \colon \mu \in \radon_+(M), \, \sigma \in L^1(TM\NNN\otimes T^*\!M\BBB; \mu)}.
	\end{equation*}
	Note that the pair $(\sigma, \mu)$ will be usually written as $\sigma \mu$.
	Further\CCC more\BBB, for a given $\nu = \sigma \mu \in \radon(TM\NNN\otimes T^*\!M\BBB)$ such that $\abs{\sigma} = 1$ at $\mu$-a.e.~in $M$ we will call $\mu$ the \textit{total variation} of $\nu$ (written as $\abs{\nu}$) and $\sigma$ its \textit{polar density} (written as $\sigma$ or $\polden{\nu}$ if confusion is possible).
	Two measures $\nu, \, \tilde \nu \in \radon(TM\NNN\otimes T^*\!M\BBB)$ are said to be equal if and only if $\abs{\nu} = \abs{\tilde \nu}$ in the sense of measures and $\polden{\nu} = \polden{\tilde\nu}$ at $\abs{\nu}$-a.e.~point in $M$.
	Given $\nu \in \radon(TM\NNN\otimes T^*\!M\BBB)$ and $\mu \in \radon_+(M)$ we use the notation $\nu << \mu$ if $\nu$ \CCC is \BBB \NNN \textit{absolutely continuous} \BBB and $\nu \perp \mu$ if $\nu$ is singular with respect to \CCC $\mu$\BBB.
	A sequence $(\nu_h) \subset \radon(TM\NNN\otimes T^*\!M\BBB)$ weakly* converges towards $\nu \in \radon(TM\NNN\otimes T^*\!M\BBB)$ (shortly written as $\nu_h \weakstarto \nu$) if and only if for all continuous and compactly supported $v \in C_c(TM \otimes T^*\!M)$ it holds that
	\begin{equation*}
		\lim_{h \to \infty} \int_M \sprod{\polden{\nu_h}}{v} \di\abs{\nu_h} \to \int_M \sprod{\polden{\nu}}{v} \di\abs{\nu}.
	\end{equation*}
\end{definition}

\begin{theorem}(Riesz representation for \NNN tangent vector fields\BBB)\label{thm:riesz}
	Let $T \colon C_c(TM\NNN\otimes T^*\!M\BBB) \to \rr$ be a bounded linear functional, then there exists a unique $TM\NNN\otimes T^*\!M\BBB$-valued Radon-measure $\nu \in \radon(TM\NNN\otimes T^*\!M\BBB)$ such that 
	\begin{equation*}
		T(v) = \int_M \sprod{v}{\sigma_\nu} \di \abs{\nu}.
	\end{equation*}
\end{theorem}
\begin{proof}
	We refer the reader to \NNN Theorem \BBB \cite{pallara_bvfuncs} for the proof of the statement in the scalar case.
	The same proof works also for the case of the tangent bundle or, more generally, the case of an arbitrary vector bundle with minor modifications.
\end{proof}

Given $u \in BV(TM)$ we can define a linear functional $T_u \colon C_c^\infty(TM \otimes T^* M) \to \rr$ as follows:
\begin{equation*}
	T_u(v) = \int_M \sprod{u}{\nabla^* v} \NNN \vol \BBB.
\end{equation*}
By the definition of total variation in (\ref{total_variation}) it turns out that $T_u$ is bounded since $\norm{T_u} = \var(u)$.
Due to Theorem \ref{thm:riesz} there exists a unique measure in $\radon(TM \otimes T^* M)$ which we will from now on denote by $Du$ such that for all $v \in C_c^\infty(TM \otimes T^*\!M)$ the following integration-by-parts formula holds true
\begin{equation*}
	\int_M \sprod{u}{\nabla^* v} \NNN \vol \BBB = \int_M \sprod{v}{\sigma_u} \di\abs{Du},
\end{equation*}
where $\sigma_u \defas \sigma_{Du}$ is the polar density of $Du$.

The following Radon-Nikodym decomposition holds true:
\begin{theorem}[Radon-Nikodym]\label{thm:radon_nikodym}
	For any $\nu \in \radon(TM\NNN\otimes T^*\!M\BBB)$ and $\mu \in \radon_+(M)$ there exist \CCC only two \BBB measures $\nu^a, \, \nu^s \in \radon(TM\NNN\otimes T^*\!M\BBB)$ such that $\nu^a << \mu, \, \nu^s \perp \mu$ and $\nu = \nu^a + \nu^s$.

	Furthermore, there exists a unique $\sigma^a \in L^1(TM\NNN\otimes T^*\!M\BBB; \mu)$ such that $\nu^a = \sigma^a \mu$.
\end{theorem}

In the special case of $\nu = Du$ for some $u \in BV(TM)$ and $\mu = \haus_g^n$ in Theorem \ref{thm:radon_nikodym} we will denote $\nu^a$ by $D^a u$ and $\nu^s$ by $D^s u$.

We will now define intrinsic blow-ups of a section $u \in L^1_\loc(TM)$.
An intrinsic definition will involve comparing a vector in the bundle $T_xM$ with another vector in the bundle $T_yM$ for two different points $x, \, y \in M$.
In order to assure invariance under \NNN a \BBB change of coordinates we will employ \textit{parallel transport} on $TM$, which can be briefly described as follows:
Given a smooth curve $\gamma \colon [0, 1] \to M$ with $\gamma(0) = x$ and $\gamma(1) = y$ and \CCC a \BBB vector $v_0 \in T_{\gamma(0)}M$, there exists a unique family $\setof{\NNN P_t^{(\gamma)}\BBB}_{t \in [0, 1]}$ of linear isomorphisms $\NNN P_t^{(\gamma)}\BBB \colon T_{\gamma(0)} M \to T_{\gamma(t) M}$ such that $v(t) \defas \NNN P_t^{(\gamma)}\BBB(v_0)$ satisfies:
\begin{equation*}
	\left\{
	\begin{aligned}
		\nabla_{\dot \gamma(t)} v(t) & = 0,   &  & t \in [0, 1], \\
		v(0)                         & = v_0.
	\end{aligned}
	\right.
\end{equation*}
This notion of transport between $T_xM$ and $T_yM$ depends on the choice of curve $\gamma$.
Nevertheless, for points close enough (more precisely strictly closer than the injectivity radius \NNN of \BBB $M$) we can make the transport unique by taking $\gamma$ as the geodesic between $x$ and $y$.
More precisely, for any $x \in M$ we define the \textit{transport map} $\transp_x \colon B_{r^*}(x) \times T_xM \to TB_r(x)$ from $x \in M$ as:
\begin{equation}
  \label{transport_map}
	\transp_x(y, v) \defas P^{(\gamma_y)}_1(v),
\end{equation}
where $\gamma_y \colon [0, 1] \to M$ is the unique geodesic starting at $x$ and ending at $y$ with constant speed equal to the geodesic distance $\dist_g(x, y)$ between $x$ and $y$.

\begin{definition}[Approximate limit] \label{def:approximate_limit}
	Let $u \in L^1_\loc(TM)$ and let $\transp_x$ be the transport map from $x \in M$ defined in (\ref{transport_map}).
	We say that $u$ has an \textit{approximate limit} $z \in T_xM$ at $x$ if
	\begin{equation}
    \label{def_approximate_cont}
		\lim_{r \to 0} \fint_{B_r(x)} \abs{u(y) - \transp_x(y, z)} \NNN \vol \BBB(y)
			\defas \lim_{r \to 0} \frac{1}{\haus_g^n(B_r(x))} \int_{B_r(x)} \abs{u(y) - \transp_x(y, z)} \NNN \vol \BBB(y) = 0.
	\end{equation}

	The set $\mathcal S_u$ where this property does not hold is called the \textit{approximate discontinuity set} of $u$.
	For any $x \in M \setminus \mathcal S_u$ the approximate limit $z$ in (\ref{def_approximate_cont}) is uniquely determined and will be denoted by $\tilde u(x)$.
	Finally, we say that $u$ is \textit{approximately continuous} at $x$ if $x \in M \setminus \mathcal S_u$ and $u(x) = \tilde u(x)$.
\end{definition}

It is sometimes useful to resort to coordinates.
In this regard we wish to define the pull-back of a section of $TM$ through a local trivialization.
\begin{definition}\label{def:pull_back_section}
	Given a local trivialization $\Psi \colon \Omega \times \rr^{\NNN n\BBB} \to TO$ and a \CCC section \BBB $u$ of \CCC $TO$ \BBB we define $\Psi^* u \colon \Omega \to \rr^{\NNN n \BBB}$ at $x \in \Omega$ through
	\begin{equation*}
		(\Psi^* u(\Phi(x)))^\alpha \tau_\alpha(\Phi(x)) = u(\Phi(x))\NNN.\BBB
	\end{equation*}
\end{definition}

The following proposition investigates the relationship between approximate limit points and their coordinate representations in Euclidean space.
\begin{proposition}[Approximate limits and coordinates] \label{prop:approx_cont_coords}
  Let $\Psi \colon \Omega \times \rr^{\NNN n \BBB} \to TO$ be a local trivialization.
  Then, a section $u \in L^1(TO)$ has approximate limit $z$ at $x \in O$ if and only if its coordinate representation $\Psi^* u$ has approximate limit $\Psi^* z \in \rr^{\NNN n \BBB}$ at $\Phi^{-1}(x)$.
\end{proposition}

\begin{definition}[Approximate jump points] \label{def:approximate_jump}
	Let $u \in L^1_\loc(TM)$ and let $\transp_x$ be the transport map from $x \in M$ defined in (\ref{transport_map}).
	We say that $x$ is an \textit{approximate jump} point of $u$ if there exist $a, \, b \in T_xM$ with $a \neq b$ and a unit vector $\nu \in T_xM$ such that
	\begin{equation}
    \label{def_approximate_jump}
		\lim_{r \to 0} \fint_{B^+_r(x, \nu)} \abs{u(y) - \transp_x(y, a)} \NNN \vol \BBB(y) = 0, \quad \lim_{r \to 0} \fint_{B^-_r(x, \nu)} \abs{u(y) - \transp_x(y, b)} \NNN \vol \BBB(y) = 0,
	\end{equation}
	where $B_r^+(x, \nu)$ and $B_r^-(x, \nu)$ are the geodesic half balls defined by
	\begin{align*}
		B_r^+(x, \nu) &\defas \exp_x(\setof{X \in T_x M \colon \abs{X} < r, \, \sprod{X}{\nu} \CCC > \BBB 0}), \\
		B_r^-(x, \nu) &\defas \exp_x(\setof{X \in T_x M \colon \abs{X} < r, \, \sprod{X}{\nu} \CCC < \BBB 0}).
	\end{align*}

	The triplet $(a, b, \nu)$ is uniquely determined by \ref{def_approximate_jump} up to switching $a$ and $b$ as well as changing the sign of $\nu$.
	\CCC The triplet in the definition will be denoted by $(u^+(x), u^-(x), \nu(x))$ and the set of approximate jump points will be denoted by $\mathcal J_u$. \BBB
\end{definition}

The relation between approximate jump points and their coordinate representations in Euclidean space is as follows:
\begin{proposition}[Approximate jumps and coordinates] \label{prop:approx_jump_coords}
	Let $\Psi \colon \Omega \times \rr^{\NNN n\BBB} \to TO$ be a local trivialization.
	Then, a section $u \in L^1_\loc(TO)$ has an approximate jump at $x \in O$ with triplet $(a, b, \nu)$ if and only if $\Psi^* u$ has an approximate jump at $\Phi^{-1}(x)$ in the usual Euclidean sense with triplet $(\Psi^*a, \Psi^*b, \bar \nu)$, such that
  \begin{equation*}
    \nu^k = \frac{1}{\sqrt{g^{ij} \bar \nu^i \bar \nu^j}} g^{kl} \bar \nu^l \qquad \NNN\text{for } k \in \setof{1, \dots, n} \BBB
  \end{equation*}
  \CCC and \BBB $(g^{ij})$ \CCC denotes \BBB the inverse of the metric tensor $(g_{ij})$.
\end{proposition}

\begin{definition}[Approximate differentiability] \label{def:approximate_diff}
	Let $u \in L^1_\loc(TM)$ and let $\transp_x$ be the transport map from $x \in M$ defined in (\ref{transport_map}).
	We say that $x$ is an \textit{approximate differentiability} point of $u$ if $x \in M \setminus \mathcal S_u$ and if there exists $L \in T_xM \otimes T^*_x M$ such that
	\begin{equation}
    \label{def_approximate_diff}
		\lim_{r \to 0} \fint_{B_r(x)} r^{-1} \abs{u(y) - \transp_x(y, \tilde u(x)) - \transp_x(y, L(X))} \NNN \vol \BBB(y) = 0, \quad X \defas \exp^{-1}_x(y),
	\end{equation}
	where we identified $T_xM \otimes T_x^* M$ with the space of linear maps from $T_xM$ to $T_xM$.
	The tensor $L$ is uniquely determined by (\ref{def_approximate_diff}) and will be denoted by $\nabla u(x)$.
	The set of approximate differentiability points of $u$ will be written as $\mathcal D_u$.
\end{definition}

The next proposition clarifies the relationship between approximate differentiability points and their coordinate representations.
\begin{proposition}\label{prop:approx_diff_coords}
	Let $\Psi \colon \Omega \times \rr^{\NNN n\BBB} \to TO$ be a local trivialization with induced frame \NNN $\setof{\tau_1, \ldots, \tau_n}$\BBB.
	Then, any section $u \in \NNN L^1\BBB(TO)$ is approximately differentiable at $x \in O$ with approximate gradient $L \NNN \in T_xM \otimes T_x^* \! M\BBB$ if and only if $\Psi^*u$ is approximately differentiable at $\Phi^{-1}(x)$ in the usual Euclidean sense with approximate gradient $\bar L$ and approximate limit $\bar z \NNN \in \rr^n \BBB$ \CCC such that \BBB
	\begin{equation}\label{approx_grad_coords}
		L = (\bar L_i^\alpha + \Gamma_{i \beta}^\alpha \bar z^\beta) \, \tau_\alpha \otimes \di x^i,
	\end{equation}
	where $(\Gamma_{i \beta}^\alpha)$ \NNN denotes \BBB the Christoffel symbols at $x$.
\end{proposition}

\CCC In the next definition we recall the notion of \NNN \textit{rectifiability} \BBB on a Riemannian manifold $M$.\BBB
\begin{definition}[$\haus_g^{n-1}$-rectifiable]
	A set $N \subset M$ is $\haus_g^{n-1}$-rectifiable if and only if there exists a countable family $\setof{N_h}_h$ of $C^1$-regular $(n-1)$-dimensional submanifolds of $M$ such that
	\begin{equation*}
		\haus_g^{n-1}(M \setminus \cup_h N_h) = 0.
	\end{equation*}
\end{definition}

We are ready to state \CCC a fundamental theorem for \NNN tangent vector fields \BBB of bounded variation. \BBB
\begin{theorem}[Decomposition of \NNN tangent vector fields \BBB of bounded variation]\label{thm:decomp_secs}
	Let $u \in BV(TM)$, then the discontinuity set $\mathcal S_u$ is $\haus_g^{n-1}$-rectifiable, $\haus_g^{n-1}(\mathcal S_u \setminus \mathcal J_u) = 0$, and
	the restriction $D^j u \defas D^s u \mrestr {\mathcal J_u}$ of the singular part of $Du$ to $\mathcal J_u$ can \CCC be \BBB represented as
	\begin{equation*}
		D^j u = (u^+ - u^-) \otimes \nu^\flat \haus_g^{n-1} \mrestr {\mathcal J_u},
	\end{equation*}
	where the triplet $(u^+, \, u^-, \, \nu)$ is as in Definition \ref{def:approximate_jump} and $\nu^\flat$ is the 1-form given by $\nu^\flat(X) = \sprod{\nu}{X}$ for any $X \in T_x M$.

	Furthermore, $u$ is approximately differentiable at a.e.~point of $M$ and the absolutely continuous part of $Du$ can be written as
	\begin{equation*}
		D^a u = \nabla u \, \haus_g^n,
	\end{equation*}
	$\nabla u$ being the approximate gradient of $u$.	
\end{theorem}
\begin{remark}
	To summarize the above theorem, we end up with the following decomposition of $Du$:
	\begin{equation*}
		Du = (u^+ - u^-) \otimes \nu^\flat \haus_g^{n-1} \mrestr {\mathcal J_u} + \nabla u \haus_g^n + D^c u,
	\end{equation*}
	where $D^c u \defas D^s u \mrestr {(M \setminus \mathcal S_u)}$ is the so called \textit{Cantor part} of $u$.
\end{remark}

\NNN At this point we wish to shortly comment on the decomposition theorem the scalar case.
A scalar function $f \in L^1(M)$ has bounded variation if and only if
\begin{equation*}
	\var(f) \defas \sup\setof*{\int_M f \di^* v \vol \colon v \in C^\infty_c(T^*\!M), \norm{v}_{L^\infty} \leq 1} < \infty,
\end{equation*}
where $\di^*$ is the adjoint exterior derivative.
Similar definitions (to the vector-valued case) hold true for the blow-up quantities.
Instead of $\nabla f$ we will usually write $\di f$ for the approximate gradient fo $f$.
Notice that given an approximate differentiability point $x$ of $f$ and a local chart $\Phi$ in the vicinity of $x$ we have
\begin{equation*}
	\di f(x) = \frac{\partial}{\partial x^i} (f \circ \Phi)(\Phi^{-1}(x)) \di x^i.
\end{equation*}
Finally, the following decomposition holds true in the scalar setting for the distributional derivative $D f$ of $f$:
\begin{equation*}
	Df = (f^+ - f^-) \nu^\flat \haus^{n-1}_g \mrestr \mathcal{J}_f
				+ \di f \haus^{n-1}_g + D^c f.
\end{equation*}
\BBB

The definition of \textit{special} \NNN tangent vector fields \BBB of bounded variation then naturally follows:
\begin{definition}[Special sections of bounded variation]\label{def:spec_bvsecs}
	The set of \textit{special sections of bounded variation} consists of $u \in BV(TM)$ with vanishing Cantor part.
	More precisely we set
	\begin{equation*}
		SBV(TM) \defas \setof{u \in BV(TM) \colon D^c u = 0},
	\end{equation*}
	\NNN For \BBB any $p \in (1, \infty)$ we also define the space
	\begin{equation*}
		SBV^p(TM) \defas \setof{u \in SBV(TM) \colon \nabla u \in L^p(TM \otimes T^*\!M), \, \haus_g^{n-1}(\mathcal J_u) < \infty}.
	\end{equation*}
	\NNN Let \BBB $u \in SBV^p(TM)$, then a sequence $(u_h) \subset SBV^p(TM)$ is said to converges weakly towards $u$ in $SBV^p(TM)$ (shortly written as $u_h \weakto u$) if and only if
	\begin{align*}
			&(i) & \nabla u_h &\weakto \nabla u \text{ weakly in } L^p( \CCC TM \BBB \otimes T^*\!M), \\
			&(ii) & D^j u_h &\weakstarto D^j u \text{ weakly* in } \radon(\CCC TM \BBB \otimes T^*\!M).
	\end{align*}
\end{definition}

In order to show that a section $u$ is in $SBV_\loc(TM)$ we can equivalently resort to coordinates.
More precisely:
\begin{lemma}\label{lem:sbv_in_coords}
	A section $u \NNN \in L^1_\loc(TM)\BBB$ is in $BV_\loc(TM)$ ($SBV_\loc(TM)$, $L^p_\loc(TM) \cap SBV_\loc^p(TM)$) if and only if for any local trivialization $\Psi \colon \Omega \times \rr^{\NNN n \BBB} \to TO$ the pull-back $\Psi^* u$ is in $BV_\loc(\Omega; \rr^n)$ ($SBV_\loc(\Omega; \rr^n)$, $L^p_\loc(\Omega; \rr^n) \cap SBV^p_\loc(\Omega; \rr^n))$ in the usual Euclidean sense.
\end{lemma}

The last result of this section concerns \NNN a \BBB compactness \NNN result \BBB in the spaces $SBV^p(\NNN TM \BBB)$.
\begin{theorem}[Compactness in $SBV^p$]
	\label{thm:compactness_spec_bvsecs}
	Let $M$ be a \textit{compact} manifold (with or without boundary), $p \in (1, \infty)$, and $(u_h) \subset SBV^p(TM)$ be a sequence satisfying the following bound:
	\begin{equation*}
		\sup_h \left(\norm{u_h}_{L^\infty} + \norm{\nabla u_h}_{L^p} + \haus_g^{n-1}(\mathcal{J}_{u_h})\right) < \infty.
	\end{equation*}
	Then, up to taking a subsequence, $u_h \weakto u$ weakly in $SBV^p(TM)$ and $u_h \to u$ in $L^1(TM)$.
\end{theorem}

\subsection{Vorticity}\label{sec:vorticity}
From this point on, we restrict ourselves to the case of a closed, oriented 2-dimensional Riemannian manifold $S$ with metric tensor and volume form \NNN still denote by $g$ and $\vol$, respectively\BBB.

For the moment, let $u \in C^\infty(TS)$ be a smooth tangent vector-field.
The \textit{pre-jacobian} of $u$ is the 1-form $\jac(u) \in C^\infty(T^*S)$ defined by
\begin{equation*}
	\jac(u)(X) \defas \sprod{\nabla_X u}{i u}, \quad \text{ for all } X \in TS.
\end{equation*}
Here, $i \colon TS \to TS$ is the isometry of $TS$ onto itself characterized by
\begin{equation*}
	i^2 v = -v, \quad \sprod{i v}{w} = - \sprod{v}{iw} = \vol(v, w) \qquad \NNN \text{for all } \BBB v, w \in TS,
\end{equation*}
where $\sprod{\cdot}{\cdot}$ is the scalar product on $TS$ induced by the metric tensor $g$.

Given an open subset $O \subset S$ with Lipschitz boundary such that $\abs{u} \geq c$ on $\partial O$ for some $c > 0$ we can define the \textit{degree} of $u$ \CCC on \BBB $\partial O$ as
\begin{equation}\label{def_degree_with_jac}
	\deg(u, \partial O) \defas \frac{1}{2\pi} \left( \int_{\partial O} \frac{\jac(u)}{\abs{u}^2} + \int_{O} \kappa \vol \right),
\end{equation}
where $\kappa$ is the Gauss curvature.
It can be shown that the degree is valued in $\zz$.

If $u$ is of unit length on $\partial O$\CCC, by \BBB Stokes' theorem \CCC it holds that \BBB
\begin{equation*}
  \deg(u, \partial O) = \int_{O} \vort(u), \qquad \NNN \vort(u) \BBB \defas  \di \jac(u) + \kappa \vol.
\end{equation*}
The 2-form $\vort(u)$ is called the \textit{vorticity} of $u$.

The pre-jacobian and the vorticity can be extended to a more general setting.
\CCC Given $S \subset M$, i\BBB n this paper we are mainly interested in the case of $u \in SBV(TS)$ such that $\abs{\nabla u} \in L^p(S)$ and $\abs{u} \in L^q(S)$ for $p, \, q \in [1, \infty]$ satisfying $\frac{1}{p} + \frac{1}{q} = 1$.
Note that $\nabla u$ is the approximate gradient of $u$ (see also Definition \ref{def:approximate_diff}).
By an application of Hölder's inequality we see that $\jac(u) \in L^1(T^*S)$.
It is then possible to define the vorticity of $u$ in distributional sense.
In fact, given $\alpha \in C^\infty(T^*S)$ and $\beta \in C^\infty(\Lambda^2 S)$, where $\Lambda^2(S) \defas T^*S \wedge T^*S$, the adjoint exterior derivative $\NNN \di^*\BBB $ satisfies:
\begin{equation*}
  \int_S \sprod{\di \alpha}{\beta} \vol = \int_S \sprod{\alpha}{\NNN \di^*\BBB \beta} \vol.
\end{equation*}
Further\CCC more\BBB, for $\Phi = \vphi \vol$ with $\vphi \in C^\infty(S)$, we have that
\begin{equation*}
	\NNN \di^*\BBB  \Phi = \hodge \NNN \di \BBB (\hodge \Phi) = \hodge (\NNN\di\BBB\vphi),
\end{equation*}
where $\star$ is the Hodge star.
Consequently, \NNN as $\alpha \wedge (\hodge \beta) = \sprod{\alpha}{\beta} \vol$ for $\alpha$, $\beta$ as before, \BBB for smooth $u$ the following integration-by-parts formula \NNN holds true: \BBB
\begin{equation*}
	\int_S \sprod{\di \jac(u)}{\Phi} \vol = \int_S \sprod{\jac(u)}{\hodge (\NNN\di\BBB\vphi)} \vol = \int_S \jac(u) \wedge (\hodge\!\hodge \NNN\di\BBB \vphi) = - \int_S \jac(u) \wedge \NNN\di\BBB \vphi,
\end{equation*}
where $\wedge$ is the \CCC w\BBB edge product.
This allows us to define $\di \jac(u)$ in distributional sense through its action on smooth 2-forms $\Phi = \vphi \vol \in C^\infty(\Lambda^2 S)$ as follows\NNN : \BBB
\begin{equation*}
	\di \jac(u)(\Phi) \defas -\int_S \jac(u) \wedge \di \vphi\NNN.\BBB
\end{equation*}
For \NNN any \BBB such $\Phi$\CCC, \BBB we can then define the vorticity of $u \NNN \in SBV(TS)$ in distributional sense via
\begin{equation}\label{ibp_vorticity}
  \vort(u)(\Phi) = -\int_S \NNN\Big(\BBB \jac(u) \wedge \di \vphi + \kappa \vphi \vol \NNN\Big)\BBB.
\end{equation}

\CCC With this definition the validity \BBB of \textit{Morse's index formula} in our present function setting\CCC, easily follows\NNN.\BBB
\begin{theorem}[Morse's index formula]
  \label{thm:morses_index_formula}
  For any $u \in SBV(TS) \cap L^\infty(TS)$ it holds that
  \begin{equation*}
    \int_S \vort(u) = \echar,
  \end{equation*}
  where $\echar$ is the \textit{Euler characteristic} of $S$.
\end{theorem}

\begin{proof}
  \CCC Testing \BBB (\ref{ibp_vorticity}) with $\vphi \equiv 1$ and using the Gauss-Bonnet theorem \CCC we conclude that \BBB
  \begin{equation*}
    \int_S \vort(u) = \NNN \vort(u)(\vol) \BBB = \int_S \jac(u) \wedge \di 1 + \int_S \kappa \vol = \int_S \kappa \vol = \echar.
  \end{equation*}
\end{proof}

In the absence of jumps the distributional jacobian satisfies the following two useful properties:
\begin{lemma}\label{lem:marco_trick}
	Let $O \subset S$ be an open subset and $u \in W^{1, 2}(TO)$.
	Then, the distributional jacobian $\di \jac(u)$ is in $L^1(\Lambda^2O)$ and for a.e.~point in $O$ we have that
	\begin{equation}
		\label{dist_jac_pw_bound}
		\abs{\di \jac(u)} \leq \abs{\nabla u}^2.
	\end{equation}

	Furthermore, given another vector field $v \in W^{1, 2}(TO)$ \NNN it holds that \BBB
	\begin{equation}
		\label{vort_diff_bound}
			\norm{\vort(u) - \vort(v)}_{\EEE W^{-1, \infty}_0(\Lambda^2 O)\BBB} \leq \norm{u - v}_{L^2} (\norm{\nabla u}_{L^2} + \norm{\nabla v}_{L^2}).
	\end{equation}
\end{lemma}
\begin{proof}
A proof of (\ref{dist_jac_pw_bound}\NNN) \BBB can be found in \cite{jerrard_glman} (see Lemma 5.3), while a proof of (\ref{vort_diff_bound}) is contained in \cite{canevari_xyman} (see (3.27) in Lemma 3.7).
\end{proof}

In the remainder of this section we will investigate how a vector-field $u$ and associated quantities such as its pre-jacobian $\jac(u)$ change under local ``doubling'' of the angles.
Let us first precisely define what we mean by ``doubling'' or, more generally, multiplying all angles by $m \in \setof{2, 3, \dots}$\CCC, which will be shortly written as ``$m$-pling''. \BBB
For this purpose, let $O$ be a coordinate neighborhood of $S$.
\NNN This guarantees \BBB that we can find a smooth \NNN (up to the boundary) \BBB unit-length vector field $\tau \in C^\infty(TO)$.
For each $x \in O$ we will see $\tau(x)$ as the unique unit-length vector in $T_xS$ having zero angle and represent a tangent vector $X$ through its \textit{polar coordinates} $r = r(X), \, \alpha = \alpha(X)$, which are characterized by
\begin{equation*}
  X = r \cos(\alpha) \, \tau + r \sin(\alpha) \, i \tau.
\end{equation*}
The map $\power \colon TO \to TO$ multiplying the angles by $m$ (shortly written as $m$-pling the anlges) is then given by
\begin{equation}\label{mth_power}
  \power(X) \defas r \cos(m\alpha) \, \tau + r \sin(m\alpha) \, i \tau.
\end{equation}

\CCC The \BBB next proposition shows how the derivative of a tangent vector field in $SBV(TO)$ changes after \CCC $m$-pling \BBB the angles.
From this point on, we will denote the approximate gradient of a scalar map $f \in SBV(O)$ by $\di f$ instead of $\nabla f$.
\begin{proposition}[Derivative and related quantities after \CCC $m$-pling\BBB]\label{prop:after_doubling}
  Given a simply connected open set $O \subset S$, let $u \in SBV(TO)$, and let $v \defas \power(u)$ with $\power$ as in (\ref{mth_power}) for some smooth unit-length vector field $\tau \in C^\infty(TO)$.
  Then, $\abs{u} \in SBV(O)$ and $v \in SBV(TO)$.
  The approximate gradient and jump part of $v$ \CCC are\BBB
	\begin{align}
		\nabla v 
		&= \abs{u}^{-1} \NNN v \otimes \di \abs{u} \BBB  + \NNN iv \otimes (m \abs{u}^{-2} \jac(u) - (m - 1)\jac(\CCC\tau\NNN))\BBB, \label{grad_after_doubling} \\
		D^j v
		&= (\power(u^+) - \power(u^-)) \otimes \nu_u^\flat \, \haus^1_g \mrestr {\mathcal J_u}, \label{jump_after_doubling}
	\end{align}
	where the right-hand side of (\ref{grad_after_doubling}) is implicitly set to be $0$ in $\setof{u = 0}$.
  Furthermore, the squared approximate gradients and pre-jacobians transform in the following way:
  \begin{align}
      \abs{\nabla v}^2
      &= m^2 \abs{\nabla u}^2
        + (1 - m^2) \abs{\NNN\di\BBB\abs{u}}^2
        + (m - 1)^2 \abs{u}^2 \abs{\jac(\CCC\tau\BBB)}^2 
        \NNN - \BBB 2m (m - 1) \sprod{\jac(u)}{\NNN\jac(\tau)\BBB},  \label{absgrad_after_doubling} \\
		\jac(v) 
			&= m \jac(u) - (m - 1) \abs{u}^2 \jac(\CCC\tau\BBB) \label{prejac_after_doubling}\NNN.\BBB
	\end{align}
  Additionally assuming that $u \in L^\infty(TO)$ we have the following relation between the vorticities of $u$ and $v$:
  \begin{equation}
    	\vort(v) 
			= m \vort(u) - (m - 1) \vort(\abs{u} \tau) \label{vorticity_after_doubling}\NNN,\BBB
  \end{equation}
  with $\omega(\cdot)$ defined distributionally as in (\ref{ibp_vorticity}).
\end{proposition}

Before \CCC proving \BBB Proposition \ref{prop:after_doubling} we will derive several helpful \CCC lemmas\BBB.
\begin{lemma}\label{lem:dabsu_coords}
  Let $u \in SBV(TO)$ for some open subset $O \subset \CCC S\BBB$; then, $\abs{u} \in SBV(TO)$.
  Furthermore, given coordinates $\setof{x^1, x^2}$ and \CCC a\NNN n \BBB orthonormal frame $\setof{\CCC\tau\BBB, i \CCC\tau\BBB}$, we have for any $k \in \setof{1, 2}$
  \begin{equation}
		\label{dabsu_coords}
    \NNN\di\BBB\abs{u}(\del{k})
    = \frac{1}{\abs{u}} \CCC\left(\BBB u^1 \frac{\partial u^1}{\partial x^k} + u^2 \frac{\partial \NNN u^2\BBB}{\partial x^k}\CCC\right) \NNN \qquad \text{a.e.~in } O, \BBB
  \end{equation}
  where $u^1 = \sprod{u}{\CCC\tau\BBB}$, $u^2 = \sprod{u}{i\CCC\tau\BBB}$, and the expression on the right\CCC-hand \BBB side of the equality sign is implicitly defined to be $0$ in $\setof{u = 0}$.
\end{lemma}

\begin{proof}
  Without loss of generality we can assume that the coordinates and \CCC the \BBB frame can be globally defined in $O$.
  Let $w$ denote the coordinate representation of $\abs{u}$.
  As we chose an orthonormal frame we have
  \begin{equation*}
    w = f(u^1, u^2) \defas \sqrt{(u^1)^2 + (u^2)^2}.
  \end{equation*}
  Then by Theorem 3.92 (a) from \cite{fusco_bible} the approximate gradient satisfies $\NNN \di \BBB u = 0$ a.e.~in $\setof{w = 0}$.
  At a.e.~point in the remaining set $\setof{w \neq 0}$ we derive by the \NNN Euclidean \BBB chain rule in $BV$ (see also Theorem 3.96 in \cite{fusco_bible}) that
  \begin{equation*}
    \frac{\partial \NNN w \BBB}{\partial x^k}
    = \frac{\partial \NNN f \BBB}{\partial u^1} (u) \frac{\partial \NNN u^1\BBB}{\partial x^k} + \frac{\partial \NNN f \BBB}{\partial u^2}(u) \frac{\partial \NNN u^2 \BBB}{\partial x^k}
    = \frac{1}{\sqrt{(u^1)^2 + (u^2)^2}} \left(u^1 \frac{\partial \NNN u^1 \BBB}{\partial x^k} + u^2 \frac{\partial \NNN u^2\BBB}{\partial x^k}\right).
  \end{equation*}
  With \eqref{approx_grad_coords} this directly leads to (\ref{dabsu_coords})\NNN.\BBB
\end{proof}

\begin{lemma}\label{lem:lipschitz_doubling}
  Given $m \in \setof{2, 3, \dots}$, the map $p \colon \cc \to \cc$ defined by $p(z) \defas \frac{z^m}{\abs{z}^{m-1}}$ for $z \neq 0$ and $p(0) \defas 0$ is Lipschitz continuous with Lipschitz constant bounded by $2m - 1$.
\end{lemma}

\begin{proof}
  Let $z, \, \NNN w \BBB \in \cc$.
  Without loss of generality we can assume that $\abs{z} \geq \abs{\NNN w \BBB} > 0$.
  We then compute:
  \begin{align*}
    \abs{p(z) - p(\NNN w \BBB)}
    &= \abs*{\frac{z^m}{\abs{z}^{m - 1}} - \frac{\NNN w \BBB^m}{\abs{z}^{m - 1}} + \frac{\NNN w \BBB^m}{\abs{z}^{m - 1}} - \frac{\NNN w \BBB^m}{\abs{\NNN w \BBB}^{m - 1}}} \\
    &\leq \frac{1}{\abs{z}^{m - 1}} \abs{z^m - \NNN w \BBB^m} + \frac{\abs{\NNN w \BBB}^m}{\abs{z}^{m - 1} \abs{\NNN w \BBB}^{m - 1}} \abs{\abs{z}^{m - 1} - \abs{\NNN w \BBB}^{m - 1}} \\
    &= \frac{\abs{z - \NNN w \BBB}}{\abs{z}^{m - 1}} \abs*{\sum_{k = 0}^{m - 1} z^k \NNN w \BBB^{m - 1 - k}}
      + \frac{1}{\abs{z}^{m - 2}}\abs{\abs{z} - \abs{\NNN w \BBB}} \abs*{\sum_{k = 0}^{m - 2} \abs{z}^k\abs{\NNN w \BBB}^{m - 2 - k}} \\
    &\leq m \abs{z - \NNN w \BBB} + (m - 1) \abs{z - \NNN w \BBB} = (2 m - 1) \abs{z - \NNN w \BBB},
  \end{align*}
  as desired.
\end{proof}

\begin{lemma}[Product rule in $BV$]\label{lem:product_rule}
  Let $f \in BV(O)$ for some open subset $O \subset S$ and $\NNN u \BBB \in C^\infty(T\NNN \bar O\BBB)$; then, $f v \in BV(TO)$ with its approximate gradient satisfying
  \begin{equation}\label{product_rule}
    \nabla (f \NNN u \BBB) = \NNN u \BBB \otimes \di f + f \nabla \NNN u \BBB.
  \end{equation}
\end{lemma}

\begin{proof}
  Without loss of generality we can assume that $O$ is a coordinate neighborhood.
  By linearity, the distributional derivative $D(f \NNN u \BBB)$ is uniquely determined by the values $\int_O \sprod{f \NNN u \BBB}{\nabla^* \NNN v \BBB} \NNN \vol \BBB$ for $\NNN v \BBB \in C^\infty_c(TO \otimes T^*O)$ of the special form $\NNN v \BBB = X \otimes \alpha$ with $X \in C^\infty_c(TO)$ and $\alpha \in C^\infty_c(T^*O)$.
  \NNN Assume for the moment that \BBB $f \in C^\infty_c(O)$\NNN. Then, \BBB the product rule in (\ref{product_rule}) holds pointwise.
  \NNN Consequently, taking the scalar product of both sides of (\ref{product_rule}) with \BBB $\NNN v \BBB$ and integrating by parts leads to
  \begin{equation*}
    \int_O \sprod{f\NNN u \BBB}{\nabla^* \NNN v \BBB} \vol
    = \int_O \sprod{\NNN u \BBB}{X}\sprod{\di f}{\alpha} \vol + \int_O \sprod{f \nabla \NNN u \BBB}{\NNN v \BBB} \vol
    = \int_O \sprod{f}{\NNN \di^*\BBB (\sprod{\NNN u \BBB}{X} \alpha)} + \int_O \sprod{f \nabla \NNN u \BBB}{\NNN v \BBB} \vol.
  \end{equation*}
  Note that we used that
  \begin{equation*}
    \sprod{\NNN u \BBB \otimes \di f}{X \otimes \alpha} \defas \sprod{\NNN u \BBB}{X} \sprod{\di f}{\alpha}.
  \end{equation*}
  The second equality above can be extended to any $f \in BV(O)$ by approximation in $L^1(O)$.
  By the very definition of the total variation in (\ref{total_variation}) the formula above proves that $fv$ belongs to $BV(O)$.
  Furthermore, by Riesz representation (see Theorem \ref{thm:riesz}) it follows that
  \begin{equation*}
    \int_O \sprod{\sigma_{f\NNN u \BBB}}{\NNN v \BBB} \NNN\di\BBB\abs{D(f\NNN u \BBB)}
    = \int_O \sprod{\NNN u \BBB \otimes \sigma_f}{\NNN v \BBB} \NNN \di \BBB \abs{Df} + \int_O \sprod{f \nabla \NNN u \BBB}{\NNN v \BBB} \NNN \vol \BBB.
  \end{equation*}
  Using the decomposition \NNN of \BBB $\NNN Df \BBB$ (see Theorem \ref{thm:decomp_secs}) and the uniqueness of the respective decomposition of $\NNN D(f u)\BBB$, we derive that the absolutely continuous part of $D(f\NNN u \BBB)$ with respect to $\haus^2_g$ is given by
  \begin{equation*}
    (\NNN u \BBB \otimes \di f + f \nabla \NNN u \BBB) \haus^2_g,
  \end{equation*}
  as desired.
\end{proof}

\begin{proof}[Proof of Proposition \ref{prop:after_doubling}]
  \NNN We can \BBB assume that \NNN $O$ is a coordinate neighborhood with \BBB coordinates $\setof{x^1, \, x^2}$ and an orthonormal frame $\setof{\CCC\tau\BBB, i\CCC\tau\BBB}$.
  Let us denote by $\Psi \colon \Omega \times \rr^2 \to TO$ the induced local trivialization and let $\power$ be the power map defined in (\ref{mth_power}).
  To shorten notation, we will write $\bar u \defas \Psi^* u$ and $\bar v \defas \Psi^* v$.

  \textit{\NNN Step 1 \BBB ($v \in SBV(TO)$ and chain rule for the jump part):}
  We can represent $\power$ in coordinates as
  \begin{equation*}
    \powerc(u^1, u^2) = \begin{pmatrix} r \cos(m \alpha) \\ r \sin(m \alpha)\end{pmatrix}, \qquad \text{where} \quad r = \sqrt{(u^1)^2 + (u^2)^2}, \quad \alpha = \arg(u).
  \end{equation*}
  \NNN Here, $u^1$ and $u^2$ are the components of $u$ with respect to the frame $\setof{\tau, i\tau}$.
  Furthermore, $\arg(u)$ is the argument of $u^1 + i u^2$. \BBB
  By Lemma \ref{lem:lipschitz_doubling}, the map $\powerc$ is Lipschitz continuous.
  Since $u \in SBV(TO)$, the Euclidean chain rule in $BV$ and Lemma \ref{lem:sbv_in_coords} imply that $v \in SBV(TO)$.
  
  We now wish to prove (\ref{jump_after_doubling}).
  \CCC F\BBB ix $x \in \mathcal J_u$.
  By Proposition \ref{prop:approx_jump_coords}, $\Phi^{-1}(x) \in \mathcal J_{\bar u}$.
  Consequently, by the Euclidean chain rule and Lemma \ref{lem:lipschitz_doubling}, $\bar v$ has approximate upper and lower limits at $\Phi^{-1}(x)$ given by $\bar v^\pm = \powerc(\bar u^\pm)$ and approximate normal $\bar \nu_{\bar v} = \bar \nu_{\bar u}$.
  Using Proposition \ref{prop:approx_jump_coords} again \CCC it follows \BBB that $v$ has approximate limits $v^\pm = \power(u^\pm)$ and approximate normal $\nu_v = \nu_u$.
  By the arbitrariness of $x$ we see that
  \begin{equation*}
    D^j v \mrestr \mathcal J_u = (\power(u^+) - \power(u^-)) \otimes \nu_u^\flat \haus^1_g \mrestr {\mathcal J_u}.
  \end{equation*}
  Using again the chain rule and the relations between approximate quantities on the manifold and in the \NNN Euclidean \BBB setting given by Proposition \ref{prop:approx_cont_coords} and Proposition \ref{prop:approx_jump_coords}, from $\mathcal J_{\bar v} \subset \mathcal S_{\bar u}$ and $\haus^1(\mathcal S_{\bar u} \setminus \mathcal J_{\bar u}) = 0$ it follows that $\haus^1_g(\mathcal J_v \setminus \mathcal J_u) = 0$, which together with the equality above shows (\ref{jump_after_doubling}).

  \textit{\NNN Step 2 \BBB (\NNN C\BBB hain rule for the approximate gradient):}
  As $\nabla \bar u = 0$ a.e~in $\setof{\bar u = 0}$ (see also Proposition 3.92 a) in \cite{fusco_bible}) \CCC and \BBB $\bar u$ has approximate limit $0$ at a.e.~point in $\setof{\bar u = 0}$, by (\ref{approx_grad_coords}), it follows that $\nabla u = 0$ a.e.~in $\setof{u = 0}$.
  Therefore, (\ref{grad_after_doubling}) is satisfied at a.e.~point in $\setof{u = 0}$.

  \NNN By \BBB $\haus^n_g(\mathcal S_u) = 0$ (see Theorem \ref{thm:decomp_secs}), and therefore $\haus^n(\mathcal S_u \setminus \mathcal S_v)$, it remains to investigate points $x \in O \setminus \setof{u = 0}$ at which $v$ and $u$ are approximately differentiable.
  Using
  \CCC
  \begin{align*}
    \NNN \frac{\partial}{\partial u^1} \BBB (r \cos(m \alpha))
    &= \frac{u^1}{r} \cos(m \alpha) + m \frac{u^2}{r} \sin(m \alpha), &
    \NNN \frac{\partial}{\partial u^1} \BBB (r \sin(m \alpha))
    &= \frac{u^1}{r} \sin(m \alpha) - m \frac{u^2}{r} \cos(m \alpha), \\
    \NNN \frac{\partial}{\partial u^2} \BBB (r \cos(m \alpha))
    &= \NNN \frac{u^2}{r} \cos(m \alpha) - m \frac{u^1}{r} \sin(m \alpha) \BBB, &
    \NNN \frac{\partial}{\partial u^2} \BBB (r \sin(m \alpha))
    &= \NNN \frac{u^2}{r} \sin(m \alpha) + m \frac{u^1}{r} \cos(m \alpha) \BBB,
  \end{align*}
  \BBB
  we see that
  \begin{equation}\label{der_powerc}
    \begin{aligned}
    \NNN \frac{\partial}{\partial u^1} \BBB \powerc
    &= \frac{u^1}{(u^1)^2 + (u^2)^2} \powerc - m \frac{u^2}{(u^1)^2 + (u^2)^2} i \powerc, \\
    \NNN \frac{\partial}{\partial u^2} \BBB \powerc
    &= \frac{u^2}{(u^1)^2 + (u^2)^2} \powerc + m \frac{u^1}{(u^1)^2 + (u^2)^2} i \powerc,
  \end{aligned}
  \end{equation}
  where for every $p \in \rr^2$ we have written $i p$ for its anticlockwise rotation by $\frac{\pi}{2}$.
  By the Euclidean chain rule it follows that
  \begin{equation*}
    \frac{\partial \NNN \bar v \BBB}{\partial x^k}
    = \frac{\partial}{\partial u^1} \powerc(\bar u) \frac{\partial \NNN u^1 \BBB}{\partial x^k} + \frac{\partial}{\partial u^2} \powerc(\bar u) \frac{\partial \NNN u^2 \BBB}{\partial x^k}\NNN.\BBB
  \end{equation*}
  With the help of (\ref{approx_grad_coords}) and (\ref{der_powerc}) this implies
  \begin{align}
    \nabla_{\del{k}} v
    &= \frac{1}{(u^1)^2 + (u^2)^2} \left((u^1 v - m u^2 iv) \frac{\partial \NNN u^1 \BBB}{\partial x^k} + (u^2 v + m u^1 iv) \frac{\partial \NNN u^2 \BBB}{\partial x^k}\right)
      + v^1 \nabla_{\del{k}} \CCC \tau \BBB + v^2 \nabla_{\del{k}} \NNN ( \BBB i\CCC \tau \BBB \NNN ) \BBB \nonumber\\
    \begin{split}
      &= \frac{1}{(u^1)^2 + (u^2)^2} \NNN\left(\BBB u^1 \frac{\partial \NNN u^1 \BBB}{\partial x^k} + u^2 \frac{\partial \NNN u^2 \BBB}{\partial x^k}\NNN\right)\BBB v
      + m \frac{1}{(u^1)^2 + (u^2)^2} \NNN\left(\BBB u^1 \frac{\partial \NNN u^2\BBB}{\partial x^k} - u^2 \frac{\partial \NNN u^1 \BBB}{\partial x^k}\NNN\right)\BBB iv \\
      &\phantom{=} \quad + v^1 \nabla_{\del{k}} \CCC \tau \BBB + v^2 \nabla_{\del{k}} \NNN ( \BBB i \CCC \tau \BBB \NNN ) \BBB,
    \end{split} \label{nabla_xk_v_nonintrinsic}
  \end{align}
  where we have \NNN used \BBB
  \begin{equation}\label{gamma_through_A}
    v^1 (\Gamma_{k 1}^1 \CCC \tau \BBB + \Gamma_{k 1}^2 i\CCC \tau \BBB) + v^2 (\Gamma_{k 2}^1 \CCC \tau \BBB + \Gamma_{k 2}^2 i \CCC \tau \BBB) = v^1 \nabla_{\del{k}} \CCC \tau \BBB + v^2 \nabla_{\del{k}} \NNN ( \BBB i \CCC \tau \BBB \NNN ) \BBB.
  \end{equation}
  In order to find an intrinsic expression of (\ref{nabla_xk_v_nonintrinsic}), we need to make use of several additional formulas that we derive below.
  By differentiating the identity $\sprod{\tau}{\tau} = 1$ as well as $\sprod{\NNN i\tau \BBB}{\NNN \tau \BBB} = 0$ we have \CCC $\sprod{\nabla_{\del{k}} \tau}{\tau} = \sprod{\nabla_{\del{k}} i\tau}{i\tau} = 0$ and $\sprod{\nabla_{\del{k}} \NNN ( i \CCC \tau \NNN )\CCC}{\NNN \tau} = - \sprod{\NNN \nabla_{\del{k}} \tau \CCC}{\NNN i \tau \CCC}$. \BBB
  Hence,
  \begin{equation}
    \begin{aligned}
      \nabla_{\del{k}} \CCC \tau \BBB &= \sprod{\nabla_{\del{k}} \CCC \tau \BBB}{i \CCC \tau \BBB} i\CCC \tau\BBB = \jac(\CCC\tau\BBB)\NNN\left(\BBB\del{k}\NNN\right)\BBB i\CCC \tau\BBB,\\
      \nabla_{\del{k}} i\CCC \tau\BBB &= -\sprod{\nabla_{\del{k}} \CCC \tau\BBB}{i \CCC \tau\BBB} \CCC \tau \BBB = -\jac(\CCC\tau\BBB)\NNN\left(\BBB\del{k}\NNN\right)\BBB \CCC \tau \BBB.
    \end{aligned}\label{nabla_e_ie}
  \end{equation}
  \NNN Hence\BBB, we can express the pre-jacobian \NNN in \BBB coordinates \NNN as \BBB
  \begin{align*}
    \jac(u)\NNN\left(\BBB\del{k}\NNN\right)\BBB 
    = \sprod{\nabla_{\del{k}} u}{iu}
    &= \NNN \sprod*{\BBB\frac{\partial \NNN u^1\BBB}{\partial x^k} \CCC \tau \BBB + \frac{\partial \NNN u^2\BBB}{\partial x^k} i \CCC \tau \BBB + u^1 \nabla_{\del{k}} \CCC \tau \BBB + u^2 \nabla_{\del{k}} i\CCC \tau}{-u^2 \CCC \tau \BBB + u^1 i \CCC \tau \NNN}\BBB \\
    &= u^1 \frac{\partial \NNN u^2\BBB}{\partial x^k} - u^2 \frac{\partial \NNN u^1\BBB}{\partial x^k} + \abs{u}^2 \jac(\CCC\tau\BBB)\NNN\left(\BBB\del{k}\NNN\right)\BBB.
  \end{align*}
  Substituting (\ref{nabla_xk_v_nonintrinsic}) \NNN into \BBB (\ref{dabsu_coords}), by the above equality \NNN and \BBB (\ref{nabla_e_ie}) \NNN we have \BBB
  \begin{align*}
    \nabla v
    &= \abs{u}^{-1} v \otimes \NNN\di\BBB\abs{u} + m \abs{u}^{-2} i v \otimes (\jac(u) - \abs{u}^2 \jac(\CCC\tau\BBB)) + iv \otimes \jac(\CCC\tau\BBB) \\
    &= \abs{u}^{-1} v \otimes \NNN\di\BBB \abs{u} + iv \otimes (m \abs{u}^{-2} \jac(u) - (m-1) \jac(\CCC\tau\BBB)),
  \end{align*}
  as desired.
  
  \textit{\NNN Step 3 \BBB (\NNN C\BBB hain rule for the squared gradient and pre-jacobian):}
  As already discussed above, we can restrict ourselves to \NNN points in \BBB $O \setminus \setof{u = 0}$ at which $v$, $u$, and $\abs{u}$ are approximately differentiable.
  Using (\ref{grad_after_doubling})\NNN, \BBB $\sprod{v}{iv} = 0$\NNN, and $\abs{u} = \abs{v}$ \BBB we have that
  \begin{align}
    \abs{\nabla v}^2
    &= \abs{u}^{-2} \NNN \abs{v}^2 \BBB \abs{\NNN\di\BBB\abs{u}}^2 + m^2 \abs{u}^{-4} \NNN \abs{v}^2 \BBB \abs{\jac(u)}^2 + (m-1)^2 \NNN \abs{v}^2 \BBB \abs{\jac(\CCC\tau\BBB)}^2 
      - 2m(m-1)\abs{u}^{-2} \NNN \abs{v}^2 \BBB \sprod{\jac(u)}{\jac(\CCC\tau\BBB)} \nonumber\\
    &= \abs{\NNN\di\BBB\abs{u}}^2 + m^2 \abs{u}^{-2} \abs{\jac(u)}^2 + (m-1)^2 \NNN \abs{u}^2 \BBB \abs{\jac(\CCC\tau\BBB)}^2 - 2m(m-1)\sprod{\jac(u)}{\jac(\CCC\tau\BBB)}, \label{absgradv_v1}
  \end{align}
  Let \CCC us decompose \BBB $\nabla u$ into the components parallel and orthogonal to $u$:
  \begin{equation}\label{nablau_decomp}
    \nabla u = \abs{u}^{-2} ( \NNN u \otimes \sprod{\nabla u}{u} \BBB + \NNN (iu) \otimes \jac(u) \BBB ).
  \end{equation}
  Employing the coordinate representation of $\nabla u$ \NNN in \BBB (\ref{approx_grad_coords}), (\ref{gamma_through_A}) with $v$ replaced by $u$, \NNN \eqref{nabla_e_ie}, \BBB and (\ref{dabsu_coords}) leads to
  \begin{align*}
    \sprod{\nabla_{\del{k}} u}{u}
		&= \NNN\sprod*{\BBB\frac{\partial u^1}{\partial x^k} \NNN \tau \BBB + \frac{\partial u^2}{\partial x^k} i \NNN \tau \BBB + \jac(\CCC\tau\BBB)\NNN\left(\BBB\del{k}\NNN\right)\BBB u^1 i \NNN \tau \BBB - \jac(\CCC\tau\BBB)\NNN\left(\BBB\del{k}\NNN\right)\BBB u^2 \NNN \tau \BBB}{u^1 \NNN \tau \BBB + u^2 i \NNN \tau \NNN}\BBB \\
		&= u^1 \frac{\partial \NNN u^1\BBB}{\partial x^k} + u^2 \frac{\partial \NNN u^2\BBB}{\partial x^k} = \abs{u} \NNN\di\BBB \abs{u}\NNN\left(\CCC\frac{\partial}{\partial_{x^k}}\NNN\right)\BBB.
  \end{align*}
  \NNN Using the above identity in \eqref{nablau_decomp}, \BBB we derive that 
  \begin{equation}\label{abs_nabla_u_squared}
    \abs{\nabla u}^2 
    = \abs{u}^{-4} (\abs{u}^2 \abs{\NNN\di\BBB\abs{u}}^2 \abs{u}^2 + \abs{\jac(u)}^2 \abs{u}^2)
    = \abs{\NNN\di\BBB\abs{u}}^2 + \abs{u}^{-2} \abs{\jac(u)}^2.
  \end{equation}
  \NNN Hence, \BBB (\ref{absgradv_v1}) implies
  \begin{align*}
    \abs{\nabla v}^2 
    &= m^2 \abs{\nabla u}^2  + (1 - m^2) \abs{\NNN\di\BBB\abs{u}}^2 + (m-1)^2 \abs{u}^2 \abs{\jac(\CCC\tau\BBB)}^2 - 2m(m-1)\sprod{\jac(u)}{\jac(\CCC\tau\BBB)},
  \end{align*}
  which shows (\ref{absgrad_after_doubling}).

  From the very definition of the pre-jacobian by a direct computation using (\ref{grad_after_doubling}) \NNN and $\abs{u} = \abs{v}$ \BBB we obtain (\ref{prejac_after_doubling}).

  \textit{\NNN Step 4 \BBB (chain rule for the vorticity):}
  From this point on, we additionally assume that $u \in L^\infty(TO)$.
	Using the smoothness of \NNN $\tau$, $m \geq 1$, \BBB and (\ref{absgrad_after_doubling}) we can estimate
  \begin{align*}
    \norm{\nabla v}_{L^1}
    &\leq m \norm{\nabla u}_{L^1} + (m - 1) \norm{u}_{L^1} \norm{\jac(\CCC\tau\BBB)}_{L^\infty}
    + 2m(m-1) \NNN \norm{\nabla u}_{L^1} \BBB \norm{\CCC\jac(\tau)\BBB}_{L^\infty} < \infty,
  \end{align*}
  by which we can define $\vort(v)$ in distributional sense.
  Let now $\Phi = \vphi \vol$ with $\vphi \in C^\infty_c(O)$, by (\ref{ibp_vorticity}) and (\ref{prejac_after_doubling}) it follows that
  \begin{align}
    \vort(v)(\Phi)
    = \int_O \NNN - \BBB \jac(v) \wedge \NNN\di\BBB \vphi + \kappa \vol
    &= \NNN m \BBB \int_O \NNN \left(- \BBB \jac(u) \wedge \di \vphi + \kappa \vol \NNN \right) - \BBB (m - 1) \int_O \NNN \left( - \BBB \abs{u}^2 \jac(\CCC\tau\BBB) \wedge \di \vphi \NNN + \BBB \kappa \vol \NNN \right) \BBB \nonumber\\ 
    &= \CCC m\vort(u)(\Phi) \BBB \NNN - \BBB (m-1) \int_O \NNN \left( - \BBB \abs{u}^2 \jac(\CCC\tau\BBB) \wedge \di \vphi \NNN + \BBB \kappa \vol \right). \label{vort_after_using_jac}
  \end{align}
  By the product rule from Lemma \ref{lem:product_rule} applied to $\abs{u} \NNN \tau \BBB$ \NNN and $\abs{\di u\abs{u}} \leq \abs{\nabla u}$ which follows from \eqref{abs_nabla_u_squared} \BBB we can estimate
  \begin{equation*}
    \norm{\nabla (\abs{u} \NNN \tau \BBB)}_{L^1} \leq \norm{\nabla u}_{L^1} + \norm{u}_{L^1} \norm{\nabla \NNN \tau \BBB}_{L^\infty} < \infty.
  \end{equation*}
  \CCC Hence, $\vort(\abs{u} \tau)$ \BBB can be defined in distributional sense.
  Furthermore, by the same product rule we get that
  \begin{equation*}
    \jac(\abs{u} \NNN \tau \BBB) = \sprod{\NNN \tau \BBB \otimes \NNN\di\BBB\abs{u} + \abs{u} \nabla \NNN \tau \BBB}{i \abs{u} \NNN \tau \BBB} = \abs{u}^2 \jac(\CCC\tau\BBB).
  \end{equation*}
  The \NNN above \BBB equality combined with (\ref{vort_after_using_jac}) gives (\ref{vorticity_after_doubling}).
\end{proof}

\section{Statement of the main result}
This section is devoted to the statement of our \NNN $\Gamma$\BBB-convergence result.
Let us first define all required objects.
Given $m \in \nn$, \CCC the set of the \BBB \textit{admissible spin fields} \CCC is \BBB 
\begin{equation*}
  \asm\CCC(S)\BBB \defas \setof*{u \in SBV^2(TS) \colon (u^+)^m = (u^-)^m \text{ $\haus^1_g$-a.e. on } \mathcal J_u}.
\end{equation*}
Here, $u^+$\CCC, $u^-$, and $\mathcal{J}_u$ are \BBB the approximate upper-value\CCC, the approximate lower-value, and the approximate jump-set \BBB of $u$ (see also Definition \ref{def:approximate_jump}).
Furthermore, for $x \in \mathcal J_u$ the expression $(u^+(x))^m$ stands for $\power(u^+(x))$ for some local frame $\setof{e, i e}$ in the vicinity of $x$ (see (\ref{mth_power}) for a definition of $\power$).
In this context it is important to note that the condition $(u^+)^m = (u^-)^m$ is independent of the choice of local frame.
Note also that, in the case $m = 1$ the spin field $u$ satisfies $u^+ = u^-$ at $\haus^1_g$-a.e. point on $\mathcal J_u$ and therefore
\begin{equation*}
  \as^{(1)}\NNN(S)\BBB = W^{1, 2}(TS).
\end{equation*}

\CCC Given $\eps > 0$, t\BBB he Ginzburg-Landau energy \CCC of $u \in \asm$ \NNN is defined as \BBB
\begin{equation}\label{def_gl}
  \gl(u) \defas \frac{1}{2} \int_S \abs{\nabla u}^2 + \frac{1}{2\eps^2} (1 - \abs{u}^2)^2 \vol + \haus^1_g(\mathcal J_u).
\end{equation}
\NNN Given $O \subset S$ open, t\BBB he following two sets of Dirac measures on \NNN $O$ \BBB will be \CCC relevant for \BBB the compactness result:
\begin{align*}
  \tilde X^{(m)}\CCC(\NNN O \CCC)\BBB &\defas \setof*{\mu = \sum_{k = 1}^K \frac{d_k}{m} \dirac_{x_k} \colon K \in \nn, \, d_k \in \zz, \, x_k \in \NNN O \BBB, \, \mu(S) = \echar}, \\
  X^{(m)}\CCC(\NNN O \CCC)\BBB &\defas \setof*{\mu = \sum_{k = 1}^K \frac{d_k}{m} \dirac_{x_k} \in \tilde X^{(m)}\NNN(O)\BBB \colon d_k \in \setof{-1, 1}, \, x_k \neq x_{l} \text{ for } k \neq l},
\end{align*}
where $\echar$ is the \NNN \textit{Euler characteristic} \BBB of $S$.
\NNN 
For any Dirac measure $\mu = \sum_{k = 1}^K d_k \dirac_{x_k}$ we will from now on implicitly assume that $x_k \neq x_l$ for $k \neq l$.
Furthermore, we will denote by $\abs{\mu}(S) \defas \sum_{k = 1}^K \abs{d_k}$ its total variation in $S$.
\BBB

The set $\lsm\NNN(S)\BBB$ of \CCC the \BBB \textit{limit spin fields} consists exactly of those $u \in SBV(TS)$ such that
\begin{enumerate}[label=(\roman*)]
  \item $\abs{u} = 1$ a.e. in $S$;
  \item $(u^+)^m = (u^-)^m$ $\haus^1_g$-a.e. on $\mathcal J_u$ and $\haus^1_g(\mathcal J_u) < \infty$;
  \item $\vort(u) \in X^{(m)}$ and $u \in SBV^2_\loc(S \setminus \spt(\vort(u)); TS)$;
  \item $\abs{\nabla u} \in L^p(S)$ for all $p \in [1, 2)$.
\end{enumerate}

Let $u \in \lsm\NNN(S)\BBB$ \NNN with \BBB $\vort(u) = \sum_{k = 1}^K \frac{d_k}{m} \dirac_{x_k}$\CCC. \BBB \CCC We \BBB define the \textit{renormalized energy} of \CCC such \BBB $u$ as
\begin{equation}\label{def_ren_energy}
  \renm(u) \defas \lim_{r \to 0} \left( \frac{1}{2} \int_{S_r(\vort(u))} \abs{\nabla u}^2 \vol - \frac{\abs{\vort(u)}\CCC(S)\BBB}{m} \pi \abs{\log r} \right),
\end{equation}
\CCC where we have introduced the notation \BBB
\begin{equation*}
  S_r(\vort(u)) \defas S \setminus \bigcup_{k = 1}^K \bar B_r(x_k).
\end{equation*}
Note that we will show in Lemma \ref{lem:ren_ener_welldefined} that the renormalized energy is well-defined, i.e.~the limit in (\ref{def_ren_energy}) exists and belongs to $[-\infty, \infty)$ for all $u \in \lsm\NNN(S)\BBB$.

Let us continue by introducing a minimum problem on Euclidean balls.
Given $r > 0$, let $B_r(0)$ denote the Euclidean open ball centered at the origin.
For any $v \in W^{1, 2}(B_r(0); \rr^2)$ we define
\begin{equation}\label{barGLeps}
	\glmflat(v, B_r(0))
	\defas \frac{1}{2 m^2} \int_{B_r(0)} \abs{\nabla{v}}^2 
		+ (m^2 - 1) \abs{\nabla \abs{v}}^2 
		+ \frac{m^2}{2\eps^2} (1 - \abs{v}^2)^2 \di x.
\end{equation}
Then, for $\lambda \in \sph^1$ let
\begin{equation}\label{def_core_minprob_flat}
  \bar \gamma_\eps^{(m)}(r, \lambda)
  \defas \min\setof*{\glmflat(v, B_r(0)) \colon v \in W^{1, 2}(B_r(0); \rr^2), \, v = \lambda \frac{x}{\abs{x}} \text{ on } \partial B_r(0)}\CCC,\BBB
\end{equation}
\CCC where the product $\lambda \frac{x}{\abs{x}}$ is meant as a product in $\cc$. \BBB
Note that by direct methods, we can show that the minimum in (\ref{def_core_minprob_flat}) exists.
As $\glmflat(v, B_r(0)) = \glmflat(\tilde v, B_r(0))$ for $\tilde v(x) \defas \lambda^{-1} v(\eps x)$ and $\tilde v$ is admissible for the minimum problem in the definition of $\bar \gamma_1^{(m)}(\NNN\eps^{-1} r\BBB, 1)$ we see that for any $r > 0$, $\eps > 0$, and $\lambda \in \sph^1$
\begin{equation*}
	\bar \gamma^{(m)}(\NNN\eps^{-1}r\BBB) \defas \bar \gamma_1^{(m)}(\NNN\eps^{-1}r\BBB, 1) = \bar \gamma_\eps^{(m)}(r, \lambda).
\end{equation*}
The following convergence result was proved in \cite{goldman_glflat} (see Lemma 3.9):
\begin{lemma}
	\label{lem:core_convergence_flat}
	There exists $\gamma_m \in \rr$ such that:
	\begin{equation}
		\label{core_convergence_flat}
		\lim_{R \to \infty} \left(\bar \gamma^{(m)}(R) - \frac{\pi}{m^2} \abs{\log(R)}\right) = \gamma_m.
	\end{equation}
  Consequently, given $(\lambda_\eps) \subset \sph^1$ and $(r_\eps) \subset \rr_+$ such that $\lim_{\eps \to 0} \NNN \eps^{-1} r_\eps \BBB = \infty$, we have that
	\begin{equation}
		\label{core_convergence_flat_version2}
		\lim_{\eps \to 0} \left(\bar \gamma_\eps^{(m)}(r_\eps, \lambda_\eps) - \frac{\pi}{m^2} \log \NNN\left(\BBB\frac{r_\eps}{\eps}\NNN\right)\BBB\right) = \gamma_m.
	\end{equation}
\end{lemma}

We are ready to state our \NNN$\Gamma$\BBB-convergence result:
\begin{theorem}[\NNN$\Gamma$\BBB-convergence]\label{thm:gamma_convergence} The following Gamma-convergence result holds true for the sequence of functionals $(GL_\eps)$ \CCC in \eqref{def_gl}.\BBB
\begin{enumerate}[label=(\roman*)]
  \item \textit{Compactness}: Let $(u_\eps) \subset \asm\CCC(S)\BBB$ be a bounded sequence in $L^\infty(TS)$ such that for all $\eps > 0$
  \begin{equation}\label{energy_bound}
    \gl(u_\eps) \leq \frac{N}{m} \pi \leps + C,
  \end{equation}
  where $N \in \nn$ and $C > 0$ are constants independent of $\eps$.
  Then, there exists $\mu \in \tilde X^{(m)}$ with $\abs{\mu} \leq N$ such that, up to subsequences, it holds that
  \begin{equation}\label{compactness_vorts}
    \vort(u_\eps) \weakto \mu \text{ weakly in } W^{-1, \infty}(S).
  \end{equation}
  If $\abs{\mu} = N$ we can find $u \in \lsm\CCC(S)\BBB$ satisfying $\vort(u) = \mu$ such that, up to subsequences,
  \begin{equation}\label{w12loc_comp}
    u_\eps \weakto u \text{ weakly in } SBV^2_\loc(S \setminus \spt(\mu); TS)
  \end{equation}
  and for all $p \in [1, 2)$
  \begin{equation}\label{w1p_comp}
    u_\eps \weakto u \text{ weakly in } SBV^p(TS).
  \end{equation}
  \item \textit{\NNN$\Gamma$\BBB-liminf}: Let $(u_\eps) \subset \asm\CCC(S)\BBB$ and $u \in \lsm\CCC(S)\BBB$ \CCC be \BBB such that $u_\eps \to u$ in $L^1(TS)$.
	Then,
	\begin{equation}\label{gamma_liminf}
		\liminf_{\eps \to 0} \gl(u_\eps) - \frac{\NNN \abs{\vort(u)} \BBB}{m} \pi \leps \geq \renm(u) + \haus^1_g(\mathcal J_u) + m \NNN \abs{\vort(u)} \BBB \gamma_m\NNN.\BBB
	\end{equation}
  \item \textit{\NNN$\Gamma$\BBB-limsup}: For any $u \in \lsm(S)$ there exists a sequence $(u_\eps) \subset \asm(S)$ such that $u_\eps \to u$ in $L^1(TS)$ and
	\begin{equation}\label{gamma_limsup}
		\limsup_{\eps \to 0} \gl(u_\eps) - \frac{\NNN \abs{\vort(u)} \BBB}{m} \pi \leps
		\leq \renm(u) + \haus^1_g(u) + m \NNN \abs{\vort(u)} \BBB \gamma_m\NNN.\BBB
	\end{equation}
\end{enumerate}
\end{theorem}

\section{Proof of \NNN$\Gamma$\BBB-Convergence}
\NNN All constants appearing in this paper are implicitly assumed to be independent of $\eps$ and may change from line to line.
We also employ standard asymptotic notation. For example $\BigO_r(1)$ stands for a bounded term as $r \to 0$.
Furthermore, we will not explicitly write out $\eps$ in our asymptotic notation.
For example $\littleo(1)$ is shorthand for $\littleo_\eps(1)$.
\BBB
\subsection{Compactness}\label{sec:comp}
In this subsection we will prove \CCC the \BBB compactness resul\CCC t \BBB for our sequence of vector fields \CCC with equibounded energy\BBB.
Given an open set $O \subset S$ we define
\begin{equation*}
  \asm(O) \defas \setof*{u \in SBV^2(TO) \colon (u^+)^m = (u^-)^m \text{ $\haus^1_g$-a.e. on } \mathcal J_u}.
\end{equation*}
and \NNN for $u \in \asm(O)$ \BBB the localized \CCC generalized \BBB Ginzburg-Landau energy as
\begin{equation*}
  \gl(u, O) \defas \frac{1}{2} \int_O \abs{\nabla u}^2 + \frac{1}{2\eps^2} (1 - \abs{u}^2)^2 \vol + \haus^1_g(\mathcal J_u \cap O).
\end{equation*}

Furthermore, given $v \in C^\infty(TO)$ and a geodesic ball $B \subset S$ the localized degree of $v$ around $\partial B$ in $O$ is given by
\begin{equation}\label{def_local_degree}
  \dg(v, \partial B; O) \defas \begin{cases}
    \deg(v, \partial B) & \text{if } B \subset O \text{ and } \abs{v} \geq \frac{1}{2} \text{ on } \partial B, \\
    0 & \text{else,}
  \end{cases}
\end{equation}
where $\deg(v, \partial B)$ is as in (\ref{def_degree_with_jac}).

In the proofs of this subsection we will employ the following ball-construction result:
\begin{theorem}[Ball construction in an open subset]\label{thm:ballcon}
  For every $T, \, C > 0$, every integer $n > T - 1$ and every $\NNN q \in (0,  1 - T(n+1)^{-1})$, there exist constants $\eps_0, \, \sigma_0, \, \tilde C > 0$ such that the following holds true: if $\eps \in (0, \eps_0)$, $\sigma \in [\eps^q, \sigma_0]$, and $v \in C^\infty(TO)$, for an open set $O \subset S$ with Lipschitz boundary satisfying the energy upper bound
  \begin{equation}\label{ballcon_upper_bound}
    \gl(v, O) \leq T \pi \leps + C,
  \end{equation} 
  there exists $K_\sigma \in \NNN \nn\BBB$ and a finite collection of pairwise disjoint geodesic balls $\balls^{(\sigma)} = \setof{B_k^{(\sigma)}}_{k = 1}^{K_\sigma}$, each one with radius denoted by $r_k^{(\sigma)}$, such that the following properties are satisfied:
  \begin{enumerate}[label=(\roman*)]
    \item $\setof{x \in O \colon \abs{v(x)} \leq \frac{1}{2}} \subset \bigcup_{k = 1}^{K_\sigma} B_k^{(\sigma)}$; 
    \item $D_\sigma \defas \sum_{k = 1}^{K_\sigma} \abs{d_k^{(\sigma)}} \leq n$, where $d_k^{(\sigma)} \defas \dg(v, \partial B_k^{(\sigma)}; O\NNN)\BBB$;
    \item  $\sum_{k = 1}^{K_\sigma} r_k^{(\sigma)} \leq (n+1)\sigma$;
    \item 
      $\gl(v_\eps, B_k^{(\sigma)} \cap O) \geq \abs{d_k^{(\sigma)}} ( \pi \log\left(\frac{\sigma}{\eps} \right) - \tilde C)$ for $k = 1, \dots, K_\sigma$.
  \end{enumerate}
\end{theorem}

Note that the above result is a generalization of Proposition 8.2 in \cite{jerrard_glman}.
The necessary modifications can be found in Appendix \ref{sec:proofballcon}.

For the readers convenience we \CCC recall here \BBB an important result from \cite{jerrard_glman} (see also Lemma A.1):
\begin{lemma}[Energy lower bound on circles]\label{lem:ener_lowbound_circles}
	Let $u \in C^\infty(TS)$, $\eps > 0$, $x \in O$, and $r \in (\eps, r^*)$.
	Then,
	\begin{equation}\label{ener_lowbound_circles}
		\frac{1}{2} \int_{\partial B_r(x)} \abs{\nabla u}^2 + \frac{1}{2\eps^2} \NNN(1 - \abs{u}^2)^2 \vol \BBB \geq \lambda_\eps\NNN\left(\BBB\frac{r}{d}\NNN\right)\BBB
	\end{equation}
	for
	\begin{equation*}
			d \defas \begin{cases}
				\deg(u, \partial B_r(x)) 
					&\text{if } \min_{\partial B_r(x)} \abs{u} \geq \frac{1}{2}, \\
			\text{any positive integer}
			&\text{else.}
		\end{cases}
	\end{equation*}
Here, $\lambda_\eps \colon \rr \to \rr$ satisfies
	\begin{equation*}
		\lambda_\eps(r) \geq \frac{\pi(1 - Cr^2)}{r + C \eps},
	\end{equation*}
	where $C$ is a universal constant only depending on $S$.
\end{lemma}

\begin{lemma}
  \label{lem:nablav_repr}
	Given an open set $O \subset S$ and $v \in W^{1, 1}(TO)$, the following holds true at a.e.~point in $O$:
	\begin{equation}
		\label{orth_decomp_nablav}
		\nabla v = \begin{cases}
			\NNN \abs{v}^{-1} v \otimes \di \abs{v} + \abs{v}^{-2} (i v) \otimes \jac(v) \BBB, &\text{ if } v \neq 0, \\
			0 &\text{ else.}
		\end{cases}
	\end{equation}
\end{lemma}

\begin{proof}
  By an approximation argument we can reduce ourselves to the case of $v \in C^\infty(TO)$.
  Furthermore for a.e.~$x \in \setof{x \in O \colon v(x) = 0}$ it holds that $\nabla v\CCC(x)\BBB = 0$.
 	It remains to investigate a point\NNN s in $O$ \BBB \CCC where \BBB $\NNN v\BBB \neq 0$.
  Using the product rule\CCC, we have that\BBB
	\begin{equation*}
		\nabla v
    = \nabla{\NNN( \abs{v} \abs{v}^{-1} v )\BBB}
		= \NNN \abs{v}^{-1} v \otimes \di\abs{v} \BBB + \abs{v} \nabla{\NNN( \abs{v}^{-1} v)\BBB}.
	\end{equation*}
  Differentiating both sides of $\sprod{\frac{v}{\abs{v}}}{\frac{v}{\abs{v}}} = 1$ we see that \NNN $\sprod{\nabla(\abs{v}^{-1} v)}{v} = 0$. \BBB
  \NNN Hence, by yet another product rule, it follows that \BBB
  \begin{equation*}
    \nabla \NNN(\abs{v}^{-1} v)\BBB 
		= \NNN \abs{v}^{-1} (i v) \otimes \sprod{\nabla (\abs{v}^{-1} v)}{\abs{v}^{-1} i v} \BBB
    = \NNN\abs{v}^{-2} (iv) \otimes \BBB \sprod{\di(\abs{v}^{-1}) \otimes v + \abs{v}^{-1} \nabla v}{\NNN i v\BBB}
    = \NNN \abs{v}^{-3} (iv) \otimes \jac(v)\BBB,
  \end{equation*}
  which leads to (\ref{orth_decomp_nablav}).
\end{proof}

From this point on, given a coordinate neighborhood $O \subset S$, we will shortly write
\begin{equation*}
  C^\infty(\NNN \bar{O} \BBB; \sph^1) = \setof{\NNN \tau \BBB \in C^\infty(T\NNN \bar{O}\BBB) \colon \abs{\NNN \tau \BBB} = 1 \text{ on } O}
\end{equation*}
for smooth \NNN (up to the boundary) \BBB unit-length tangent vector fields on $O$.
We also define \NNN the truncation \BBB $\trunc \colon TS \to TS$ by
\begin{equation*}
  \trunc(X) \defas \begin{cases}
    \frac{X}{\abs{X}}, & \abs{X} \geq \frac{1}{2}, \\
    2 X, &\text{ else},
  \end{cases}
\end{equation*}
for any $X \in TS$.

\begin{lemma}\label{lem:absveps_e_comp}
  Given an open coordinate neighborhood $O$, let $\NNN \tau \BBB \in C^\infty(\bar O; \sph^1)$ and $(v_\eps) \subset C^\infty(TO)$ be a bounded sequence in $L^\infty(TO)$ such that for all $\eps$
  \begin{equation}
    \label{absveps_e_comp_ebound}
    \gl(v_\eps, O) \leq C \leps\NNN . \BBB
  \end{equation}
  Then,
  \begin{equation}
    \label{absveps_e_conv}
    \vort(\abs{v_\eps} \NNN \tau \BBB) \weakto 0 \text{ weakly in } W^{-1, \infty}_0(O).
  \end{equation}
\end{lemma}

\begin{proof}
  \textit{\NNN Step 1\BBB:}
  Let $n > \frac{C}{\pi} - 1$ and $q \in (0, \CCC 1 - \BBB \frac{C}{\pi(n+1)})$, where $C$ is as in (\ref{absveps_e_comp_ebound}).
  For each \NNN$\eps$\BBB we then apply Theorem \ref{thm:ballcon} to $v_\eps$, where $n, \, q$ are as above and $\sigma = \sigma_\eps = \eps^q$.
  Hence, there exists $\eps_0 > 0$ such that for all $\eps \in (0, \eps_0)$ there exists a finite collection of disjoint closed geodesic balls $\balls_\eps = \setof{B_k^{(\eps)}}_{k = 1}^{K_\eps}$ such \NNN that\BBB
  \begin{align}
    \setof{x \in \NNN O \BBB \colon \abs{v_\eps} < \tfrac{1}{2}} &\NNN\subset\BBB \bigcup_{k = 1}^{K_\eps} B_k^{(\eps)}, \label{covering_absveps_e_comp} \\
    \sum_{k = 1}^{K_\eps} r_k^{(\eps)} &\NNN\leq\BBB (n+1) \eps^q, \label{rad_absveps_e_comp}
  \end{align}
  where $r_k^{(\eps)}$ is the radius of $B_k^{(\eps)}$.

  \textit{\NNN Step 2\BBB:}
  \NNN In $O$ \BBB we define $\NNN w_\eps\BBB \defas \NNN\abs{v_\eps} \tau\BBB$ and $\NNN \tilde w_\eps\BBB \defas \trunc \NNN w_\eps\BBB$.
  Our aim now is to show
  \begin{equation}
    \label{vortweps_vorttildweps}
    \vort(w_\eps) - \vort(\tilde w_\eps) \weakto 0 \text{ weakly in } W^{-1, \infty}_0(\NNN O \BBB).
  \end{equation}
  Note that \NNN there exists a constant $C > 0$ \BBB such that $\haus^2_g(B_r(x))) \leq C r^2$ for all $x \in \NNN S \BBB$ and $r \in (0, \NNN r^* \BBB)$.
  Hence, \NNN \eqref{absveps_e_conv}, \BBB (\ref{covering_absveps_e_comp})\NNN, and $\abs{v_\eps} = \abs{u_\eps}$ \BBB we derive that
  \begin{equation*}
    \int_{\setof{\abs{w_\eps} < \frac{1}{2}}} \abs{w_\eps}^2 \vol 
    \leq \NNN \frac{1}{4} \BBB \sum_{k = 1}^{K_\eps} \NNN \haus^2_g(B_k^{(\eps)}) \BBB
    \leq \NNN C \BBB \sum_{k = 1}^{K_\eps} (r_k^{(\eps)})^2
    \leq \NNN C \BBB \left(\sum_{k = 1}^{K_\eps} r_k^{(\eps)}\right)^2
    \leq \NNN C(n+1)^2 \BBB \eps^{2q}.
  \end{equation*}
  On the set $\setof{\abs{w_\eps} \geq \CCC 1/2\BBB}$, by the definition of $w_\eps$ \NNN and \BBB $\tilde w_\eps$ \CCC it follows \BBB that
  \begin{equation*}
    \abs{w_\eps - \tilde w_\eps}^2
    = (1 - \abs{w_\eps})^2 
    \CCC \leq (1 - \abs{v_\eps})^2 (1 + \abs{v_\eps})^2\BBB,
  \end{equation*}
  and therefore, with (\ref{absveps_e_comp_ebound}) \NNN and the definition of $GL_\eps$ \BBB
  \begin{equation*}
    \int_{\setof{\abs{w_\eps} \geq \frac{1}{2}}} \abs{w_\eps - \tilde w_\eps}^2 \vol
    \leq \CCC 4 \BBB \gl(v_\eps, \NNN O \BBB) \eps^2 \leq \CCC C \BBB \leps \eps^2.
  \end{equation*}
  Combining the aforementioned estimates it follows that 
  \begin{equation}
    \label{absveps_e_comp_l2_dist}
    \norm{w_\eps - \tilde w_\eps}_{L^2}^2 \leq \NNN C \BBB \leps \eps^2 + \NNN C(n+1)^2 \BBB \eps^{2q} \NNN \leq C \eps^{2q}\BBB,
  \end{equation}
  where we have used $q < 1$.
  We next estimate $\norm{\nabla w_\eps}_{L^2}$.
  By the product rule,
  \begin{equation*}
    \nabla w_\eps = \nabla (\abs{v_\eps} \NNN \tau \BBB) = \NNN \tau \otimes \BBB \di\abs{v_\eps} + \abs{v_\eps} \nabla \NNN \tau \BBB.
  \end{equation*}
  Hence, due to the boundedness of $(v_\eps)$ in $L^\infty(T\NNN O \BBB)$, the smoothness of \NNN $\tau$ \BBB, \NNN and $\abs{\di \abs{v_\eps}} \leq \abs{\nabla v_\eps}$ which follows from \BBB (\ref{orth_decomp_nablav}) it holds that \NNN $\abs{\nabla w_\eps} \leq \abs{\nabla v_\eps} + C$\BBB.
  With \NNN (\ref{absveps_e_comp_ebound}) \BBB this shows
  \begin{equation}
    \label{absveps_e_comp_gradweps}
    \norm{\nabla w_\eps}_{L^2}^2 \leq \int_{\NNN O \BBB} \abs{\nabla v_\eps}^2 \vol + C \leq 2 \gl(v_\eps, \NNN O \BBB) + C \NNN \leq C \leps\BBB.
  \end{equation}
  We will now derive a similar estimate for $\norm{\CCC \nabla \BBB \tilde w_\eps}_{L^2}$.
  At \NNN points in $O$ for which \BBB $\abs{\NNN v_\eps \BBB} > \frac{1}{2}$\NNN, \BBB by the product rule\NNN, it holds that\BBB
  \begin{equation*}
    \nabla \tilde w_\eps = \nabla \NNN (\abs{w_\eps}^{-1} w_\eps ) \BBB
		= - \NNN \abs{w_\eps}^{-2} w_\eps \otimes \BBB \di\abs{v_\eps}  + \NNN \abs{w_\eps}^{-1} \BBB \nabla w_\eps.
  \end{equation*}
  Taking the norm on both sides of the above \CCC equation leads \BBB to the following \NNN bound:\BBB
  \begin{equation*}
    \abs{\nabla \tilde w_\eps} \leq 2 \abs{\di\abs{v_\eps}} + 2 \abs{\nabla w_\eps}.
  \end{equation*}
  At \NNN points in $O$ \BBB with $\abs{\NNN v_\eps \BBB} < \frac{1}{2}$ we have \NNN $\abs{\nabla \tilde w_\eps} = 2 \abs{\nabla w_\eps}$\BBB.
  The last two estimates combined with \CCC \eqref{orth_decomp_nablav} and \BBB (\ref{absveps_e_comp_gradweps}) \CCC result in \BBB
  \begin{equation}
    \label{absveps_e_comp_gradtildeweps}
    \norm{\nabla \tilde w_\eps}_{L^2}^2 \NNN \leq C \leps\BBB,
  \end{equation}
	\CCC which eventually gives \eqref{vortweps_vorttildweps} by using (\ref{vort_diff_bound}) together with (\ref{absveps_e_comp_l2_dist}) and (\ref{absveps_e_comp_gradweps}). \BBB

  \textit{\NNN Step 3\BBB:}
  By the previous step\NNN, \BBB (\ref{absveps_e_conv}) is proved if we show that
  \begin{equation*}
    \vort(\tilde w_\eps) \weakto 0 \text{ weakly in } W^{-1, \infty}_0(\NNN O \BBB).
  \end{equation*}
  Fix an arbitrary test-function $\vphi \in W^{1, \infty}_0(\NNN O\BBB)$ with $\norm{\NNN\di\BBB \vphi}_{L^\infty} \leq 1$.
  Let $B = B_r(x) \in \balls_\eps$ such that $B \subset \NNN O\BBB$.
  As $\tilde w_\eps = \NNN \tau \BBB$ on $\partial B$, by Stoke's theorem and the smoothness of $\NNN \tau \BBB$
  \begin{equation*}
    \int_B \vort(\tilde w_\eps) = \int_{\partial B} \jac(\CCC \tilde w_\eps\BBB) = \int_{\partial B} \jac(\NNN \tau \BBB) = \deg(\NNN \tau \BBB, \partial B) = 0.
  \end{equation*}
  Hence, by $\norm{\di \vphi}_{L^\infty} \leq 1$, (\ref{dist_jac_pw_bound}), (\ref{absveps_e_comp_ebound}), and (\ref{absveps_e_comp_gradtildeweps}) it follows that 
  \begin{align*}
    \abs*{\int_B \vphi \vort(\tilde w_\eps)}
    &= \abs*{\vphi(x) \int_B \vort(\tilde w_\eps) + \int_B (\vphi(y) - \vphi(x)) \vort(\tilde w_\eps(y))} \\
    &\leq \CCC 2 \BBB r \int_B \abs{\vort(\tilde w_\eps)} \vol
    \leq \CCC 2 \BBB r \int_B \abs{\nabla \tilde w_\eps}^2 + \abs{\kappa} \vol
    \NNN \leq C r \leps.\BBB
  \end{align*}
  Let us now consider $B = B_r(x) \in \balls_\eps$ such that $B \setminus \NNN O \BBB \neq \emptyset$, instead.
  Take $x' \in B \cap \partial O$, \CCC using that $\vphi(x') = 0$ \BBB we can derive \CCC, similarly to the previous case, that \BBB
  \begin{align*}
    \abs*{\int_B \vphi \vort(\tilde w_\eps)}
    \leq \int_B \abs{\vphi(y) - \vphi(x')} \abs{\vort(\tilde w_\eps(y))} \vol(y)
    \leq 2r \int_B \abs{\nabla \tilde w_\eps}^2 + \abs{\kappa} \vol \NNN \leq C r \leps\BBB.
  \end{align*}
  In $\NNN O \BBB \setminus \cup_{k = 1}^{K_\eps} B_k^{(\eps)}$ we have that $\abs{\tilde w_\eps} = 1$ and therefore $\vort(\tilde w_\eps) = 0$.
  Consequently, \NNN by \eqref{dist_jac_pw_bound}, \eqref{absveps_e_comp_ebound}, and \eqref{rad_absveps_e_comp} \BBB we conclude that
  \begin{align*}
    \abs*{\int_{\NNN O \BBB} \vphi \vort(\tilde w_\eps)}
    &= \sum_{k = 1}^{K_\eps} \abs*{\int_{B_k^{(\eps)}} \vphi \vort(\tilde w_\eps)}
    \leq \NNN C \BBB \left(\sum_{k = 1}^{K_\eps} r_k^{(\eps)}\right) \NNN \leps \BBB \leq \NNN C \BBB (n+1) \eps^q \NNN \leps \BBB.
  \end{align*}
  The weak convergence in (\ref{absveps_e_conv}) follows \CCC by the arbitrariness of $\vphi$\BBB.
\end{proof}

In the next lemma we will derive our initial compactness result for the vorticities.
\begin{lemma}[Initial \CCC vorticity \BBB compactness]\label{lem:initial_comp}
  Let $(u_\eps) \subset \asm\CCC(S)\BBB$ be a bounded sequence in $L^\infty(TS)$ such that for all \NNN$\eps$\BBB:
	\begin{equation}
    \label{initial_comp_ebound}
		\gl(u_\eps) \leq C \leps\NNN.\BBB
	\end{equation}
  Then, there exists a measure $\mu \in \tilde X^{(m)}$ such that, up to subsequences,
	\begin{equation*}
		\vort(u_\eps) \weakto \mu \text{ weakly in } W^{-1, \infty}(S).
	\end{equation*}
\end{lemma}

\begin{proof}
  \textit{\NNN Step 1 (locally $m$-pling the angles)\BBB:}
	We start by localizing the problem.
  Let \NNN $O \subset S$ be a coordinate neighborhood with smooth \BBB boundary.
	Furthermore, choose an arbitrary $\NNN\tau\BBB \in C^\infty(\NNN\bar O\BBB; \sph^1)$ and set $v_\eps \defas \power(u_\eps)$ in $\NNN O \BBB$.
  By (\ref{absgrad_after_doubling}) \NNN it \BBB holds that
  \begin{align*}
    \abs*{\int_{\NNN O\BBB} \abs{\nabla{v_\eps}}^2 \vol}
    &= m^2 \int_{\NNN O\BBB} \abs{\nabla{u}_\eps}^2 \vol + (1 - m^2) \int_{\NNN O\BBB} \abs{\di\abs{u_\eps}}^2 \vol
    + (m-1)^2 \int_{\NNN O\BBB} \abs{u_\eps}^2 \abs{\NNN \jac(\tau)\BBB}^2 \vol \\
		&\phantom{=}\quad- 2 m (m-1) \int_{\NNN O\BBB} \sprod{\jac(u_\eps)}{\jac(\NNN\tau\BBB)} \vol.
  \end{align*}
  Using Young's inequality we \NNN derive\BBB
  \begin{equation*}
      \int_{\NNN O\BBB} \sprod{\jac(u_\eps)}{\jac(\NNN \tau \BBB)} \vol \leq \int_{\NNN O \BBB} \abs{u_\eps} \abs{\nabla u_\eps} \abs{\NNN \jac(\tau)\BBB} \vol 
      \leq \frac{1}{2} \int_{\NNN O\BBB} \abs{u_\eps}^2 \abs{\nabla u_\eps}^2 \vol + \frac{1}{2} \int_{\NNN O\BBB} \abs{\NNN \jac(\tau)\BBB}^2 \vol.
  \end{equation*}
  Consequently, by the energy bound (\ref{initial_comp_ebound}), the boundedness of $(u_\eps)$ in $L^\infty(TS)$, and the smoothness of $\NNN \tau \BBB$ it follows that
  \begin{align}
    &\gl(v_\eps, \NNN O\BBB)
    = \frac{1}{2} \int_{\NNN O\BBB} \abs{\nabla v_\eps}^2 \vol + \frac{1}{4\eps^2} \int_{\NNN O\BBB} \NNN(1-\abs{v_\eps}^2)^2\BBB \vol \nonumber \\
		&\quad\leq \frac{(m-1)^2}{2} \norm{u_\eps}_{L^\infty} \int_{\NNN O\BBB} \abs{\NNN \jac(\tau)\BBB}^2 \vol
    + \left(\frac{m^2}{2} + m(m-1) \norm{u_\eps}_{L^\infty}\right) \int_{\NNN O\BBB} \abs{\nabla u_\eps}^2 \vol
    + \frac{1}{4\eps^2} \int_{\NNN O\BBB} (1-\abs{u_\eps}^2)^2 \vol \nonumber \\
		&\quad\leq \CCC C \BBB(1 + \gl(u_\eps)) \leq C \leps\NNN.\BBB \label{initial_comp_gradveps_bound}
	\end{align}
  Using a standard approximation argument in $W^{1, 2}(T\NNN O\BBB)$ we can assume that $v_\eps \in C^\infty(T\NNN O\BBB)$ for all $\NNN\eps\BBB$ and (\ref{initial_comp_gradveps_bound}) holds true with a possibly larger constant $C$.
  We now apply Theorem \ref{thm:ballcon} to the sequence $(v_\eps)$, where $n > \frac{C}{\pi} - 1$, $q \in (0, 1 - \NNN C(\pi(n+1))^{-1}\BBB)$, and $\sigma = \sigma_\eps = \eps^q$.
  Hence, there exists $\eps_0 > 0$ such that for all $\eps \in (0, \eps_0)$ we can find a finite collection of disjoint geodesic balls $\balls_\eps = \setof{B_k^{(\eps)}}_{k = 1}^{K_\eps}$ such that
  \begin{gather}
    \setof{x \in \NNN O\BBB \colon \abs{v_\eps} \leq \tfrac{1}{2}} \subset \bigcup_{k = 1}^{K_\eps} B_k^{(\eps)}, \label{initial_comp_covering} \\
    D_\eps = \sum_{k = 1}^{K_\eps} \abs{d_k^{(\eps)}} \leq n, \quad \CCC \text{where } d_k^{(\eps)} \defas \dg(v_\eps, \partial B_k^{(\eps)}, \NNN O \BBB), \BBB \label{initial_comp_degree_bound} \\
    \sum_{k = 1}^{K_\eps} r_k^{(\eps)} \leq (n+1) \eps^q, \label{initial_comp_rad_bound}
  \end{gather}
  \CCC where \BBB $r_k^{(\eps)}$ is the radius of $B_k^{(\eps)}$.

	\textit{\NNN Step 2 (Compactness locally)\BBB:}
  We will prove the compactness of $(\vort(u_\eps))$ in $W^{-1, \infty}_0(\NNN O\BBB)$.
  With the notation from the first step, for $\eps \in (0, \eps_0)$ let $\nu_\eps \in \tilde X^{\NNN(1)\BBB}(\NNN O \BBB)$ be the measure
  \begin{equation*}
    \nu_\eps \defas \sum_{k = 1}^{K_\eps} d_k^{(\eps)} \dirac_{x_k^{(\eps)}},
  \end{equation*}
  where $x_k^{(\eps)}$ is the center of the ball $B_k^{(\eps)}$.
  Using (\ref{initial_comp_degree_bound}) we see that \NNN$\abs{\nu_\eps} = D_\eps \leq n < \infty$\BBB.
  Consequently, there exists $\nu \in \tilde X(\NNN O\BBB)$ and a (not relabeled) subsequence $(\nu_\eps)$ such that $\nu_\eps \weakstarto \nu$ weakly* in $\radon(\NNN O\BBB)$, \NNN in particular\BBB
  \begin{equation}\label{intial_comp_nueps_conv}
    \nu_\eps \weakto \nu \text{ weakly in } W^{-1, \infty}_0(\NNN O\BBB).
  \end{equation}
  Setting $\tilde v_\eps \defas \trunc v_\eps$, by (\ref{initial_comp_covering}) and (\ref{initial_comp_rad_bound}), following the proof of \NNN Step 2 \BBB in Lemma \ref{lem:absveps_e_comp} we have that
  \begin{equation}
    \label{initial_comp_veps_tildeveps_conv}
    \vort(v_\eps) - \vort(\tilde v_\eps) \weakto 0 \text{ weakly in } W_0^{\CCC - \BBB 1, \infty}(\NNN O \BBB).
  \end{equation}
  Let $\vphi \in W^{1, \infty}_0(\NNN O\BBB)$ such that $\norm{\nabla \vphi}_{L^\infty} \leq 1$.
  For a ball $B = B_k^{(\eps)} \in \balls_\eps$ that is contained in $\NNN O\BBB$ we see \NNN by \eqref{initial_comp_covering} and \eqref{initial_comp_covering} \BBB that $d_k^{(\eps)} = \deg(v_\eps, \partial B)$.
  Hence, reasoning as in the proof of \CCC Step 3 in \BBB Lemma \ref{lem:absveps_e_comp}, we have
  \begin{align*}
    \abs*{\int_B \vphi \vort(\tilde v_\eps) - \int_B \vphi \di \nu_\eps}
    &= \abs*{\int_B \vphi \vort(\tilde v_\eps) - d_k^{(\eps)} \vphi(x_k^{(\eps)})} \\
    &= \abs*{\int_B \vphi \vort(\tilde v_\eps) - \vphi(x_k^{(\eps)})\int_{\partial B} \jac(\tilde v_\eps)} \\
    &\leq \int_B \abs{\vphi(y) - \vphi(x_k^{(\eps)})} \abs{\vort(\tilde v_\eps)(y)} \vol(y) \\
    &\leq r_k^{(\eps)} \int_B \abs{\vort(\tilde v_\eps)} \vol \leq C r_k^{(\eps)} \leps.
  \end{align*}
  Let us now consider a ball $B = B_k^{(\eps)} \in \balls_\eps$ such that $B \setminus \NNN O\BBB \neq \emptyset$.
  Taking $x' \in \partial \NNN O\BBB \cap B$ and using $\vphi(x') \NNN = d_k^{(\eps)} \BBB = 0$ we \NNN obtain the same estimate as above:\BBB
  \begin{align*}
    \abs*{\int_{B \cap O} \vphi \vort(\tilde v_\eps) - \int_{B \cap O} \vphi \di \nu_\eps}
    = \abs*{\int_{B \cap O} \vphi \vort(\tilde v_\eps)}
    &\leq \int_{B \cap O} \abs{\vphi(y) - \vphi(x')} \abs{\vort(\tilde v_\eps)(y)} \vol(y) \\
    &\leq 2 r_k^{(\eps)} \int_{B \cap O} \abs{\vort(\tilde v_\eps)} \vol \leq C r_k^{(\eps)} \leps.
  \end{align*}
  \NNN Consequently, a\BBB s $\vort(\tilde v_\eps) = 0$ outside $\cup_{k = 1}^{K_\eps} B_k^{(\eps)}$, by the arbitrariness of $\vphi$ and (\ref{initial_comp_rad_bound}) it follows that
  \begin{equation*}
    \vort(\tilde v_\eps) - \nu_\eps \weakto 0 \text{ weakly in } W^{-1, \infty}_0(\NNN O\BBB).
  \end{equation*}
  By (\ref{intial_comp_nueps_conv}) and (\ref{initial_comp_veps_tildeveps_conv}) $\vort(v_\eps) \weakto \nu$ weakly in $W^{-1, \infty}_0(\NNN O\BBB)$.
  Hence, using (\ref{vorticity_after_doubling}), Lemma (\ref{lem:absveps_e_comp}), and possibly extracting a further subsequence, we see that
  \begin{equation*}
    \vort(u_\eps) = \frac{1}{m} \vort(v_\eps) + \frac{m-1}{m} \vort(\abs{v_\eps} \NNN \tau\BBB) \weakto \frac{1}{m} \nu =: \mu \in \tilde X^{(m)}(\NNN O \BBB) \text{ weakly in } W^{-1, \infty}_0(\NNN O \BBB).
  \end{equation*}

  \textit{\NNN Step 3 (Partition of unity)\BBB:}
  The global compactness result then follows by a partition of unity argument.
  Let $K \in \nn$ and let $\setof{O_k}_{k = 1}^K$ be \NNN a \BBB finite famil\NNN y\BBB of coordinate neighborhood\NNN s with smooth boundary such that \BBB $S = \cup_{k = 1}^K \NNN O_k\BBB$.
  \NNN Furthermore, l\BBB et $\setof{\rho_k}_{k = 1}^K$ be a smooth partition of unity subordinate to the cover $\setof{\NNN O_k\BBB}_{k = 1}^K$.
  Due to \NNN Step 2 \BBB we can find for each $k$ a measure $\mu_k \in \tilde X^{(m)}(\NNN O_k\BBB)$ and a (not relabeled) subsequence such that $\vort(u_\eps) \weakto \mu_k$ weakly in $W^{-1, \infty}_0(\NNN O_k\BBB)$.
  \NNN Then, for \BBB the measure $\mu \defas \sum_{k = 1}^K \rho_k \mu_k$ \NNN and \BBB $\vphi \in W^{1, \infty}(TS)$ we have
  \begin{equation*}
    \int_S \vphi \vort(u_\eps)
    = \sum_{k = 1}^K \int_S \rho_k \vphi \vort(u_\eps)
    \to \sum_{k = 1}^K \sprod{\rho_k \mu_k}{\vphi}
    = \sprod{\mu}{\vphi},
  \end{equation*}
  where we have used $\rho_k \vphi \in W^{1, \infty}_0(\NNN O_k\BBB)$.
  Let us check that $\mu \in \tilde X^{(m)}$.
  For this purpose consider $x \in \spt(\mu_k) \cap \NNN O_l \BBB$ for $k, \, l \in \setof{1, \dots, K}$ with $k \neq l$.
  Take $\vphi \in \NNN C^\infty_c\BBB(B_r(x))$ with $\vphi(x) = 1$, where $r > 0$ is sufficiently small such that $B_r(x) \cap \spt(\mu_k) = \setof{x}$, $B_r(x) \cap \spt(\mu_l) \setminus \setof{x} = \emptyset$, and $B_r(x) \subset \NNN O_k \BBB \cap \NNN O_l \BBB$.
  As $\vphi \in W^{1, \infty}_0(\NNN O_k \BBB)$ we have
  \begin{equation*}
    \int_S \vphi \vort(u_\eps) = \int_{\NNN O_k \BBB} \vphi \vort(u_\eps) \to \vphi(x) \mu_k(x) = \mu_k(x).
  \end{equation*}
  As we also have $\vphi \in W^{1, \infty}_0(\NNN O_l \BBB)$ we can similarly derive that
  \begin{equation*}
    \int_S \vphi \vort(u_\eps) \to \mu_l(x),
  \end{equation*}
  and therefore $\mu_k(x) = \mu_l(x)$.
  By the arbitrariness of $x$ it follows that $(\mu_k - \mu_l)\mrestr{\NNN O_k \BBB \cap \NNN O_l \BBB} = 0$.
	\NNN By the arbitrariness of $k$ and $l$, \BBB in order to prove $\mu \in \tilde X^{(m)}\NNN(S)\BBB$ \CCC we \BBB only need to check that $\mu(S) = \echar$. 
  This follows by the \NNN $W^{-1, \infty}$-convergence \CCC of \BBB $(\vort(u_\eps))$ and Morse's index formula (see Theorem \ref{thm:morses_index_formula}) \CCC since \BBB
  \begin{equation*}
    \mu(S) = \lim_{\eps \to 0} \int_S 1 \vort(u_\eps) = \echar.
  \end{equation*}
\end{proof}

In order to improve the above compactness result we will need to employ harmonic vector \CCC fields. \BBB
\begin{definition}[Harmonic vector fields]
  \label{def:harmonic_vectorfields}
  Let $O \subset S$ be a coordinate neighborhood with Lipschitz boundary.
  We call a vector field $\NNN \tau\BBB \in C^\infty(O, \sph^1)$ \textit{harmonic} (on $O$) if and only if $\jac(\NNN \tau \BBB) = \di^* \Phi$ in $O$, where $\Phi$ is the $2$-form solving
	\begin{equation}
		\label{cond_phi}
		\left\{
			\begin{aligned}
				\Delta \Phi &= - \kappa \vol &&\text{in } O, \\
				\Phi &= 0 &&\text{on } \partial O,
			\end{aligned}
		\right.
	\end{equation}
	$\Delta = \di \di^*$ being the \textit{Laplace-Beltrami} operator.
\end{definition}

In the next lemma we show the existence of such vector fields.
\begin{lemma}[Existence of harmonic unit-length vector fields]
	\label{lem:existence_harmonic_frames}
	On any coordinate neighborhood $O \subset S$ with Lipschitz boundary, there exists a harmonic vector field $\NNN\tau\BBB$.
\end{lemma}

\begin{proof}
	\textit{\NNN Step 1\BBB:}
  \NNN Let \BBB $\Phi \in C^\infty(\Lambda^2\bar O)$ \NNN solve \BBB (\ref{cond_phi}).
	Fix an arbitrary point $x_0 \in O$ and a unit-length vector $\NNN \tau_0\BBB \in T_{x_0} O$.
	We then define $\NNN \tau\BBB$ at a point $x \in O$ as follows:
	Let $\gamma \colon [0, 1] \to U$ be a smooth curve with $\gamma(0) = x_0$ and $\gamma(1) = x$.
	Then by classic ODE theory there exists a smooth vector-field $X \in C^\infty(TO)$, such that
	\begin{equation}\label{ODE_for_X}
		\left\{
		\begin{aligned}
			\nabla_{\gamma'(s)} X(\gamma(s)) &= \di^*\Phi(\gamma'(s)) X^\perp(\gamma(s)) && \text{for  } s \in [0, 1], \\
			X(x_0) &= \NNN\tau_0\BBB.
		\end{aligned}
		\right.
	\end{equation}
  We set $\NNN\tau\BBB(x) \defas X(x)$.
  By construction for all $s \in [0, 1]$ it follows that
  \begin{equation*}
    \di_{\gamma'(s)} \abs{X(\gamma(s))}^2
    = 2 \sprod{\nabla_{\gamma'(s)} X(\gamma(s))}{X(\gamma(s))}
    = 2 \di^*\Phi(\gamma'(s)) \sprod{X^\perp(\gamma(s))}{X(\gamma(s))} = 0.
  \end{equation*}
  As $\abs{X(x_0)} = \abs{\NNN\tau_0\BBB} = 1$ this proves $\abs{X(\gamma(s))} = 1$ for all $s \in [0, 1]$.
  In particular, we see that $\abs{\NNN\tau\BBB(x)} = 1$.

	\textit{\NNN Step 2\BBB:}
  Let us now check that this definition of $\NNN\tau\BBB(x)$ does not depend on the path $\gamma$.
	\CCC To this end\BBB, consider another smooth curve $\CCC \zeta \BBB \colon [0, 1] \to O$ with $\CCC\zeta\BBB(0) = x_0$, $\CCC\zeta\BBB(1) = x$, and \NNN let \BBB $Y \in C^\infty(TO)$ \NNN be \BBB a solution of
	\begin{equation*}
		\left\{
		\begin{aligned}
			\nabla_{\CCC\zeta\BBB'(s)} Y(\CCC\zeta\BBB(s)) &= \di^*\Phi(\CCC\zeta\BBB'(s)) Y^\perp(\CCC\zeta\BBB(s)) && \text{for } s \in [0, 1], \\
			Y(x_0) &= \NNN\tau\BBB_0,
		\end{aligned}
		\right.
	\end{equation*}
  As, both, $X(1)$ and $Y(1)$ are of unit length it is sufficient to show that the angle between $X(1)$ and $Y(1)$ is a multiple of $2\pi$.
  \NNN Fixing \BBB $\NNN\eta\BBB \in C^\infty(O; \sph^1)$ we can find $\alpha, \, \beta \in C^\infty([0, 1])$ such that
  \begin{align*}
    X(\gamma(s)) &= e^{i\alpha(s)} \NNN\eta\BBB(\gamma(s)), & Y(\CCC\zeta\BBB(s)) &= e^{i\beta(s)} \NNN\eta\BBB(\CCC\zeta\BBB(s)).
  \end{align*}
	We then have for any $s \in [0, 1]$ that
	\begin{align*}
		\sprod{\nabla_{\gamma'(s)} X(\gamma(s))}{X^\perp(\gamma(s))}
		&= \sprod{i e^{i\alpha(s)} \alpha'(s) \NNN\eta\BBB(\gamma(s)) + e^{i\alpha(s)} \jac(\NNN\eta\BBB)(\gamma'(s)) \NNN i \eta\BBB(\gamma(s))}{e^{i\alpha(s)} \NNN i\eta\BBB(\gamma(s))} \\
		&= \alpha'(s) + \jac(\NNN\eta\BBB)(\gamma'(s)),
	\end{align*}
	and similarly
	\begin{equation*}
		\sprod{\nabla_{\CCC\zeta\BBB'(s)} Y(\CCC\zeta\BBB(s))}{Y^\perp(\CCC\zeta\BBB(s))} = \beta'(s) + \jac(\NNN\eta\BBB)(\CCC\zeta\BBB'(s)).
	\end{equation*}
	Let $\NNN\omega\BBB \colon [0, 1] \to O$ be the curve:
	\begin{equation*}
		\NNN\omega\BBB(s) \defas \begin{cases}
			\gamma(2s) &\text{if } s \in[0, \frac{1}{2}), \\
			\CCC\zeta\BBB(1-2s) &\text{if } s \in [\frac{1}{2}, 1].
		\end{cases}
	\end{equation*}
	By the simple connectedness of $O$ the curve $\eta$ is homologous to $0$.
	Therefore, there exists an integrable function $f \colon O \to \zz$ such that for any $1$-form $\NNN\vphi\BBB$ on $O$ we have (see Section 5.4 in \cite{jerrard_glman})
	\begin{equation*}
		\int_\eta \NNN\vphi\BBB = \int_O f \di \NNN\vphi\BBB.
	\end{equation*}
	Applying the above result for $\NNN\vphi\BBB = \di^* \Phi + A$ \NNN leads to \BBB
	\begin{align*}
    \alpha(1) - \beta(1)
    &= \int_0^1 \sprod{\nabla_{\gamma'\CCC(\BBB s \CCC)\BBB} X(s)}{\NNN i \BBB X(s)} \di s - \int_0^1 \sprod{\nabla_{\CCC\zeta\BBB'\CCC(\BBB s \CCC)\BBB} Y(s)}{\NNN i \BBB Y(s)} \di s - \int_\gamma \jac(\NNN\eta\BBB) + \int_\mu \jac(\NNN\eta\BBB) \\
    &= \int_\eta \di^*\Phi - \jac(\NNN\eta\BBB)
		= \int_O f (\Delta \Phi \NNN + \BBB \kappa \vol) = 0 \mod 2\pi.
	\end{align*}
  Consequently, $\NNN \tau \BBB$ is a well-defined unit-length tangent vector field on $O$.
	Its smoothness follows from the smoothness of $\Phi$.
	Lastly, the fact that $\jac(\NNN\tau\BBB) = \di^* \Phi$ \NNN is directly implied by \BBB(\ref{ODE_for_X}).
\end{proof}

\NNN By possibly decreasing $O$ we can from now on assume without loss of generality that any harmonic vector field we encounter is smooth up to the boundary. \BBB

\begin{lemma}\label{lem:mixed_term_convergence}
	Let $O$ be a coordinate neighborhood with \NNN smooth \BBB boundary, \CCC $\tau$ a \BBB harmonic \CCC unit-length \BBB vector field in $O$, and $(v_\eps) \subset \as^{\NNN(1)\BBB}(O)$ such that
	\begin{equation*}
		\vort(v_\eps) \weakto k \dirac_x \text{ weakly in } W^{-1, \infty}_0(O),
	\end{equation*}
	where $k \in \zz$, and $x \in O$.
	\CCC Let $\Phi$ be the $2$-form solving \eqref{cond_phi} and $\jac(\tau) = \di^* \Phi$, t\BBB hen,
	\begin{equation}
		\label{mixed_term_convergence}
		\int_O \sprod{\jac(v_\eps)}{\jac(\NNN\tau\BBB)} \vol \to k \, (\hodge\Phi)(x) + \int_O \abs{\NNN\jac(\tau)\BBB}^2 \vol\CCC.\BBB
	\end{equation}
\end{lemma}

\begin{proof}
  By (\ref{cond_phi}), integration by parts, and the definition of $\vort(v_\eps)$ we derive that
	\begin{align*}
		\int_O \sprod{\jac(v_\eps)}{\jac(\NNN\tau\BBB)} \vol
		= \int_O \sprod{\jac(v_\eps)}{\di^*\Phi} \vol
		&= \int_O \sprod{\vort(v_\eps) - \kappa \vol}{\Phi} \vol + \int_{\partial O} \hodge\Phi \jac(v_\eps) \\
		&= \int_O (\hodge \Phi) \vort(v_\eps) - \int_O \kappa \, \Phi
		\to k \, (\hodge\Phi)(x) - \int_O \kappa \, \Phi
	\end{align*}
	as $\eps \to 0$.
	\CCC To show \eqref{mixed_term_convergence} it is enough to \BBB rewrite the last integral \NNN on the right-hand side above \BBB as follows:
	\begin{equation*}
		\int_O - \kappa \, \Phi
		= \int_O \sprod{\di \di^*\Phi}{\Phi} \vol
		= \int_O |\di^*\Phi|^2 \vol = \int_O \abs{\NNN\jac(\tau)\BBB}^2 \vol\CCC.\BBB
	\end{equation*}
\end{proof}

\begin{lemma}\label{lem:apriori_Linfty_bound_Phi}
	For any $r_0 \in (0, r^*)$ there exists a constant $C > 0$ such that for all $x \in S$, $r \in (0, r_0]$, \NNN$2$\BBB-form $\Phi$ solving (\ref{cond_phi}) on $B_r(x)$\NNN, and the corresponding harmonic unit-length vector field $\tau$ \BBB it holds that
	\begin{equation}\label{apriori_Linfty_bound_Phi}
		\norm{\Phi}_{L^\infty} \leq C.
	\end{equation}
	and
	\begin{equation}\label{apriori_dirichlet_ener_bound}
		\int_{B_r(x)} \abs{\NNN\jac(\tau)\BBB}^2 \vol \leq C \int_S \abs{\kappa} \vol.
	\end{equation}
\end{lemma}
\begin{proof}
	The bound in (\ref{apriori_Linfty_bound_Phi}) follows by standard elliptic theory.
	Further\CCC more\BBB, \NNN by \BBB $\jac(\NNN \tau \BBB) = \di^*\Phi$ it \NNN holds \BBB that
	\begin{equation*}
		\int_{B_r(x)} \abs{\NNN\jac(\tau)\BBB}^2 \vol = \int_{B_r(x)} \abs{\di^*\Phi}^2 \vol = -\int_{B_r(x)} \kappa \Phi \leq C \int_S \abs{\kappa} \vol,
	\end{equation*}
	which is (\ref{apriori_dirichlet_ener_bound}).
\end{proof}

\begin{lemma}[Localized $\Gamma$-liminf inequality]
	\label{lem:localized_liminf_frac}
	Let $(u_\eps) \subset \asm\NNN(S)\BBB$ be a bounded sequence in $L^\infty(TS)$ such that
	\begin{gather}
		\gl(u_\eps) \leq C \leps \text{ for all } \eps > 0 \quad \text{\CCC and \BBB} \label{localized_liminf_ener_bound} \\
		\vort(u_\eps) \weakto \mu \defas \sum_{k = 1}^K \frac{d_k}{m} \dirac_{x_k} \in \tilde X^{(m)}\NNN(S)\BBB \text{ weakly in } W^{-1, \infty}(S)\NNN.\BBB \label{localized_liminf_conv_vortueps}
	\end{gather}
	Further\CCC more\BBB, let $r_0 \in (0, r^*)$ be small enough such that the balls $\setof{B_{r_0}(x_k)}$ are disjoint.
	Then, there exist $C \in \rr$ such that for every $k \in \setof{1, \dots, K}$ and $r \in (0, r_0]$ it holds that
	\begin{equation}
		\label{localized_liminf_frac}
		\liminf_{\eps \to 0} \left(\gl(u_\eps, B_r(x_k)) - \frac{\pi \abs{d_k}}{m^2} \log\CCC\left(\BBB\frac{r}{\eps}\CCC\right)\BBB  \right) \geq C.
	\end{equation}
\end{lemma}

\begin{proof}
	Fix $k$ and $r \in (0, r_0]$ and shortly write $B \defas B_r(x_k)$.
	Furthermore, let $v_\eps \defas \power(u_\eps)$, where $\NNN\tau\BBB$ is a harmonic \CCC unit-length \BBB vector field on $B$.	
  \CCC B\BBB y (\ref{absgrad_after_doubling}), the boundedness of $(u_\eps)$ in $L^\infty$, and (\ref{apriori_dirichlet_ener_bound}) we derive that
	\begin{align}
		\gl(u_\eps, B)
		&\geq 
			\frac{1}{4\eps^2} \int_B \CCC (1 - \abs{v_\eps}^2)^2 \BBB \vol
			\CCC + \BBB \frac{1}{2m^2} \int_B \abs{\nabla{v_\eps}}^2 \vol
			+ \frac{m^2 - 1}{2m^2} \int_B \abs{\di\abs{u_\eps}}^2 \vol \nonumber \\
		&\phantom{=}\quad
			- \frac{(m-1)^2}{2 m^2} \int_B \abs{u_\eps}^2 \abs{\NNN\jac(\tau)\BBB}^2 \vol 
			- \frac{m-1}{m} \int_B \sprod{\jac(u_\eps)}{\jac(\NNN\tau\BBB)} \vol \nonumber \\
		&\geq 
		\frac{1}{m^2} \gl(v_\eps, B)
			- \norm{u_\eps}_{L^\infty}^2 \int_B \abs{\NNN \jac(\tau)\BBB}^2 \vol
			- \frac{m-1}{m} \int_B \sprod{\jac(u_\eps)}{\jac(\NNN\tau\BBB)} \vol \nonumber \\
		&\geq
			\frac{1}{m^2} \gl(v_\eps, B)
			\NNN -\BBB C
			- \frac{m-1}{m} \int_B \sprod{\jac(u_\eps)}{\jac(\NNN \tau \BBB)} \vol \label{locliminf_gl_lower_bound}
	\end{align}
	for \CCC some \BBB constant $C$ independent of $\eps$ and $r$.
	By (\ref{prejac_after_doubling}) and (\ref{apriori_dirichlet_ener_bound}) we also have that
	\begin{align}
		- \frac{m-1}{m} \int_B \sprod{\jac(u_\eps)}{\jac(\NNN\tau\BBB)} \vol
		&=
		- \frac{m-1}{m^2} \int_B \sprod{\jac(v_\eps)}{\jac(\NNN\tau\BBB)} \vol
		- \frac{(m-1)^2}{m^2} \int_B \abs{u_\eps}^2 \abs{\NNN\jac(\tau)\BBB}^2 \vol \nonumber \\
		&\geq - \frac{m-1}{m^2} \int_B \sprod{\jac(v_\eps)}{\jac(\NNN\tau\BBB)} \vol \NNN - \BBB C. \label{locliminf_mixed_term_lower_bound}
	\end{align}
	\CCC N\BBB ote that by (\ref{localized_liminf_ener_bound}), (\ref{localized_liminf_conv_vortueps}), Lemma \ref{lem:absveps_e_comp}, and (\ref{vorticity_after_doubling})
	\begin{equation*}
		\vort(v_\eps) = m \vort(u_\eps) - (m - 1) \vort(\abs{v_\eps} \NNN \tau \BBB) \weakto d_k \dirac_{x_k}
	\end{equation*}
	weakly in $W^{-1, \infty}_0(B)$ and hence by Lemma \ref{lem:mixed_term_convergence}, (\ref{locliminf_mixed_term_lower_bound})\NNN, \eqref{mixed_term_convergence}, \BBB (\ref{apriori_Linfty_bound_Phi}), and (\ref{apriori_dirichlet_ener_bound}) it follows that
	\begin{equation}\label{locliminf_mixed_term_liminf}
		\liminf_{\eps \to 0} - \frac{m-1}{m} \int_B \sprod{\jac(u_\eps)}{\jac(\NNN\tau\BBB)} \vol
		\geq - \frac{m-1}{m^2} d_k \, (\hodge \Phi)(x_k) - \frac{m-1}{m^2} \int_B \abs{\NNN \jac(\tau)\BBB}^2 \NNN -\BBB C \geq \NNN - \BBB C.
	\end{equation}
	\CCC Recalling Proposition 4(ii) in \BBB \cite{canevari_xyman} \CCC stating that under our assumptions \BBB
	\begin{equation*}
		\liminf_{\eps \to 0} \left(\gl(v_\eps, B_r(x_k)) - \pi \abs{d_k} \log \left(\frac{r}{\eps}\right) \right) \geq \NNN - \BBB C,
	\end{equation*}
	for a constant $C$ independent of \NNN$r$\CCC, \BBB by (\ref{locliminf_gl_lower_bound}) and (\ref{locliminf_mixed_term_liminf}) we \CCC eventually obtain \BBB (\ref{localized_liminf_frac}).
\end{proof}

\begin{proposition}[Improved \CCC vorticity \BBB compactness]
	\label{prop:improved_vortex_comp}
  Let $(u_\eps) \subset \asm$ be a bounded sequence in $L^\infty(TS)$ such that for all $\NNN\eps\BBB$
	\begin{equation}
		\label{precise_energy_bound}
		\gl(u_\eps) \leq \frac{N}{m} \pi \leps + C,
	\end{equation}
	where $N \in \nn$ \NNN and \BBB assume \BBB that $\vort(u_\eps) \weakto \mu \in \tilde X^{(m)}\NNN(S)\BBB$ weakly in $W^{-1, \infty}(S)$.
	Then, $\abs{\mu} \leq N$\CCC. Moreover, \BBB if $\abs{\mu} = N$\CCC, then \BBB $\mu \in X^{(m)}\NNN(S)\BBB$.
\end{proposition}

\begin{proof}
	\textit{\NNN Step 1 (Upper bound on $\abs{\vort(u)}$)\BBB:}
	\NNN Let us write $\mu = \sum_{k = 1}^K \frac{d_k}{m} \dirac_{x_k}$. \BBB
	Further\CCC more\BBB, let $r_0 \in (0, r^*)$ be chosen sufficiently \CCC small that \BBB the balls $\setof{B_{r_0}(x_k)}$ are disjoint.
	We will shortly write $B_k$ for the ball $B_{r_0}(x_k)$.
	By (\ref{precise_energy_bound}) and Lemma \ref{lem:localized_liminf_frac} there exist \NNN$\tilde C$\BBB and $\eps_0 > 0$ such that for all $\eps \in (0, \eps_0)$
	\begin{equation*}
		\frac{\NNN \pi \BBB N}{m} \leps + C
		\geq \sum_{k = 1}^K \gl(u_\eps, B_k) 
		\geq \sum_{k = 1}^K \frac{\NNN \pi \BBB\abs{d_k}}{m^2} \log\left(\frac{r_0}{\eps}\right) \NNN - \BBB K \tilde C
		\geq \frac{\NNN \pi \BBB\abs{\mu}}{m} \leps \NNN - \BBB K \tilde C - \frac{\NNN\pi \BBB\abs{\mu}}{m} \abs{\log r_0}.
	\end{equation*}
	Taking $\eps \to 0$, this can only hold true if $\abs{\mu} \leq N$.

	\textit{\NNN Step 2 (Case $\abs{\vort(u)} = N$) \BBB:}
	Let us assume that $\abs{\mu} = N$.
	We will show that this implies $\abs{d_k} = 1$ for all $k$ and hence $\mu \in X^{(m)}\NNN(S)\BBB$.
	For fixed $k$ let $v_\eps \defas \power(u_\eps)$ in $B_k$ for a harmonic \CCC unit-length \BBB vector field $\NNN\tau\BBB$ on $B_k$.
	\CCC From \BBB (\ref{absgrad_after_doubling}), the boundedness of $(u_\eps)$ in $L^\infty$, and (\ref{apriori_dirichlet_ener_bound}) we derive that
	\begin{align}
		\gl(v_\eps, B_k) 
		&= 
			\frac{1}{4\eps^2} \int_{B_k} \NNN(1-\abs{u_\eps}^2)^2\BBB \vol 
			+ \frac{m^2}{2} \int_{B_k} \abs{\nabla u_\eps}^2 \vol
			+ \frac{1-m^2}{2} \int_{B_k} \abs{\di {\abs{u_\eps}}}^2 \vol \nonumber \\ 
		&\quad + \frac{(m-1)^2}{2} \int_{B_k} \abs{u_\eps}^2 \abs{\NNN\jac(\tau)\BBB}^2 \vol
			- m(m-1) \int_{B_k} \sprod{\jac(u_\eps)}{\jac(\NNN\tau\BBB)} \vol \nonumber \\ 	
		&\begin{aligned}
			&\leq m^2 \gl(u_\eps, B_k) -m(m-1) \int_{B_k} \sprod{\jac(u_\eps)}{\jac(\NNN\tau\BBB)} \vol + C.
		\end{aligned}\label{impr_vortcomp_glveps_upper_bound}
	\end{align}
	Further\CCC more\BBB, using (\ref{prejac_after_doubling}), we can write
	\begin{align*}
		-m(m-1) \int_{B_k} \sprod{\jac(u_\eps)}{\jac(\NNN\tau\BBB)} \vol
		&= 
			-(m-1) \int_{B_k} \sprod{\jac(v_\eps)}{\jac(\NNN\tau\BBB)} \vol
			- (m-1)^2 \int_{B_k} \abs{u_\eps}^2 \abs{\NNN\jac(\tau)\BBB}^2 \vol \\
		&\leq -(m-1) \int_{B_k} \sprod{\jac(v_\eps)}{\jac(\NNN\tau\BBB)} \vol.
	\end{align*}
	\NNN Hence, b\BBB y Lemma \ref{lem:mixed_term_convergence} it holds that
	\begin{equation*}
		\limsup_{\eps \to 0} -m(m-1) \int_{B_k} \sprod{\jac(u_\eps)}{\jac(\NNN\tau\BBB)} \vol
		\leq -(m-1) \left(\hodge \Phi(x_k) + \int_{B_{\CCC k\BBB}} \abs{\NNN\jac(\tau)\BBB}^2 \vol\right).
	\end{equation*}
	\NNN B\BBB y (\ref{impr_vortcomp_glveps_upper_bound}) and Lemma \ref{lem:apriori_Linfty_bound_Phi} we have that
	\begin{equation}\label{impr_vortcomp_glveps_upper_bound_v2}
		\gl(v_\eps, B_k) \leq m^2 \gl(u_\eps, B_k) + C.
	\end{equation}
	As the balls $\setof{B_l}$ are disjoint we derive \NNN using \BBB $\abs{\mu} = N$ and Lemma \ref{lem:localized_liminf_frac} the following upper bound on $\gl(u_\eps, B_k)$ \CCC \NNN for \BBB $\eps$ small enough:\BBB
	\begin{align}
		\gl(u_\eps, B_k) 
		&\leq \gl(u_\eps) - \sum_{l \neq k} \gl(u_\eps, B_l) \nonumber \\
		&\leq \NNN\pi\BBB \frac{N}{m} \leps - \pi \sum_{l \neq k} \frac{\abs{d_l}}{m^2}\log\left(\frac{r_0}{\eps}\right) + C \nonumber \\
		&\leq \NNN\pi\BBB \frac{N}{m} \leps - \pi \frac{\abs{\mu} - \frac{\abs{d_k}}{m}}{m} \leps + C
		= \pi \frac{\abs{d_k}}{m^2} \leps + C. \label{GLueps_Bk_bound}
	\end{align}
	Using this bound in (\ref{impr_vortcomp_glveps_upper_bound_v2}) we see that
	\begin{equation}\label{local_glveps_prop5}
		\gl(v_\eps) \leq \pi \abs{d_k} \leps + C.
	\end{equation}
	Following the by now standard arguments in the proof of Theorem B (i) in \cite{canevari_xyman} (see also Theorem 5.3 (i) in \cite{marcello_gl}) this bound can only hold true for \CCC $\eps \to 0$ \BBB if $\abs{d_k} = 1$, as desired.
\end{proof}

From now on we define $A_{r, r'}(x) \defas B_{r'}(x) \setminus B_r(x)$ for any $x \in S$ and $0 < r < r' < r^*$.
\begin{lemma}[Well-definedness of the renormalized energy]
	\label{lem:ren_ener_welldefined}
	Let $u \in \lsm\CCC(S)\BBB$ \CCC be such that \BBB $\vort(u) = \sum_{k = 1}^K \frac{d_k}{m} \dirac_{x_k} \in X^{(m)}$\NNN. \BBB
  Then, the limit in (\ref{def_ren_energy}) exists and belongs to $(-\infty, \infty]$.
	Moreover, \CCC for \BBB $r_0 \in (0, r^*)$ sufficiently \CCC small that \BBB the balls $\setof{B_{r_0}(x_k)}_k$ are disjoint we have that
	\begin{align*}
		\renm(u) 
		&= 
			\frac{1}{2} \int_{S_{r_0}(\vort(u))} \abs{\nabla u}^2 \vol 
			+ \sum_{k = 1}^{\CCC K\BBB} \frac{1}{m^2} \ren(v^{(k)}, B_{r_0}(x_k))
			- \sum_{k = 1}^{\CCC K \BBB} \frac{(m - 1)^2}{2 m^2} \int_{B_{r_0}(x_k)} \abs{\NNN\jac(\tau^{(k)})\BBB}^2 \vol \\
		&\phantom{=}\quad 
			+ \sum_{k = 1}^{\CCC K \BBB} \frac{m - 1}{m} \int_{B_{r_0}(x_k)} \sprod{\jac(v^{(k)})}{\jac(\NNN\tau^{(k)}\BBB)} \vol,
	\end{align*}
  \CCC where \BBB $\NNN\tau^{(k)}\BBB \in C^\infty(\bar B; \sph^1)$, $v^{(k)} \defas \NNN p_{\tau^{(k)}}(u)\BBB$, and
	\begin{equation}\label{def_renergy_nonfrac}
		\ren(v^{(k)}, B_{r_0}(x_k))
		=
		\lim_{r \to 0} \left( \frac{1}{2} \int_{A_{r, r_0}(x_k)} \abs{\nabla v^{(k)}}^2 \vol - \pi \abs{\log r} \right).
	\end{equation}
\end{lemma}

\begin{proof}
	For any $0 < r < r_0$ we can write
	\begin{align*}
		\frac{1}{2} \int_{S_r(\vort(u))} \abs{\nabla{u}}^2 \vol - \NNN\pi\BBB\frac{\CCC \abs{\vort(u)}(S) \BBB}{m} \abs{\log r}
		&\CCC = \frac{1}{2} \int_{S_r(\vort(u))} \abs{\nabla{u}}^2 \vol - \NNN \pi \BBB \frac{K}{m^2} \abs{\log r} \BBB \\
		&= \frac{1}{2} \int_{S_{r_0}(\vort(u))} \abs{\nabla{u}}^2 \vol
			+ \sum_{k = 1}^{\CCC K \BBB} \left( \frac{1}{2} \int_{A_{r, {r_0}}(x_k)} \abs{\nabla{u}}^2 \vol - \frac{\pi}{m^2} \abs{\log r} \right).
	\end{align*}
	Therefore, \CCC to prove the existence of the limit in \eqref{def_ren_energy}\NNN, \BBB it is enough to show \CCC that \BBB for every $k$
	\begin{equation}\label{local_ren_energy_fractional}
		\lim_{r \to 0} \frac{1}{2} \int_{A_{r, {r_0}}(x_k)} \abs{\nabla{u}}^2 \vol - \frac{\pi}{m^2} \abs{\log r} \in (-\infty, \infty].
	\end{equation}
	Let us fix $k$ and \CCC set \BBB $B \defas B_{r_0}(x_k)$ \CCC and \BBB $A_r \defas A_{r, r_0}(x_k)$.
	Furthermore, let $\NNN\tau\BBB \in C^\infty(\bar B; \sph^1)$ and $v \defas \power(u)$ in $B$.
	Then, by \NNN\eqref{absgrad_after_doubling} and \BBB $\abs{u} = 1$ a.e.~\CCC in $S$ we \CCC get \BBB that
	\begin{align*}
		\frac{1}{2} \int_{A_r} \abs{\nabla{u}}^2 \vol - \frac{\pi}{m^2} \abs{\log r}
		&= 
			\frac{1}{m^2} \left(\frac{1}{2} \int_{A_r} \abs{\nabla{v}}^2 \vol - \pi \abs{\log r} \right)
			- \frac{(m-1)^2}{2m^2} \int_{A_r} \abs{\NNN\jac(\tau)\BBB}^2 \vol \\
		&\phantom{=}\quad
			+ \frac{m-1}{m} \int_{A_r} \sprod{\jac(v)}{\jac(\NNN\tau\BBB)} \vol.
	\end{align*}
	\NNN S\CCC ince $u \in \lsm(S)$ is such that $\nabla u \in L^1(TS)$, by \BBB (\ref{absgrad_after_doubling}\NNN)\BBB we have that $\nabla v \in L^1(TB)$, which together with the smoothness of \NNN$\tau$ leads \BBB to
	\begin{align*}
		&\lim_{r \to 0} \left(
			\frac{m-1}{m} \int_{A_r} \sprod{\jac(v)}{\jac(\NNN\tau\BBB)} \vol
			- \frac{(m-1)^2}{2m^2} \int_{A_r} \abs{\NNN\jac(\tau)\BBB}^2 \vol 
		\right) \\
		&\quad=
			\frac{m-1}{m} \int_B \sprod{\jac(v)}{\jac(\NNN\tau\BBB)} \vol
			- \frac{(m-1)^2}{2m^2} \int_B \abs{\NNN\jac(\tau)\BBB}^2 \vol.
	\end{align*}
	\CCC Finally, \eqref{local_ren_energy_fractional} follows thanks to
	\begin{equation*}
		\lim_{r \to 0} \left(\frac{1}{2} \int_{A_r} \abs{\nabla{v}}^2 \vol - \pi \abs{\log r} \right) \in (-\infty, \infty]
	\end{equation*}
	whose proof can be found in \cite{canevari_xyman} (see Subsection 6.1).
	\BBB
\end{proof}

\begin{lemma}[Localization of a unit vortex]
	\label{lem:ballcon_spec}
  Let $O \subset \CCC S \BBB$ \CCC be \BBB an open set and $(v_\eps) \subset C^\infty(TO)$ \CCC a \BBB bounded sequence in $L^\infty(TO)$ such that for all \NNN$\eps$\BBB:
  \begin{equation*}
    \gl(v_\eps, O) \leq \pi \leps + C.
  \end{equation*}
	Furthermore, assume that for $x_0 \in O$ and $d \in \setof{-1, 1}$
  \begin{equation*}
    \vort(v_\eps) \weakto d \dirac_{x_0} \text{ weakly in } W^{-1, \infty}_0(O).
  \end{equation*}
  Then, there exists $\eps_0 > 0$ and $\tilde C > 0$ such that for all $\eps \in (0, \eps_0)$ we can find $x_\eps \in O$ such that
  \begin{enumerate}[label=(\roman*)]
    \item $\lim_{\eps \to 0} x_\eps = x_0$; 
    \item $\deg(v_\eps, \partial B_r(x_\eps)) = d$ \CCC for \BBB $r > 0$ such that $\partial B_r(x_\eps) \subset O$ and $\abs{v_\eps} \geq \frac{1}{2}$ on $\partial B_r(x_\eps)$;
    \item $\gl(v_\eps, B_{r_\eps}(x_\eps)) \geq \pi \log(\frac{r_\eps}{\eps}) - \tilde C$, where $r_\eps = 2\eps^{\frac{1}{3}}$.
  \end{enumerate}
\end{lemma}

\begin{proof}
	We start by applying Theorem \ref{thm:ballcon} with $T = 1$, $n = 1$, and $q = \frac{1}{3}$.
	Hence, taking $\eps_0$ small enough, we can find for each $\eps \in (0, \eps_0)$ a finite family of closed geodesic balls $\balls^{(\eps)} = \setof{B_k^{(\eps)}}_{k = 1}^{K_\eps}$ satisfying condition \NNN (i)-(iv) \BBB from Theorem \ref{thm:ballcon}.
	Let us now define the measure
	\begin{equation*}
		\nu_\eps \defas \sum_k d_k^{(\eps)} \dirac_{x_k^{(\eps)}},
	\end{equation*}
	where $x_k^{(\eps)}$ \CCC is the center \BBB of $B_k^{(\eps)}$ and $d_k^{(\eps)} \defas \deg(v_\eps, \partial B_k^{(\eps)})$.
	By Theorem \ref{thm:ballcon} \NNN (ii) \BBB we have that $\abs{\nu_\eps} \leq 1$.
	Hence, up to taking a subsequence, $\nu_\eps \weakstarto \mu$ weakly* in $\radon(O)$.
	\CCC Arguing \BBB as in the proof of Lemma \ref{lem:initial_comp} \CCC the \BBB whole sequence $\CCC(\BBB \nu_\eps \CCC)\BBB$ weakly* converges towards $\mu = d \dirac_{x_0}$.
	By possibly decreasing $\eps_0$ this assures that for any $\eps \in (0, \eps_0)$ we have that $B_{r_\eps}(x_\eps) \subset O$ and we can find a unique $k_\eps \in \setof{1, \dots, K_\eps}$ such that $d_{k_\eps}^{(\eps)} = 1$ while $d_k^{(\eps)} = 0$ for all $k \in \setof{1, \dots, K_\eps} \setminus \setof{k_\eps}$.
	With Theorem \ref{thm:ballcon} \NNN (i) \BBB this shows (ii) in the statement.
	Furthermore, by Theorem \ref{thm:ballcon} \NNN (iv) \BBB we see that
	\begin{equation*}
		\gl(v_\eps, B_{r_\eps}(x_\eps)) \geq \pi \log\frac{r_\eps}{2\eps} - \tilde C = \pi \log\frac{r_\eps}{\eps} - \pi \log(2) - \tilde C,
	\end{equation*}
	which shows (iii) \CCC in \BBB the statement\NNN, after redefining $\tilde C$\BBB.
\end{proof}

\begin{lemma}
	\label{lem:comp_grad}
  Let $(u_\eps) \subset \asm\CCC(S)\BBB$ be a bounded sequence in $L^\infty(TS)$ such that for all $\NNN\eps\BBB$
	\begin{equation}\label{comp_grad_ener_bound}
		\gl(u_\eps) \leq \frac{N}{m} \pi \leps + C
	\end{equation}
	for constants $N \in \nn, \, C > 0$ independent of $\eps$.
  Suppose that 
	\begin{equation*}
		\vort(u_\eps) \weakto \mu \defas \sum_{k = 1}^{m N} \frac{d_k}{m} \dirac_{x_k} \in X^{(m)} \text{ weakly in } W^{-1, \infty}(S)\NNN.\BBB
	\end{equation*}
	Furthermore, set $r_0 \in (0, r^*)$ to be sufficiently small \CCC that \BBB the balls $\setof{B_{r_0}(x_k)}_{\NNN k \BBB}$ are disjoint.
  Then, for any $p \in [1, 2)$ and $r \in (0, r_0)$ it holds that
	\begin{align}
		\sup_\eps \left(\norm{\nabla u_\eps}_{L^2(T^*S_r(\mu) \otimes TS_r(\mu))} + \haus^1_g(\NNN \mathcal{J}_{u_\eps}\BBB \cap S_r(\mu)\right) &< \infty, \label{sbv2_loc_bound} \\
		\sup_\eps \norm{\nabla u_\eps}_{L^p(T^*S \otimes TS)} &< \infty. \label{sbvp_bound}
	\end{align}
\end{lemma}

\begin{proof}
	\textit{\NNN Step 1 ($L^2$-bound outside vortices)\BBB:}
	\CCC Let \BBB $r \in (0, r_0)$.
	By Lemma \ref{lem:localized_liminf_frac} and the energy bound in (\ref{comp_grad_ener_bound}) we \CCC have \BBB that \CCC for $\eps$ small enough \BBB
	\begin{align*}
		\gl(u_\eps, S_r(\mu)) 
		&\leq \gl(u_\eps) - \sum_{k = 1}^{m N} \gl(u_\eps, B_r(x_k)) \\
		&\leq \frac{N}{m} \pi \leps - \sum_{k = 1}^{m N} \frac{\pi}{m^2} \log \left(\frac{r}{\eps}\right) + C
		\leq \frac{N}{m} \pi \abs{\log r} + C \NNN\leq C \abs{\log r}\BBB.
	\end{align*}
  By the definition of $\gl$ this leads to (\ref{sbv2_loc_bound}).

	\textit{\NNN Step 2 (P\BBB artitioning of $B_{r_0}(x_k)$):}
	For (\ref{sbvp_bound}) it now suffices to prove for all $k$ that
	\begin{equation}\label{sbvp_loc_bound}
		\sup_\eps \int_{B_{r_0}(x_k)} \abs{\nabla u_\eps}^p \vol < \infty.
	\end{equation}
	For this purpose, \CCC let us \BBB fix $k$ and \NNN shortly write \BBB $B \defas B_{r_0}(x_k)$.
	Furthermore, let \NNN$\tau$ \BBB be \CCC a \BBB harmonic \CCC unit-length vector field \BBB on $B$ and $v_\eps \defas \power(u_\eps)$ in $B$.
	By \NNN (\ref{grad_after_doubling}), (\ref{jump_after_doubling})\BBB, and the definition of $\asm(S)$, we have that $v_\eps \in W^{1, 2}(TB)$.
	Our goal now is to show that \CCC the energy bound on $(v_\eps)$ and the assumption $\abs{d_k} = 1$ gives \BBB
	\begin{equation}\label{sobolev_loc_bound}
		\sup_\eps \int_B \abs{\nabla v_\eps}^p \vol < \infty.
	\end{equation}
	The above estimate will be achieved by partitioning $B$ and eventually estimating the $L^p$-norm of $\nabla v_\eps$ on each component of the partition, separately.
	Given $\tilde r_0 \defas \frac{3}{4}r_0$ we define $J_\eps \in \nn$ as the largest natural number satisfying $2^{-J_\eps} \tilde r_0 > 2\eps^{\frac{1}{3}}$.
	By this choice we have that
	\begin{equation}\label{comp_grad_Jeps_bound}
		J_\eps \leq \NNN - \BBB 1 + \frac{1}{\log(2)} \left( \frac{1}{3} \leps - \log \tilde r_0 \right) \NNN \leq C \leps\BBB.
	\end{equation}
	\NNN A\BBB s in the proof of Proposition (\ref{prop:improved_vortex_comp}) we can show that
	\begin{equation}\label{comp_grad_bound_gl_veps}
		\gl(v_\eps\NNN, B\BBB) \leq \pi \leps + C.
	\end{equation}
	By the weak convergence of $(\vort(u_\eps))$, (\ref{vorticity_after_doubling}), and Lemma \ref{lem:absveps_e_comp} it follows that
	\begin{equation}\label{vortveps_conv}
		\vort(v_\eps) \weakto d_k \dirac_{x_k} \text{ weakly in } W^{-1, \infty}_0(B).
	\end{equation}
  \CCC Since\BBB, using a standard approximation argument, we can assume that $v_\eps$ is smooth. 
  We are now in a position to apply Lemma \ref{lem:ballcon_spec} for the sequence $(v_\eps)$ and \CCC for \BBB $O = B$.
	Let $(x_k^{(\eps)})_\eps$ denote the resulting sequence of vortex centers.
	For $j \NNN \in \setof{\BBB 0, \ldots, J_\eps\NNN}\BBB$ we then set
	\begin{equation*}
		B_j^{(\eps)} \defas B_{2^{-j}\tilde r_0}(x_k^{(\eps)})
	\end{equation*}
	\CCC and \BBB for $j \NNN \in \setof{\BBB 0, \ldots, J_\eps-1\NNN}\BBB$
	\begin{equation*}
		A_j^{(\eps)} = B_j^{(\eps)} \setminus B_{j+1}^{(\eps)}.
	\end{equation*}
	With the above notation we partition $B$ as follows:
	\begin{equation*}
		B = B_{J_\eps}^{(\eps)} \cup \bigcup_{j = 0}^{J_\eps - 1} A_j^{(\eps)} \cup (B \setminus B_0^{(\eps)}).
	\end{equation*}

	\textit{\NNN Step 3 (B\BBB ound on $B_{J_\eps}^{(\eps)}$):}
  We \CCC estimate \BBB the $L^p$-norm of $\nabla v_\eps$ on the set $B_{J_\eps}^{(\eps)}$.
	By the maximality of $J_\eps$ we must have $2^{-(J_\eps+1)} \tilde r_0 \leq 2\eps^{\frac{1}{3}}$, and therefore \NNN$\haus^2_g(B_{J_\eps}^{(\eps)}) \leq C \eps^{\frac{2}{3}}$.
	\CCC H\BBB ölder's inequality and the energy bound (\ref{comp_grad_ener_bound}) \NNN then \BBB \CCC lead \BBB to
	\begin{equation*}
		\int_{B_{J_\eps}^{(\eps)}} \abs{\nabla v_\eps}^p \vol
		\leq \haus^2_g(B_{J_\eps}^{(\eps)})^{\frac{2}{2 - p}} \int_B \abs{\nabla v_\eps}^2 \vol
		\leq C\eps^{\frac{4}{3(2 - p)}} \leps = \NNN\littleo\BBB(1).
	\end{equation*}

	\textit{\NNN Step 4 (B\BBB ound on $B \setminus B_0^{(\eps)})$:}
	By Lemma \ref{lem:ballcon_spec} \NNN (i) \BBB we have $\abs{x_k - x_k^{(\eps)}} \leq \frac{1}{4} r_0$ for sufficiently small $\eps$.
	\NNN Then, by o\BBB ur choice of $\tilde r_0$ for all such $\eps$ \NNN it holds \BBB that $B_{\frac{r_0}{2}}(x_k) \subset B_0^{(\eps)} \subset B$.
	\NNN B\BBB y Hölder's inequality and (\ref{sbv2_loc_bound}) \NNN it follows that \BBB
	\begin{equation*}
		\int_{B \setminus B_0^{(\eps)}} \abs{\nabla v_\eps}^p \vol
    \leq \int_{B \setminus B_{\frac{r_0}{2}}(x_k)} \abs{\nabla v_\eps}^p \vol
		\leq \haus^2_g(B)^{\frac{2}{2-p}} \int_{B \setminus B_{\frac{r_0}{2}}(x_k)} \abs{\nabla v_\eps}^2 \vol = \BigO(1).
	\end{equation*}

	\textit{\NNN Step 5 (B\BBB ound on $\bigcup_{j = 0}^{J_\eps-1} A_j^{(\eps)}$):}
	Using (\ref{comp_grad_bound_gl_veps})\CCC, \BBB Lemma \ref{lem:ballcon_spec} \NNN (iii) \CCC, and the maximality of $J_\eps$ \BBB we have that
	\begin{align}
		\gl(v_\eps, B \setminus B_{J_\eps}^{(\eps)})
		&= \gl(v_\eps, B) - \gl(v_\eps, B_{J_\eps}^{(\eps)}) \nonumber \\
		&\leq \pi \leps - \pi \log\CCC\left(\BBB\frac{2\eps^{\frac{1}{3}}}{\eps}\CCC\right)\BBB + C
		\leq -\pi \log(2^{-(J_\eps + 1)} \tilde r_0) + C
		\leq \NNN \pi \BBB J_\eps \log(2) + C. \label{comp_grad_bound_outside_VJeps}
	\end{align}
	Furthermore, by Lemma \ref{lem:ballcon_spec} \NNN (ii) \BBB, (\ref{ener_lowbound_circles}), and $2^{-j} \tilde r_0 \geq 2^{-J_\eps} \tilde r_0 \geq 2 \eps^{\frac{1}{3}}$ \NNN for any $j \in \setof{0, \ldots, J_\eps - 1}$ \BBB we can estimate
	\begin{align*}
		\gl(v_\eps, A_j^{(\eps)})
		&\geq \int^{2^{-j}\tilde r_0}_{2^{-(j+1)}\tilde r_0} \frac{\pi(1 - Cr^2)}{r + C \eps} \di r \\
		&\NNN \geq \int^{2^{-j}\tilde r_0}_{2^{-(j+1)}\tilde r_0} \frac{\pi}{r + C \eps} - Cr \di r \BBB \\
		&\NNN \geq \pi \log \left( \frac{2^{-j} \tilde r_0 + C\eps}{2^{-(j+1)}\tilde r_0 + C \eps} \right) - C 2^{-2j} \tilde r_0^2 \BBB \\
		&\geq \pi \log(2) - \pi \log\left(1 + C \frac{\eps}{2^{-j} \tilde r_0}\right) - C 2^{-2j} \tilde r_0^2
		\geq \pi \log(2) - C(\eps^{\frac{2}{3}} + 2^{-2j} \tilde r_0^2)
	\end{align*}
	Fix $j' \in \setof{0, \dots, J_\eps - 1}$.
	We are now able to prove boundedness of $\gl(v_\eps, A_{j'}^{(\eps)})$ independently of $\eps$.
	\NNN B\BBB y (\ref{comp_grad_bound_outside_VJeps}) and \NNN the above estimate \BBB i\NNN t \BBB follows that
	\begin{align*}
		(J_\eps - 1) \pi \log(2) - C \Big(\CCC (J_\eps - 1) \BBB \eps^{\frac{2}{3}} + \tilde r_0^2 \sum_{j \neq j'} 2^{-2j}\Big)
		&\leq \sum_{j \neq j^*} \gl(v_\eps, A_j^{(\eps)}) \\
		&\leq \gl(v_\eps, B \setminus B_{J_\eps}^{(\eps)}) - \gl(v_\eps, A_{j'}^{(\eps)}) \\
		&\leq J_\eps \pi \log(2) + C - \gl(v_\eps, A_{j'}^{(\eps)}).
	\end{align*}
	Solving for $\gl(v_\eps, A_{j'}^{(\eps)})$ and using (\ref{comp_grad_Jeps_bound}) then leads t\NNN o\BBB
	\begin{align*}
		\gl(v_\eps, A_{j'}^{(\eps)})
		&\leq \pi \log(2) 
			+ C\Big(1 + J_\eps \eps^{\frac{2}{3}} + \tilde r_0^2 \sum_{j \neq j'} 2^{-2j}\Big) \\
		&\leq \pi \log(2) + C \Big(1 + \eps^{\frac{2}{3}} \leps + \tilde r_0^2 \sum_{j = 0}^\infty 2^{-2j}\Big)
		= \BigO(1).
	\end{align*}
	Consequently, by the arbitrariness of $j'$ and Hölder's inequality we derive for $p \in [1, 2)$ that
	\begin{align*}
		\sum_{j = 0}^{J_\eps - 1} \int_{A_j^{(\eps)}} \abs{\nabla v_\eps}^p \vol
		&\leq \sum_{j = 0}^{J_\eps - 1} \haus^2_g(A_j^{(\eps)})^{\frac{2}{2-p}} \int_{A_j^{(\eps)}} \abs{\nabla v_\eps}^2 \vol \\
		&\leq 2 \sum_{j = 0}^{J_\eps-1} \haus^2_g(B_j^{(\eps)})^{\frac{2}{2-p}} \gl(v_\eps, A_j^{(\eps)})
		\leq C \tilde r_0^{\frac{4}{2-p}} \sum_{j = 0}^\infty (2^{-\frac{4}{2-p}})^j = \BigO(1).
	\end{align*}
	\CCC Finally \eqref{sobolev_loc_bound} holds true by c\BBB ombining the estimates from the previous three steps.

	\textit{\NNN Step 6 (P\BBB roof of (\ref{sbvp_loc_bound})):}
	Let us first show \CCC that \BBB
	\begin{equation}\label{comp_grad_dabsueps_bound}
		\sup_\eps \int_B \abs{\di \abs{u_\eps}}^2 \vol < \infty.
	\end{equation}
  \CCC Thanks to \eqref{comp_grad_bound_gl_veps} and \eqref{vortveps_conv} we can apply Lemma \ref{lem:localized_liminf_frac} to  $(v_\eps)$ \NNN with $m = 1$\BBB.
	Exploiting also \eqref{absgrad_after_doubling}, Lemma \ref{lem:mixed_term_convergence}, \NNN \eqref{comp_grad_ener_bound}, \BBB the boundedness of $(u_\eps)$ in $L^\infty$\NNN, and repeating the argument for the proof of \eqref{GLueps_Bk_bound} \CCC we obtain the estimate \BBB
	\begin{align*}
		\pi \leps - C
		&\leq \int_B \frac{1}{2}\abs{\nabla v_\eps}^2 + \frac{1}{4\eps^2} \NNN (1-\abs{v_\eps}^2)^2\BBB \vol \\
		&= \int_B \frac{m^2}{2} \abs{\nabla u_\eps}^2 + \frac{1}{4\eps^2} \NNN(1-\abs{u_\eps}^2)^2 \BBB \vol + \frac{1 - m^2}{2} \int_B \abs{\di \abs{u_\eps}}^2 \vol \\
		&\phantom{=}\quad + \frac{(m - 1)^2}{2} \int_B \abs{u_\eps}^2 \abs{\NNN\jac(\tau)\BBB}^2 \vol - m(m - 1) \int_B \sprod{\jac(u_\eps)}{\jac(\NNN\tau\BBB)} \vol \\
		&\leq m^2 \gl(u_\eps, B) + \tilde C - \frac{m^2 - 1}{2} \int_B \abs{\di \abs{u_\eps}}^2 \vol\\
		&\leq \pi \leps + \tilde C - \frac{m^2 - 1}{2} \int_B \abs{\di \abs{u_\eps}}^2 \vol.
	\end{align*}
	Consequently, (\ref{comp_grad_dabsueps_bound}) follows.
	Further\CCC more\BBB, by \NNN\eqref{grad_after_doubling}\BBB we see that
	\begin{equation*}
		\abs{u_\eps}^{-2} \jac(u_\eps) \otimes \NNN iv_\eps \BBB
		= \frac{1}{m} (\nabla v_\eps - \abs{u}^{-1} \di \abs{u_\eps} \otimes v_\eps
			+ (m - 1) \jac(\NNN\tau\BBB) \otimes \NNN(i v_\eps)\BBB).
	\end{equation*}
	Hence, by the triangular inequality, (\ref{sobolev_loc_bound}), and the boundedness of $(v_\eps)$ in $L^\infty$ it follows for any $p \in [1, 2)$ that
	\begin{equation}\label{comp_grad_jacueps_bound}
		\norm{\abs{u_\eps}^{-1} \jac(u_\eps)}_{L^p(\NNN T^*\!B\BBB)}
		\leq \frac{1}{m} (\norm{\nabla v_\eps}_{L^p(\NNN T^*\!B\BBB \otimes TB)} + \norm{\di u_\eps}_{L^p(\NNN T^*\!B\BBB)} + C) = \BigO(1).
	\end{equation}
	Finally, combining (\ref{comp_grad_dabsueps_bound}), (\ref{comp_grad_jacueps_bound}), and (\ref{nablau_decomp}) leads to (\ref{sbvp_loc_bound}).
\end{proof}

We are ready to prove our main compactness result.
\begin{proof}[Proof of Theorem \ref{thm:gamma_convergence} \NNN (i)\BBB]
	Let $r_0 \in (0, r^*)$ be chosen small \CCC enough that \BBB the balls $\setof{B_{r_0}(x_k)}_{\NNN k\BBB}$ are disjoint.
 	The existence of $\mu \in \tilde X^{(m)}\NNN(S)\BBB$ such that (\ref{compactness_vorts}) holds true follows by combining Lemma \ref{lem:initial_comp} and Proposition \ref{prop:improved_vortex_comp}.
	Let us assume that $\abs{\mu} = N$.
	Thanks to Proposition \ref{prop:improved_vortex_comp} we know that $\mu$ belongs to $X^{(m)}\NNN(S)\BBB$.

	\textit{\NNN Step 1 \BBB ($SBV^2_\loc$-compactness):}
	For any $r \in (0, r_0)$, by (\ref{sbv2_loc_bound}), the boundedness of $(u_\eps)$ in $L^\infty$, and Theorem \ref{thm:compactness_spec_bvsecs} there exists $u \in SBV^2(S_r(\mu))$ such that, up to taking a subsequence, $u_\eps \weakto u$ weakly in $SBV^2(S_r(\mu))$.
	Then, via a standard diagonal argument (\ref{w12loc_comp}) follows.
	Suppose from this point on that we have extracted a subsequence, without relabeling, such that (\ref{w12loc_comp}) holds true.
	Given $p \in [1, 2)$, \CCC by (\ref{sbvp_bound}) we have \BBB that $\nabla u_\eps \weakto G$ weakly in $L^p(T^*S \otimes TS)$.
	As by (\ref{w12loc_comp}) $\nabla u_\eps \weakto \nabla u$ weakly in $L^p(T^*S_r(\mu) \otimes TS_r(\mu))$ for $r \in (0, r_0)$, we derive that $\nabla u = G$ a.e.~in $S$.
	This shows that $\nabla u \in L^p(T^*S \otimes TS)$ for any $p \in [1, 2)$.

	\textit{\NNN Step 2 (P\BBB ointwise properties of $u$):}
	We continue by showing that $u$ has unit length.
	By the energy-bound in (\ref{energy_bound}) \CCC and \BBB the definition of $\gl$ we derive for any $r \in (0, r_0)$
	\begin{equation*}
		\int_{S_r(\mu)} (1-\abs{u_\eps})^2 \vol \leq \CCC 4 \NNN \eps^2 \BBB \gl(u_\eps, S) \BBB \leq C\eps^2 \leps = \littleo(1).
	\end{equation*}
  By \CCC the \NNN strong \CCC convergence \BBB of $(u_\eps)$ in $L^2(TS_r(\mu))$, \CCC up to subsequences (not relabeled) \BBB we have that $u_\eps \to u$ pointwise a.e.~in $S_r(\mu)$.
  Then, by the previous estimate, the boundedness of $(u_\eps)$ in $L^\infty(TS)$, and the dominated convergence theorem, \CCC it follow that \BBB
	\begin{equation*}
		\int_{S_r(\mu)} (1 - \abs{u})^2 \vol
		= \lim_{\eps \to 0} \int_{S_r(\mu)} (1 - \abs{u_\eps})^2 \vol
		= 0.
	\end{equation*}
	and therefore $\abs{u} = 1$ a.e.~in $S_r(\mu)$, and in $S$ due to the arbitrariness of $r$.
	Let us now consider a coordinate neighborhood $O \subset S_r(\mu)$ and arbitrary $\NNN \tau \BBB \in C^\infty(\NNN \bar O\BBB; \sph^1)$.
	We set $v_\eps \defas \power(u_\eps)$.
	By the definition of $\asm\NNN(S)\BBB$ and (\ref{jump_after_doubling}) applied to $u_\eps$ we have
	\begin{equation*}
		v_\eps^+ = \power(u_\eps^+) = \power(u_\eps^-) = v_\eps^- \qquad \text{$\haus^1_g$-a.e.~on } \mathcal J_{u_\eps} \cap O
	\end{equation*}
	and, therefore, $v_\eps \in W^{1, 1}(TO)$.
	Using the boundedness of $(u_\eps)$ in $L^\infty(TS)$, (\ref{sbv2_loc_bound}), and (\ref{absgrad_after_doubling}) applied to $u_\eps$ we derive that $\sup_\eps \norm{v_\eps}_{W^{1, 2}(TO)} < \infty$.
	Consequently, we can find $v \in W^{1, 2}(TO)$ such that, up to taking a subsequence, $v_\eps \to v$ pointwise a.e.~in $O$.
	We have already shown that $u_\eps \to u$ pointwise a.e.~in $S_r(\mu)$, up to subsequences.
	By the continuity of $\power$ this leads to $v = \power(u)$ a.e.~in $O$. 
	Hence, applying (\ref{jump_after_doubling}) to $u$, and using $v \in W^{1, 2}(TO)$ we see that
	\begin{equation*}
		\power(u^+) = v^+ = v^- = \power(u^-) \qquad \text{$\haus^1_g$-a.e.~on } \mathcal J_u \cap O.
	\end{equation*}
	By the arbitrariness of $O$ and $r$ the above property extends to $S$.

	\textit{\NNN Step 3 \BBB ($\vort(u) = \mu$):}
	We will now relate the vorticity of $u$ with the limit $\mu$ \CCC in \BBB \ref{compactness_vorts}.
	Note that by $\norm{\nabla u}_{L^1} < \infty$ and $\abs{u} = 1$ it follows that $\vort(u)$ is well defined \CCC in distributional \BBB sense.
  \NNN By \BBB (\ref{sbvp_bound}) and the fact that $\abs{\jac(u_\eps)} \leq \abs{\nabla u_\eps} \abs{u_\eps}$ \NNN there \BBB exists $j \in L^1(T^*S)$ such that, up to taking a subsequence, $\jac(u_\eps) \weakto j$ weakly in $L^1(T^*S)$.
	We will now show that $j = \jac(u)$ a.e.~in $S$.
	Let $r \in (0, r_0)$ and $\vphi \in L^\infty(T^*S_r(\mu))$, then
	\begin{align*}
		\int_{S_r(\mu)} \sprod{\vphi}{\jac(u_\eps) - \jac(u)} \vol
		&= \int_{S_r(\mu)} \sprod{\vphi}{\sprod{\nabla u_\eps}{i(u_\eps - u)}} \vol + \int_{S_r(\mu)} \sprod{\vphi}{\sprod{\nabla u_\eps - \nabla u}{iu}} \vol \\
		&= \int_{S_r(\mu)} \sprod{\nabla u_\eps}{i(u_\eps - u) \NNN \otimes \vphi \BBB} \vol + \int_{S_r(\mu)} \sprod{\nabla u_\eps - \nabla u}{iu \NNN \otimes \vphi \BBB} \vol.
	\end{align*}
	By (\ref{w12loc_comp}) we have that $\nabla u_\eps \weakto \nabla u$ weakly in $L^2(T^*S_r(\mu) \otimes TS_r(\mu))$ and $u_\eps \to u$ in $L^2(TS_r(\mu))$.
	Hence, \NNN by weak-strong convergence, \BBB both integrals in the last line above converge to $0$ as $\eps \to 0$.
	Furthermore, the weak convergence of $(j(u_\eps))$ in $L^1(T^*S)$ implies that
	\begin{equation*}
		\lim_{\eps \to 0} \int_{S_r(\mu)} \sprod{\vphi}{\jac(u_\eps) - \jac(u)} \vol = \int_{S_r(\mu)} \sprod{\vphi}{j - \jac(u)} \vol.
	\end{equation*}
	The arbitrariness of $\vphi$ and $r$ shows that $j = \jac(u)$ a.e.~in $S$.
	Using the weak convergence of $(\jac(u_\eps))$ we derive for any $\vphi \in W^{1, \infty}(TS)$ that
	\begin{equation*}
		\sprod{\vort(u)}{\vphi}
		= \int_S \di \vphi \wedge \jac(u)
		= \lim_{\eps \to 0} \int_S \di \vphi \wedge \jac(u_\eps)
		= \lim_{\eps \to 0} \sprod{\vort(u_\eps)}{\vphi}
		= \sprod{\mu}{\vphi}
	\end{equation*}
	\CCC which gives \BBB $\vort(u) = \mu$ \CCC by \BBB the arbitrariness of $\vphi$.

	\textit{\NNN Step 4 (Finite jump)\BBB:}
	In order to prove that $u \in \lsm$ it remains to show $\haus^1_g(\mathcal J_u) < \infty$.
	By the same argument as in the proof of \NNN Proposition \BBB \ref{prop:improved_vortex_comp} we can show that for any $r \in (0, r_0)$
	\begin{equation*}
		\gl(u_\eps, S_r(\mu)) \leq \pi \frac{N}{m} \abs{\log r} + C
	\end{equation*}
	for some constant $C$ independent of $r$ and $\eps$.
	Solving \NNN the above inequality \BBB for $\haus^1_g(\mathcal J_\eps)$ we derive that
	\begin{equation*}
		\haus^1_g(\mathcal J_{u_\eps} \cap S_r(\mu))
		\leq \pi \frac{N}{m} \abs{\log r} + C - \frac{1}{2} \int_{S_r(\mu)} \abs{\nabla u_\eps}^2 \vol.
	\end{equation*}
	Then, by (\ref{w12loc_comp})
	\begin{align*}
		\haus^1_g(\mathcal J_u \cap S_r\NNN(\mu)\BBB)
		\leq \liminf_{\eps \to 0} \haus^1_g(\mathcal J_{u_\eps} \cap S_r(\mu))
		&\NNN \leq C + \pi \frac{N}{m} \abs{\log r} + \limsup_{\eps \to 0} - \frac{1}{2} \int_{S_r(\mu)} \abs{\nabla u_\eps}^2 \vol \BBB \\
		&\leq C - \left(\frac{1}{2} \int_{S_r(\mu)} \abs{\nabla u}^2 \vol - \pi \frac{N}{m} \abs{\log r}\right).
	\end{align*}
	With $\renm(u) > -\infty$ (see Lemma \ref{lem:ren_ener_welldefined}), it then follows:
	\begin{equation*}
		\haus^1_g(\mathcal J_u)
		= \limsup_{r \to 0} \haus^1(\mathcal J_u \cap S_r)
		\leq C - \renm(u) < \infty,
	\end{equation*}
	as desired.
	\NNN From the inequality above we can also derive that \BBB
	\begin{equation*}
		\limsup_{\eps \to 0} \haus^1_g(\mathcal J_{u_\eps}) \NNN \leq C - \renm(u) < \infty.\BBB
	\end{equation*}
	\NNN Hence, by possibly selecting a subsequence we can assume that $\sup_\eps \haus^1_g(\mathcal J_{u_\eps}) < \infty$, \BBB
	\CCC which, c\BBB ombined with (\ref{sbvp_bound}) and the boundedness of $(u_\eps)$ in $L^\infty$ leads to \eqref{w1p_comp} thanks to Theorem \ref{thm:compactness_spec_bvsecs}.
\end{proof}

\subsection{\NNN$\Gamma$\BBB-liminf}
In this \CCC section \BBB we will prove the \CCC $\liminf$-inequality of Theorem \ref{thm:gamma_convergence} \NNN (ii)\BBB.
\CCC F\BBB or any open set $O \subset \CCC S \BBB$ \CCC we \BBB define the modified Ginzburg-Landau energy $\glm \colon \as^{\NNN(1)\BBB}(O) \to \rr$ as:
\begin{equation}\label{def_glm}
	\glm(v) = \glm(v, O) 
	\defas \frac{1}{2m^2} \int_O \abs{\nabla{v}}^2 
		+ (m^2 - 1) \abs{\der \abs{v}}^2 
		+ \frac{m^2}{2\eps^2} \NNN(1-\abs{v}^2)^2\BBB \vol.
\end{equation}
Note that\CCC, \BBB the functional above is the natural candidate to keep track of the energy concentration in our setting.
More precisely, let $O = B$ be a geodesic ball with radius $r \in (0, r^*)$, $u \in \asm(B)$, $\NNN \tau \BBB \in C^\infty(\NNN \bar B\BBB; \sph^1)$, and set $v \defas \power(u) \in \as^{\NNN(1)\BBB}(B)$.
Then, by (\ref{absgrad_after_doubling})
\begin{equation}\label{glm_relation}
		\gl(u)
		= \glm(v, B) 
		- \frac{m^2 - 1}{m^2} \int_B \abs{u_\eps}^2 \abs{\NNN\jac(\tau)\BBB}^2 \vol
		+ \frac{2m(m - 1)}{m^2} \int_B \sprod{\jac(u_\eps)}{\jac(\NNN\tau\BBB)} \vol,
\end{equation}
where the latter two terms will turn \CCC out \BBB to be negligible \NNN for small balls.\BBB

\begin{remark}
	\label{rem:coordinate_representation_ball}
	Throughout this subsection, given $x_0 \in S$ and $r \in (0, r^*)$, $\Psi$ will stand for a local trivialization of $TS$ corresponding to centered (at $x_0$) normal coordinates on $B_r(x_0)$ with chart denoted by $\Phi$ and an auxiliary orthonormal frame $\setof{\NNN\tau_1\BBB, \NNN\tau_2\BBB}$ on $TB_r(x_0)$ \NNN (smooth up to the boundary)\BBB.
	Objects such as $g^{ij}$, $\afac$, $\Gamma_{ij}^k$, etc.~will always correspond to the above choice of coordinates.
  For an arbitrary section $v$ of $TB_r(x_0)$ we will write $\Psi^* v$ for its coordinate representation.

	Note that\CCC, \BBB under this assumptions\CCC, \BBB the following holds true:
	\begin{equation}\label{taylor_normal_coords}
		g^{ij} = \delta^{ij} + \BigO(r), \quad \Gamma_{ij}^k = \BigO(r), \quad \afac = 1 + \BigO(r).
	\end{equation}
\end{remark}

\begin{lemma}\label{comparison_glm_glmflat}
	Let $v \in W^{1, 2}(TB_{r_0}(x_0)) \cap L^\infty(TB_{r_0}(x_0))$ for $x_0 \in S$ and $r_0 \in (0, r^*)$.
	Further\CCC more, \BBB let $\Psi$ be a local trivialization \NNN of \BBB $\NNN T\BBB B_{r_0}(x_0)$ as described in Remark \ref{rem:coordinate_representation_ball}.
	Then, for any $r \in (0, r_0)$ it holds that
	\begin{align}
		\int_{B_r(x_0)} \abs{\nabla v}^2 \vol &= (1 + \BigO(r)) \int_{\CCC B_r(0)\BBB} \abs{\nabla (\Psi^* v)}^2 \di x + \BigO(r) \norm{v}_{L^\infty}^2, \label{comparison_L2_grad_norms} \\
		\glm(v, B_r(x_0)) &= (1 + \BigO(r)) \glmflat(\Psi^*v, B_r(0)) + \BigO(r) \norm{v}_{L^\infty}^2, \label{comparison_glms}
	\end{align}
	where all $\BigO(r)$-terms are independent of $v$ \NNN and $\overline{GL}_\eps^{(m)}$ is as in \eqref{barGLeps}. \BBB
\end{lemma}

\begin{proof}
	For the sake of shorter notation we will write $\bar v$ instead of $\Psi^* v$.
	By a standard approximation procedure we can assume without loss of generality that $v$ is smooth.
	Further\CCC more\BBB, by the equivalence of norms, we have that $\norm{\bar v}_{L^\infty} \leq C \norm{v}_{L^\infty}$.
	Hence, using (\ref{taylor_normal_coords}) and Young's inequality it follows that
	\begin{align*}
		&\int_{B_r(x_0)} \abs{\nabla v}^2 \vol
		= \int_{B_r(0)} \sum_{k = 1}^2
			\left( \frac{\partial \bar v^k}{\partial x^i} + \Gamma_{il}^k \bar v^l \right)
			\left( \frac{\partial \CCC \bar v\BBB^k}{\partial x^j} + \Gamma_{jl'}^k \bar v^{l'} \right) g^{ij} \afac \di x \\
		&\quad= \int_{B_r(0)} \abs{\nabla \bar v}^2 + \BigO(1)\abs{\bar v}(\abs{\bar v} + \abs{\nabla \bar v}) \di x + \BigO(r) \int_{B_r(0)} \abs{\bar v}^2 + \abs{\nabla \bar v}^2 \di x \\
		&\quad= \int_{B_r(0)} \abs{\nabla \bar v}^2 + \BigO(r^{-1}) \abs{\bar v}^2 + \BigO(r) \abs{\nabla \bar v}^2 \di x + \BigO(r) \int_{B_r(0)} \abs{\bar v}^2 + \abs{\nabla \bar v}^2 \di x \\
		&\quad= (1 + \BigO(r)) \int_{B_r(0)} \abs{\nabla \bar v}^2 \di x + \BigO(r^{-1}) \int_{B_r(0)} \abs{\bar v}^2 \di x
		= (1 + \BigO(r)) \int_{B_r(0)} \abs{\nabla \bar v}^2 \di x + \BigO(r) \norm{v}_{L^\infty}^2,
	\end{align*}
	which shows (\ref{comparison_L2_grad_norms}).
	Using the above result together with (\ref{taylor_normal_coords}) we can similarly show \NNN \eqref{comparison_glms}. \BBB
\end{proof}



Let us now consider $v \in W^{1, 2}(TB_r(x_0))$ for some $x_0 \in S$ and $r_0 \in (0, r^*)$.
Then, given any $r \in (0, r_0)$, we define
\begin{equation*}
	\drot(v, r) \defas \inf_{z \in \rots} \setof*{r^{-1} \norm{\Psi^*v - z}_{L^2(B_r(0); \rr^2)} + \norm{\nabla (\Psi^*v) - \nabla z}_{L^2(B_r(0); \rr^{2 \times 2})}},
\end{equation*}
where
\begin{equation*}
	\rots \defas \setof*{v \colon \rr^2 \setminus \setof{0} \to \rr^2 \colon v(x) = \lambda \frac{x}{\abs{x}} \text{ for some } \lambda \in \sph^1}.
\end{equation*}

\begin{lemma}\label{lem:tradeoff}
  Let $(r_\eps) \subset (0, r^*)$ with $\lim_{\eps \to 0} r_\eps = 0$, $(x_\eps) \subset S$ \CCC and $(v_\eps) \subset \BBB W^{1, 2}(TA_\eps)$\BBB, where $A_\eps \defas A_{\frac{r_\eps}{2}, r_\eps}(x_\eps)$.
	We assume that $\sup_\eps \norm{v_\eps}_{L^\infty(A_\eps)} < \infty$, $\deg(v_\eps, \partial B_{r_\eps}(x_\eps)) = s \in \setof{-1, 1}$,
	\begin{equation}
		\label{tradeoff_length1_condition}
		\lim_{\eps \to 0} r_\eps^{-2} \int_{A_\eps} (1 - \abs{v_\eps})^2 \vol = 0,
	\end{equation}
	and
	\begin{equation}
		\label{tradeoff_distance_bound}
		\drot(v_\eps, r_\eps) \geq \delta
	\end{equation}
	for \CCC some $\delta > 0$\BBB.
	Then, there exists $\omega(\delta) > 0$ such that
	\begin{equation*}
		\frac{1}{2} \int_{A_\eps} \abs{\nabla v_\eps}^2 \vol \geq \pi \log(2) + \omega(\delta) + \littleo(1).
	\end{equation*}
\end{lemma}

\begin{proof}
	Without loss of generality we can assume that $s = 1$, as the other case follows by a similar argument.
	Suppose, by contradiction, that up \CCC to a \BBB subsequenc\NNN e\BBB
	\begin{equation}
		\label{tradeoff_contradiction}
		\limsup_{\eps \to 0} \frac{1}{2} \int_{A_\eps} \abs{\nabla v_\eps}^2 \vol \leq \pi \log(2).
	\end{equation}
	\NNN L\BBB et $\bar A_\eps \defas A_{\frac{r_\eps}{2}, r_\eps}(0)$ be the Euclidean annulus corresponding to $A_\eps$ and $\bar v_\eps \defas \Psi^* v_\eps$.
	Note that by the equivalence of norms and our assumptions on \CCC $(v_\eps)$ \BBB
	\begin{equation*}
		\sup_\eps \setof*{\norm{\bar v_\eps}_{L^\infty(\bar A_\eps)} + \int_{\bar A_\eps} \abs{\nabla \bar v_\eps}^2 \di x} < \infty.
	\end{equation*}
	\NNN As $\bar r_\eps \to 0$, w\BBB ith (\ref{comparison_L2_grad_norms}) and the contradiction assumption (\ref{tradeoff_contradiction}) we \CCC have \BBB
	\begin{equation}
		\label{limsup_barAeps}
		\limsup_{\eps \to 0} \int_{\bar A_\eps} \abs{\nabla \bar v_\eps}^2 \di x \leq \pi \log(2).
	\end{equation}
	\NNN We \BBB rescale each $\bar v_\eps$ to a vector-field $\bar w_\eps$ defined on the unit annulus $\bar A_1$.
	More precisely, we set $\bar w_\eps(x) \defas \bar v_\eps(r_\eps x)$ for $x \in \bar A_1$.
	From (\ref{limsup_barAeps}) and a change of coordinates \CCC it follows \BBB that
	\begin{equation*}
		\limsup_{\eps \to 0} \int_{\bar A_1} \abs{\nabla \bar w_\eps}^2 \di x
		= \limsup_{\eps \to 0} \int_{\bar A_\eps} \abs{\nabla \bar v_\eps}^2 \di x \leq \pi \log(2).
	\end{equation*}
  Together with the boundedness of $(\bar w_\eps)$ in $L^\infty$, this implies that, up to selecting a subsequence, $\bar w_\eps \weakto \bar w$ weakly in $W^{1, 2}(\bar A_1; \rr^2)$ with $\bar w$ satisfying
	\begin{equation}
		\label{tradeoff_upper_boundw}
		\frac{1}{2} \int_{\bar A_1} \abs{\nabla \bar w}^2 \di x
		\leq \liminf_{\eps \to 0} \int_{\bar A_1} \abs{\nabla \bar w_\eps}^2 \di x \leq \pi \log(2).
	\end{equation}
	Furthermore, by (\ref{tradeoff_length1_condition}) \NNN and a change of coordinates \BBB we \CCC also have \BBB 
	\begin{equation*}
		\int_{\bar A_1} (1 - \abs{\bar w_\eps})^2 \di x
		= r_\eps^{-2} \int_{\bar A_\eps} (1 - \abs{\bar v_\eps})^2 \di x
		\leq C r_\eps^{-2} \int_{A_\eps} (1 - \abs{v_\eps})^2 \vol \to 0,
	\end{equation*}
	as $\eps \to 0$, and therefore $\abs{\bar w} = 1$ a.e.~in $\bar A_1$.
	Finally, by the continuity of the degree with respect to weak convergence in $W^{1, 2}$, it follows that $\deg(\bar w, \partial B_1(0)) = 1$.
	\CCC Combining \NNN this result \CCC with (\ref{tradeoff_upper_boundw}) gives that $\bar w \in \rots$ thanks to Remark 5.2 \BBB in \cite{marcello_gl}.
	By (\ref{tradeoff_upper_boundw})\CCC, recalling that $\frac{1}{2} \int\abs{\nabla v} \di x \geq \pi \log(2)$ for any $v \in \bar{\mathcal{H}}$, we have \BBB
	\begin{equation*}
		\lim_{\eps \to 0} \int_{\bar A_1} \abs{\nabla \bar w_\eps}^2 \di x = \int_{\bar A_1} \abs{\nabla \bar w}^2 \di x,
	\end{equation*}
	By the weak convergence of $(\bar w_\eps)$ this leads to $\bar w_\eps \to \bar w$ strongly in $W^{1, 2}(\bar A_1; \rr^2)$.
	Hence, changing coordinates and using the definition of $\drot$ we derive that
	\begin{align*}
		\drot^2(v_\eps, r_\eps)
		&\leq r_\eps^{-2} \int_{\bar A_\eps} \abs{\bar v_\eps(x) - \bar w(x)}^2 \di x
			+ \int_{\bar A_\eps} \abs{\nabla \bar v_\eps(x) - \nabla \bar w(x)}^2 \di x \\
		&= \int_{\bar A_1} \abs{\bar w_\eps(x) - \bar w(x)}^2 \di x
			+ \int_{\bar A_1} \abs{\nabla \bar w_\eps(x) - \nabla \bar w(x)}^2 \di x \to 0
	\end{align*}
	as $\eps \to 0$, which is a contradiction to (\ref{tradeoff_distance_bound}) for $\eps$ small enough.
\end{proof}

\begin{proof}[Proof of Theorem \ref{thm:gamma_convergence} \NNN (ii)\BBB]
	\CCC Let \BBB us first select a subsequence (without relabeling) such that
	\begin{equation*}
		\liminf_{\eps \to 0} \left(\gl(u_\eps) - \pi\frac{N}{m} \leps\right)
		= \lim_{\eps \to 0} \left(\gl(u_\eps) - \pi\frac{N}{m} \leps\right).
	\end{equation*}
	Note that given any $w \in \asm$, the truncation $\hat w \defas \min\setof{1, \NNN\abs{w}^{-1}\BBB} w$ has lower energy: \NNN$\gl(\hat w) \leq \gl(w)$. \BBB
  Consequently, without loss of generality, we can assume that $\sup_\eps \norm{u_\eps}_{L^\infty} \leq 1$.
	Furthermore, it is not restrictive to suppose tha\NNN t\BBB
	\begin{equation*}
		\gl(v_\eps) \leq \pi\frac{N}{m} \leps + C
	\end{equation*}
	for some constant $C$ independent of $\eps$ \CCC since \BBB otherwise (\ref{gamma_liminf}) trivially follows.
	By Theorem \ref{thm:gamma_convergence} \NNN (i)\BBB, we can select a subsequence, again without relabeling, such that
	\begin{equation}\label{sbv2_loc_conv_liminf}
		u_\eps \weakto u \text{ weakly in } SBV^2_\loc(\NNN S \setminus \spt(\mu); TS\BBB),
	\end{equation}
	where $\mu = \vort(u) = \NNN \sum_{k=1}^{mN} \BBB \frac{s_k}{m} \dirac_{x_k}$ with $s_k \in \setof{-1, 1}$ for all $k$.
	Let $r_0 \in (0, r^*)$ be small enough such that the balls in $\setof{B_{r_0}(x_k)}_k$ are pairwise disjoint and for $r \in (0, r_0)$ \CCC recall the definition of $S_r(\mu) \defas S \setminus \bigcup_k B_{r_0}(x_k)$. \BBB

	\textit{\NNN Step 1 (L\BBB ower bound in $S_r(\mu)$):}
	We first wish to derive the $\Gamma$\CCC-\BBB$\liminf$ inequality in $S_r(\mu)$ for $r \in (0, r_0)$.
  By (\ref{sbv2_loc_conv_liminf}), standard lower semicontinuity arguments, the definition of $\renm(u)$, and the fact that $\haus^1(J_u) < \infty$, \CCC it holds that \BBB
	\begin{align*}
		\liminf_{\eps \to 0} \gl(u_\eps, S_r(\mu))
		&\geq \liminf_{\eps \to 0} \left( \frac{1}{2} \int_{S_r(\mu)} \abs{\nabla u_\eps}^2 \vol + \haus^1_g(\mathcal J_{u_\eps} \cap S_r) \right) \\
		&\geq \frac{1}{2} \int_{S_r(\mu)} \abs{\nabla u}^2 \vol + \haus^1_g(\mathcal J_u \cap S_r(\mu))
		= \renm(u) + \pi \frac{N}{m} \abs{\log r} + \haus^1(\mathcal J_u) + \littleo_r(1).
	\end{align*}

	\textit{\NNN Step 2 \CCC (\NNN F\CCC rom $u_\eps$ to $v_\eps$)\BBB:}
	It remains to sho\NNN w\BBB
	\begin{equation}
		\label{liminf_cores}
		\liminf_{\eps \to 0} \left(\gl(u_\eps, B_r(x_k)) - \frac{\pi}{m^2} \log\CCC\left(\BBB\frac{r}{\eps}\CCC\right)\BBB  \right) \geq \gamma_m + \littleo_r(1)
	\end{equation}
	for any vortex center $x_k$ of $u$.
	\NNN Let us from now on fix $k$ and shortly write $B_r \defas B_r(x_k)$. \BBB
	\NNN Furthermore, l\BBB et $\NNN\tau\BBB \in C^\infty(\NNN\bar B\BBB_{r_0}; \sph^1)$, $v_\eps \defas \power(u_\eps)$ (with $\power$ as in (\ref{mth_power})) \CCC and let \BBB $\glm$ be the energy functional in (\ref{def_glm}).
	As \CCC it \BBB was done in the proof of Proposition \ref{prop:improved_vortex_comp} \CCC one \BBB can show that (see (\ref{local_glveps_prop5}))
	\begin{equation}\label{glveps_local_ener_bound}
		\gl(v_\eps, \NNN B_{r_0}\BBB) \leq m^2 \gl(u_\eps, \NNN B_{r_0}\BBB) \leq \pi \leps + C,
	\end{equation}
	for a constant $C$ independent of $\eps$.
  Using (\ref{sbvp_bound}), the boundedness of $(u_\eps)$ in $L^\infty$, the smoothness of \NNN$\tau$\BBB, and Hölder's inequality we derive that
	\begin{equation}\label{glm_gl_error}
		\sup_\eps \int_{B_r(x_0)} \abs{u_\eps}^2 \abs{\NNN\jac(\tau)\BBB}^2 + \abs{\sprod{\jac(u_\eps)}{\jac(\NNN\tau\BBB)}} \vol
		\leq \CCC C \BBB \norm{\NNN\jac(\tau)\BBB}_{L^\infty}^2 \CCC r^2 \BBB + \CCC C \BBB \norm{\NNN\jac(\tau)\BBB}_{L^\infty} \CCC r^{\frac{2}{3}} \BBB \sup_\eps \norm{\nabla u_\eps}_{L^{\frac{3}{2}}} \leq C r^{\frac{2}{3}}.
	\end{equation}
	Hence, by \ref{glm_relation}, instead of (\ref{liminf_cores}), we can equivalently show that
	\begin{equation}
		\label{liminf_coresv2}
		\liminf_{\eps \to 0} \left(\glm(v_\eps, B_r(x_0)) - \frac{\pi}{m^2} \log\CCC\left(\BBB\frac{r}{\eps}\CCC\right)\BBB  \right) \geq \gamma_m + \littleo_r(1).
	\end{equation}

	\textit{\NNN Step 3 (L\BBB ower bound outside dyadic annuli):}
	By a standard approximation argument we can assume that $v_\eps$ is smooth for every $\eps$.
	\CCC Thanks to \BBB (\ref{glveps_local_ener_bound}), we can \CCC exploit \BBB Lemma \ref{lem:ballcon_spec} \CCC with $O = \NNN B_r$\BBB.
	Let $(x_\eps)$ be the sequence \CCC in Lemma \ref{lem:ballcon_spec} \NNN (i)\BBB.
	We further set $r_\eps \defas 2\eps^{\frac{1}{3}}$ and $R_\eps \defas \eps^{\frac{1}{4}}$.
	Note that, as \NNN $(x_\eps)$ converges to the center of $B_{r_0}$\BBB, for $t \in (0, r)$ and $\eps$ small enough it holds that $B_t(x_\eps) \subset B_{r_0}$.
	Consequently, by \NNN \eqref{ener_lowbound_circles} in \BBB Lemma \ref{lem:ener_lowbound_circles} there exists a constant $C$ only depending on $S$ such that for any $r' \in (0, r)$ we have that
	\begin{align}
		\glm(v_\eps, A_{R_\eps, r'}(x_\eps))
		&\geq \frac{1}{m^2} \gl(v_\eps, A_{R_\eps, r'}(x_\eps)) \nonumber \\
		&= \frac{1}{m^2} \int_{R_\eps}^{r'} \frac{1}{2} \int_{\partial B_t(x_\eps)} \abs{\nabla v_\eps}^2 + \frac{1}{2\eps^2} \CCC (1 - \abs{v_\eps}^2)^2 \BBB \di \haus^{d-1} \di t \nonumber\\
		&\geq \frac{1}{m^2} \int_{R_\eps}^{r'} \frac{\pi(1 - C t^2)}{t + C\eps}  \di t \nonumber \\
		&\geq \frac{\pi}{m^2} \left(\log\CCC\left(\BBB\frac{r'}{R_\eps}\CCC\right)\BBB - \log(1 + C \eps^{\frac{3}{4}}) - \frac{C}{2} ((r')^2 - R_\eps^2) \right) \nonumber \\
		&\geq \frac{\pi}{m^2} \log\CCC\left(\BBB\frac{r}{R_\eps}\CCC\right)\BBB - C \left(\log\CCC\left(\BBB\frac{r}{r'}\CCC\right)\BBB + r \right) + \littleo(1). \label{liminf_outer_ring}
	\end{align}
	In the same fashion we can also show \CCC that \BBB for any $K \in \nn$ and $\eps > 0$ small enough\CCC,\BBB
	\begin{equation}\label{liminf_inner_ring}
		\glm(v_\eps, A_{r_\eps, 2^{-K} R_\eps}(x_\eps))
		\geq \frac{\pi}{m^2} \log\CCC\left(\BBB\frac{R_\eps}{r_\eps}\CCC\right)\BBB - K \frac{\pi}{m^2} \log(2) + \littleo(1).
	\end{equation}
	Lastly, by Lemma \ref{lem:ballcon_spec} \NNN (iii) \BBB it holds that
	\begin{equation}\label{liminf_core}
		\glm(v_\eps, B_{r_\eps}(x_\eps)) 
		\geq \frac{1}{m^2} \gl(v_\eps, B_{r_\eps}(x_\eps))
		\geq \frac{\pi}{m^2} \log\CCC\left(\BBB\frac{r_\eps}{\eps}\CCC\right)\BBB - \tilde C.
	\end{equation}
	Combining (\ref{liminf_outer_ring}), (\ref{liminf_inner_ring}), and (\ref{liminf_core}) leads \NNN to \BBB
	\begin{equation}\label{liminf_without_ring}
	\begin{aligned}
		&\glm(v_\eps, B_{r'}(x_\eps) \setminus A_{2^{-K}R_\eps, R_\eps}(x_\eps)) \\
		&\quad\geq \frac{\pi}{m^2} \log\left(\frac{r}{\eps}\right) \NNN - K \frac{\pi}{m^2} \log(2) \BBB - \tilde C - C \left(\log\CCC\left(\BBB\frac{r}{r'}\CCC\right)\BBB + r\right) + \littleo(1). 
	\end{aligned}
	\end{equation}
	Given $\delta > 0$, let $K = K(\delta) \in \nn$ be chosen (independently of $\eps$ and $r$) big enough such that
	\begin{equation}
		\label{K_condition}
		K \omega(\delta) \geq \NNN \gamma_m + \tilde C\BBB,
	\end{equation}
	where $\omega(\delta) > 0$ is as in Lemma \ref{lem:tradeoff}.
	We need to discern between two cases.

	\textit{\NNN Step 4 \BBB ($(v_\eps)$ away from rotations):}
	In the first case we assume that, up to taking a subsequence, $\drot(v_\eps, 2^{-k}R_\eps) \geq \delta$ for all $k \in \setof{0, \dots, K-1}$.
	\CCC Let us observe that, thanks to \eqref{glveps_local_ener_bound} we have that
	\begin{equation*}
		(2^{-k} R_\eps)^{-2} \int_{A_k^{(\eps)}} (1 - \abs{v_\eps})^2 \vol
		\leq 2^{K+2} \eps^{\frac{3}{2}} \gl(v_\eps) = \littleo(1).
	\end{equation*}
	where $A_k^{(\eps)} \defas A_{2^{-(k + 1)}R_\eps, 2^{-k}R_\eps}(x_\eps)$.	
	Hence the assumptions of Lemma \ref{lem:tradeoff} are satisfied and we obtain that
	\begin{equation}
		\label{large_energy_in_dyadic_annulus}
		\int_{A_k^{(\eps)}} \abs{\nabla v_\eps}^2 \vol \geq \frac{\pi}{m^2} \log(2) + \omega(\delta) + \littleo(1),
	\end{equation}
	\BBB
	\NNN As \BBB $B_{r'}(x_\eps) \subset B_r(x_0)$ for sufficiently small $\eps$\NNN, \BBB for all such $\eps$, by \CCC \eqref{liminf_without_ring}, \eqref{K_condition}, and \eqref{large_energy_in_dyadic_annulus} \BBB it follows that
	\begin{align*}
		&\glm(v_\eps, \NNN B_r\BBB) \geq \glm(v_\eps, B_{r'}(x_\eps)) \\
		&\quad\geq \glm(v_\eps, B_{r'}(x_\eps) \setminus A_{2^{-K}R_\eps, R_\eps}(x_\eps)) + \glm(v_\eps, A_{2^{-K}R_\eps, R_\eps}(x_\eps)) \\
		&\quad\geq \frac{\pi}{m^2} \log\CCC\left(\BBB\frac{r}{\eps}\CCC\right)\BBB  + K \omega(\delta) \NNN - \tilde C \BBB - C\left(\log\CCC\left(\BBB\frac{r}{r'}\CCC\right)\BBB + r\right) + \littleo(1) \\
		&\quad\geq \frac{\pi}{m^2} \log\CCC\left(\BBB\frac{r}{\eps}\CCC\right)\BBB  + \gamma_m  - C \left(\log\CCC\left(\BBB\frac{r}{r'}\CCC\right)\BBB + r^2\right) + \littleo(1).
	\end{align*}
	\CCC Letting first $\eps \to 0$ and then $r' \to r$, \BBB (\ref{liminf_coresv2}) follows \CCC from the previous estimate\BBB.

	\textit{\NNN Step 5 \BBB ($(v_\eps)$ close to rotations):}
	We will now deal with the second case.
	Suppose that, up to taking a subsequence, we can find $k_0 \in \setof{0, \dots, K-1}$ such that
	\begin{equation}\label{assumption_drot_step5}
		\drot(v_\eps, \sigma_\eps) < \delta \qquad \NNN \text{for all } \eps \BBB 
	\end{equation}
	where $\sigma_\eps \defas 2^{-k_0} R_\eps$.
	\CCC We will now sho\NNN w \CCC that the following inequality \eqref{required_lowb_step5} leads to the conclusion\NNN: \BBB
	\begin{equation}\label{required_lowb_step5}
		\glm(v_\eps, B_{\sigma_\eps}(x_\eps)) \geq \frac{\pi}{m^2} \log\CCC\left(\BBB\frac{\sigma_\eps}{\eps}\CCC\right)\BBB + \gamma_m - C \delta + \littleo(1)\NNN, \BBB
	\end{equation}
	\NNN where $C$ is a constant independent of $\eps$, $\delta$, and $r$. \BBB
	\CCC G\BBB iven $r' \in (0, r)$, where $r \in (0, r_0)$, \CCC b\BBB y the same argument as in the third step, \CCC we have \BBB that
	\begin{equation*}
		\glm(v_\eps, A_{\sigma_\eps, r'}(x_\eps))
		\geq \frac{\pi}{m^2} \log\CCC\left(\BBB\frac{r}{\sigma_\eps}\CCC\right)\BBB - C\left(\log\CCC\left(\BBB\frac{r}{r'}\CCC\right)\BBB + r^2\right) + \littleo(1).
	\end{equation*}
	Consequently, by (\ref{required_lowb_step5})
	\begin{align*}
		&\liminf_{\eps \to 0} \glm(v_\eps, B_r(x_0)) - \frac{\pi}{m^2} \log\CCC\left(\BBB\frac{r}{\eps}\CCC\right)\BBB  \\
		&\quad\geq \liminf_{\eps \to 0} \glm(v_\eps, A_{\sigma_\eps, r'}(x_\eps)) - \frac{\pi}{m^2} \log\CCC\left(\BBB\frac{r}{\sigma_\eps}\CCC\right)\BBB
		+ \liminf_{\eps \to 0} \glm(v_\eps, B_{\sigma_\eps}(x_\eps)) - \frac{\pi}{m^2} \log\CCC\left(\BBB\frac{\sigma_\eps}{\eps}\CCC\right)\BBB \\
		&\quad\geq \gamma_m - C\left(\log\CCC\left(\BBB\frac{r}{r'}\CCC\right)\BBB + r^2 + \delta\right),
	\end{align*}
	which shows (\ref{liminf_coresv2}) after sending \CCC first \BBB $r' \to r$, \CCC then \BBB $r \to 0$\CCC, \BBB and \CCC eventually \BBB $\delta \to 0$.
	\NNN It remains to prove \eqref{required_lowb_step5}. \BBB
	\NNN Let \BBB $\bar v_\eps \defas \Psi^*_\eps v_\eps$ for a \CCC sequence of \BBB local trivialization\CCC s \BBB $\CCC(\BBB\Psi_{\CCC \eps\BBB}\CCC)\BBB$ of $TB_{\sigma_\eps}(x_\eps)$ as described in Remark \ref{rem:coordinate_representation_ball}. \footnote{\CCC Let $\Psi$ be a local trivialization of $B_{r_0}(x_0)$ as described in Remark \ref{rem:coordinate_representation_ball}, then $\Psi_\eps(x) \defas \Psi(x - \Phi^{-1}(x_\eps))$.}
	Using (\ref{comparison_glms}) and (\ref{glveps_local_ener_bound}) we derive that
	\begin{equation*}
		\abs*{\glm(v_\eps, B_{\sigma_\eps}(x_\eps)) - \glmflat(\bar v_\eps, B_{\sigma_\eps}(0))}
		\leq \BigO(\sigma_\eps) \gl(v_\eps, B_{\sigma_\eps}(x_\eps)) + \BigO(\sigma_\eps)
		= \NNN \littleo(1)\BBB.
	\end{equation*}
	\CCC According to the above estimate, \NNN instead of \eqref{required_lowb_step5}, \BBB it suffices to prove
	\begin{equation}\label{required_lowb_step5_flat}
		\glmflat(\bar v_\eps, B_{\sigma_\eps}(0)) \geq \frac{\pi}{m^2} \log\CCC\left(\BBB\frac{\sigma_\eps}{\eps}\CCC\right)\BBB + \gamma_m - C\delta + \littleo(1).
	\end{equation}
	Note that by the definition of $\drot$ and (\ref{assumption_drot_step5}) we can find $\bar z_\eps = \lambda_\eps \frac{x}{\abs{x}} \in \rots$ such that
	\begin{equation}\label{close_to_zeps}
		\int_{A_{\frac{\sigma_\eps}{2}, \sigma_\eps}(0)} \frac{\abs{\bar v_\eps - \bar z_\eps}^2}{\sigma_\eps^2} + \abs{\nabla \bar v_\eps - \nabla \bar z_\eps}^2 \di x \leq \delta^2.
	\end{equation}
	Through an interpolation procedure we will now modify $\bar v_\eps$ into a vector-field $ \CCC \hat v_\eps \BBB $ such that $\bar v_\eps =  \CCC \hat v_\eps \BBB $ in $B_{\frac{\sigma_\eps}{2}}(0)$ and $ \CCC \hat v_\eps \BBB  = \bar z_\eps$ on $\partial B_{\sigma_\eps}(0)$.
	As \NNN $B_{\frac{r_0}{2}}(x_\eps) \subset B_{r_0}(x_0)$ \BBB for $\eps$ small enough\NNN, using \BBB (\ref{glveps_local_ener_bound}) and (\ref{glm_gl_error}) we then see by a similar argument as in \NNN Step 3 \BBB that
	\begin{align*}
		&\glm(v_\eps, A_{\frac{\sigma_\eps}{2}, \sigma_\eps(x_\eps)}) \\
		&\quad\leq \gl(u_\eps, B_{r_0}(x_0)) + C r_0^{\frac{2}{3}} - \glm(v_\eps, A_{\sigma_\eps, \frac{r_0}{2}}(x_\eps)) - \glm(v_\eps, B_{\frac{\sigma_\eps}{2}}(x_\eps)) \\
		&\quad\leq \frac{\pi}{m^2} \leps + C r_0^{\frac{2}{3}}
			- \frac{\pi}{m^2} \log\NNN\left(\BBB\frac{r_0}{2\sigma_\eps}\NNN\right)\BBB + Cr_0^2 
			- \frac{\pi}{m^2} \log\NNN\left(\BBB\frac{\sigma_\eps}{2\eps}\NNN\right)\BBB + C \\
		&\quad\leq \frac{\pi}{m^2} (2 \log(2) + \abs{\log r_0}) + C\NNN\left(\BBB1 + r_0^\frac{2}{3}\NNN\right)\BBB.
	\end{align*}
	Passing to coordinates, it follows by (\ref{comparison_glms}) that \NNN $\glmflat(\bar v_\eps, A_{\frac{\sigma_\eps}{2}, \sigma_\eps}(0)) \leq C$. \BBB
	\NNN Consequently, by \BBB Fubini's theorem \NNN and \BBB (\ref{close_to_zeps}) we can find $\tilde \sigma_\eps \in (\frac{\sigma_\eps}{2}, \frac{3\sigma_\eps}{4})$ such that
	\begin{align}
		\int_{\partial B_{\tilde \sigma_\eps}(0)} \frac{\abs{\bar v_\eps - \bar z_\eps}^2}{\sigma_\eps^2} + \abs{\nabla \bar v_\eps - \nabla \bar z_\eps}^2 \di \haus^1 &\leq \frac{C \delta^2}{\sigma_\eps} \label{close_to_zeps_on_slice} \\
		\int_{\partial B_{\tilde \sigma_\eps}(0)} \abs{\nabla \bar v_\eps}^2 + \abs{\nabla \abs{\bar v_\eps}}^2 + \frac{1}{\eps^2} (1 - \abs{\bar v_\eps}^2)^2 \di \haus^1 &\leq \frac{C}{\sigma_\eps} \label{glveps_bound_on_slice}.
	\end{align}
	Let $\theta(x)$ be the argument of $\frac{x}{\abs{x}}$ and \CCC let \BBB $\alpha_\eps \in \rr$ \CCC be \BBB such that $\lambda_\eps = e^{i \alpha_\eps}$.
	Note \CCC that $\bar z_\eps \BBB = e^{i(\theta + \alpha_\eps)}$\BBB.
	By \CCC Young's inequality \BBB we also obtain that
	\begin{equation}\label{length_close_to_1}
		\norm{\abs{\bar v_\eps} - 1}_{L^\infty(\partial B_{\tilde \sigma_\eps}(0))}^2 \leq C \eps \int_{\partial B_{\tilde \sigma_\eps}} \abs{\nabla \bar v_\eps}^2 + \frac{1}{\eps^2} (1 - \abs{\bar v_\eps}^2)^2 \di x \leq C \eps^{\frac{3}{4}}.
	\end{equation}
	Consequently, for $\eps$ small enough we have that $\rho_\eps \defas \abs{\bar v_\eps} \geq \frac{1}{2}$.
	\CCC By Lemma \ref{lem:ballcon_spec} (ii) we have that \BBB $\deg(v_\eps, \partial B_{\tilde \sigma_\eps}(0)) \CCC = 1 \BBB$.
	Therefore, \CCC since also $\deg(z_\eps, \partial B_{\sigma_\eps}) = 1$, by \BBB a standard lifting argument we can write $\bar v_\eps = \rho_\eps e^{i\theta_\eps}$, where $\theta_\eps - \theta \in H^1(\partial B_{\tilde \sigma_\eps}(0))$ and \CCC from (\ref{close_to_zeps_on_slice}) obtain \BBB
	\begin{equation}\label{angle_close_to_zeps_slice}
		\int_{\partial B_{\tilde \sigma_\eps}(0)} \frac{\abs{\bar \theta_\eps - (\theta + \alpha_\eps)}^2}{\sigma_\eps^2} + \abs{\nabla \theta_\eps - \nabla \theta}^2 \di \haus^1
		\leq C \int_{\partial B_{\tilde \sigma_\eps}(0)} \frac{\abs{\bar v_\eps - \bar z_\eps}^2}{\sigma_\eps^2} + \abs{\nabla \bar v_\eps - \nabla \bar z_\eps}^2 \di \haus^1
		\leq \frac{C \delta^2}{\sigma_\eps}.
	\end{equation}
	Let us extend $\rho_\eps$ and $\theta_\eps$ by zero homogeneity outside of $B_{\tilde \sigma_\eps}(0)$.
	Setting $\hat\sigma_\eps \defas \tilde \sigma_\eps + \eps^{\frac{3}{8}}$ we define $ \CCC \hat v_\eps \BBB $ in $B_{\hat \sigma_\eps}(0)$ through:
	\begin{equation*}
		 \CCC \hat v_\eps \BBB (x) \defas \begin{cases}
			\bar v_\eps(x) &\text{if } x \in B_{\tilde \sigma_\eps}(0), \\
			\CCC\left(\rho_\eps(x) \frac{\displaystyle \hat \sigma_\eps - \abs{x}}{\displaystyle \hat \sigma_\eps - \tilde \sigma_\eps} + \frac{\displaystyle \abs{x} - \tilde \sigma_\eps}{\displaystyle \hat \sigma_\eps - \tilde \sigma_\eps}\right) \BBB e^{i \theta_\eps(x)} &\text{if } x \in A_{\tilde \sigma_\eps, \hat \sigma_\eps}(0).
		\end{cases}
	\end{equation*}
	By (\ref{length_close_to_1}) and the definition $\hat \sigma_\eps$ we have for any $x \in A_{\tilde \sigma_\eps, \hat \sigma_\eps}(0)$:
	\begin{equation*}
		\abs{\nabla \abs{ \CCC \hat v_\eps \BBB }}^2
		= \abs*{\nabla \rho_\eps \frac{\hat \sigma_\eps - \abs{x}}{\hat \sigma_\eps - \tilde \sigma_\eps} + \frac{1 - \rho_\eps}{\hat \sigma_\eps - \tilde \sigma_\eps}}^2
		\leq \left(\abs{\nabla \rho_\eps} + \frac{\abs{1 - \rho_\eps}}{\eps^{\frac{3}{8}}}\right)^2
		\leq C (\abs{\nabla \rho_\eps}^2 + 1).
	\end{equation*}
	For the same $x$ as before we can similarly compute that
	\begin{equation*}
		\abs{\nabla  \CCC \hat v_\eps \BBB }^2 \leq C (\abs{\nabla (\rho_\eps e^{i \theta_\eps})}^2 + \abs{\nabla \rho_\eps}^2 + 1).
	\end{equation*}
	By the last two estimates, a change of coordinates, and (\ref{glveps_bound_on_slice}) we derive that
	\begin{align}
		&\glmflat( \CCC \hat v_\eps \BBB , A_{\tilde \sigma_\eps, \hat \sigma_\eps}(0)) \nonumber\\
		&\quad\CCC = \BBB C \int_{\tilde \sigma_\eps}^{\hat \sigma_\eps} \frac{r}{\tilde \sigma_\eps} \di r \int_{\partial B_{\tilde \sigma_\eps}(0)} \frac{1}{2m^2} \abs{\nabla \bar v_\eps}^2 + \frac{m^2 - 1}{2m^2} \abs{\nabla \abs{\bar v_\eps}}^2 + \frac{1}{4\eps^2} (1 - \abs{\bar v_\eps}^2)^2 \di \haus^1
		+ C \eps^{\frac{3}{8}}
		\leq C \frac{\eps^{\frac{3}{8}}}{\sigma_\eps} \leq C \eps^{\frac{1}{8}}. \label{glflat_first_annulus_bound}
	\end{align}
	Lastly, we extend $ \CCC \hat v_\eps \BBB $ into $A_{\hat \sigma_\eps, \sigma_\eps}(0)$ by linearly interpolat\NNN ing \BBB between $\theta_\eps$ and $\theta + \alpha_\eps$.
	More precisely, we set for $x \in A_{\hat \sigma_\eps, \sigma_\eps}(0\NNN)$
	\begin{equation*}
		 \CCC \hat v_\eps \BBB (x) = e^{i \hat \theta_\eps(x)}, \qquad \text{\NNN where } \hat \theta_\eps(x) \defas \frac{\sigma_\eps - \abs{x}}{\sigma_\eps - \hat \sigma_\eps} \theta_\eps\CCC(x)\BBB + \frac{\abs{x} - \hat \sigma_\eps}{\sigma_\eps - \hat \sigma_\eps} (\theta(x) + \alpha_\eps).
	\end{equation*}
	Let $x \in A_{\hat \sigma_\eps, \sigma_\eps}(0)$, as $\abs{ \CCC \hat v_\eps \BBB (x)} = 1$ we derive by Young's \CCC inequality and \BBB the definition of \NNN$\theta$\BBB
	\begin{align*}
		\abs{\nabla  \CCC \hat v_\eps \BBB (x)}^2
		&= \abs{\nabla \hat \theta_\eps(x)}^2
		= \abs*{
			\nabla \theta(x)
			+ \frac{\sigma_\eps - \abs{x}}{\sigma_\eps - \hat \sigma_\eps} (\nabla \theta_\eps(x) - \nabla \theta(x))
			+ \frac{\theta(x) + \alpha_\eps - \theta_\eps(x)}{\sigma_\eps - \hat \sigma_\eps}
		}^2 \\
		&\leq 
			(1 + 2\delta) \abs{\nabla \theta(x)}^2
			+ \NNN\left(\BBB 2 + \frac{1}{\delta}\NNN\right)\BBB\abs{\nabla \theta_\eps(x) - \nabla \theta(x)}^2
			+ \NNN\left(\BBB 2 + \frac{1}{\delta}\NNN\right)\BBB\frac{\abs{\theta_\eps(x) - (\theta(x) + \alpha_\eps)}^2}{(\sigma_\eps - \hat \sigma_\eps)^2} \\
		&\leq \frac{1}{\abs{x}^2} + \frac{C}{\delta} \left(
			\abs{\nabla \theta_\eps(x) - \nabla \theta(x)}^2
			+ \frac{\abs{\theta_\eps(x) - (\theta(x) + \alpha_\eps)}^2}{\sigma_\eps^2}
		\right) + \frac{C \delta}{\abs{x}^2}.
	\end{align*}
	\CCC As a result\BBB, by Fubini's theorem, a change of coordinates, and (\ref{angle_close_to_zeps_slice}) we \CCC get \BBB that
	\begin{align}
		&\int_{A_{\hat \sigma_\eps, \sigma_\eps}(0)} \abs{\nabla  \CCC \hat v_\eps \BBB }^2 \di x
		= \int_{\hat \sigma_\eps}^{\sigma_\eps} \int_{\partial B_r(0)} \abs{\nabla  \CCC \hat v_\eps \BBB }^2 \di \haus^1 \di r \nonumber \\
		&\quad\leq 
			\int_{\hat \sigma_\eps}^{\sigma_\eps} \frac{r}{\tilde \sigma_\eps} \di r \cdot \frac{C}{\delta} \int_{\partial B_{\tilde \sigma_\eps}(0)}
				\abs{\nabla \theta_\eps(x) - \nabla \theta(x)}^2
				+ \frac{\abs{\theta_\eps(x) - (\theta(x) + \alpha_\eps)}^2}{\sigma_\eps^2} \di \haus^1
			+2\pi \log(\frac{\sigma_\eps}{\hat \sigma_\eps}) \nonumber \\	
		&\quad\leq 2\pi \log\NNN\left(\BBB\frac{\sigma_\eps}{\hat \sigma_\eps}\NNN\right)\BBB + C \delta. \label{dirichlet_bound_tildeveps}
	\end{align}
	Using (\ref{close_to_zeps_on_slice}) we can show in a similar fashion that
	\begin{align*}
		\int_{A_{\hat \sigma_\eps, \sigma_\eps}(0)} \abs{\nabla \bar v_\eps}^2 \di x
		&\CCC = \BBB \int_{A_{\hat \sigma_\eps, \sigma_\eps}(0)} \abs{
			\nabla \bar z_\eps + \nabla \bar v_\eps - \nabla \bar z_\eps
		}^2 \di x \\
		&\geq (1 - \delta) \int_{A_{\hat \sigma_\eps, \sigma_\eps}(0)} \abs{\nabla \bar z_\eps}^2 \di x  - \frac{1}{\delta} \int_{A_{\hat \sigma_\eps, \sigma_\eps}(0)} \abs{\CCC \nabla \bar v_\eps - \BBB \nabla \bar z_\eps}^2 \di x \\
		&\geq 2\pi \log\NNN\left(\BBB\frac{\sigma_\eps}{\hat \sigma_\eps}\NNN\right)\BBB - C \delta.
	\end{align*}
	Note that our construction assures that $ \CCC \hat v_\eps \BBB  = z_\eps$ on $\partial B_{\sigma_\eps}(0)$.
	Hence, by (\ref{glflat_first_annulus_bound}), (\ref{dirichlet_bound_tildeveps}), the above estimate, and (\ref{core_convergence_flat_version2}) we derive that
	\begin{align*}
		\glmflat(\bar v_\eps, B_{\sigma_\eps}(0))
		&= \glmflat( \CCC \hat v_\eps \BBB , B_{\sigma_\eps}(0)) + \glmflat(\bar v_\eps, A_{\tilde \sigma_\eps, \sigma_\eps}(0)) - \glmflat( \CCC \hat v_\eps \BBB , A_{\tilde \sigma_\eps, \sigma_\eps}(0)) \\
		&\geq
			\bar \gamma_\eps^{(m)}(\sigma_\eps, \lambda_\eps)
			- \glmflat( \CCC \hat v_\eps \BBB , A_{\tilde \sigma_\eps, \hat \sigma_\eps}(0))
			+ \frac{1}{2m^2} \int_{A_{\hat \sigma_\eps, \sigma_\eps}(0)} 
				\abs{\nabla \bar v_\eps}^2 - \abs{\nabla  \CCC \hat v_\eps \BBB }^2 \di x \\
		&\geq \frac{\pi}{m^2} \log(\frac{\sigma_\eps}{\eps}) + \gamma_m - C (\delta + \eps^{\frac{1}{8}}) + \littleo(1),
	\end{align*}
	which is (\ref{required_lowb_step5_flat}).
\end{proof}

\subsection{\NNN$\Gamma$\BBB-limsup}
The goal in this \CCC s\BBB ection is the construction of \CCC the \BBB recovery sequence in Theorem \ref{thm:gamma_convergence} \NNN (\BBB iii).

\CCC In the next lemma we \BBB relate the non-fractional renormalized energy on a surface to the Euclidean one:

\begin{lemma}\label{lem:renergy_flat_vs_surface}
	Given $x_0 \in S$,  $r_0 \in (0, \min\setof{1, r^*})$, and $v \in W^{1, 2}_\loc(B_{r_0}(x_0) \setminus \setof{x_0}; \sph^1) \cap W^{1, 1}(TB_{r_0}(x_0))$ such that $\vort(v) = d \dirac_{x_0}$, where $d \in \setof{-1, 1}$, and $\ren(v, \CCC B_{r_0}(x_0) \BBB) < \infty$, with $\ren(v, \CCC B_{r_0}(x_0) \BBB)$ as in \eqref{def_renergy_nonfrac}.
  Then,
	\begin{equation}\label{dyadic_annulus_convergence}
		\lim_{k \to \infty} \frac{1}{2} \int_{A_k} \abs{\nabla v}^2 \vol = \pi \log(2),
	\end{equation}
	where $A_k \defas A_{2^{-(k+1)}r_0, 2^{-k}r_0}(x_0)$.
	Further, given a trivialization $\Psi$ of $TB_{r_0}(x_0)$ as described in Remark \ref{rem:coordinate_representation_ball} and $\bar v \defas \Psi^* v$ we have that \CCC the Euclidean renormalized energy of $\bar v$ in $B_{r_0}(0)$ \NNN is finite: \BBB
	\begin{equation*}
		\bar \ren(\bar v, B_{r_0}(0)) \defas \lim_{r \to 0} \left(
      \frac{1}{2} \int_{A_{r, r_0}(0)} \abs{\nabla \bar v}^2 \di x - \pi \abs{\log(r)}
    \right)\CCC < \infty. \BBB
	\end{equation*}
\end{lemma}

\begin{proof}
	\textit{\NNN Step 1 (P\CCC roof \BBB of (\ref{dyadic_annulus_convergence})):}
  Note that we can write the renormalized energy as a series as follows:
	\begin{equation*}
		\ren(v, \CCC B_{r_0} \BBB)
		= \sum_{k = 0}^\infty \left( \frac{1}{2} \int_{A_k} \abs{\nabla v}^2 \vol - \pi \log\left(\frac{2^{-k}r_0}{2^{-(k+1)}r_0}\right) \right)
		= \sum_{k = 0}^\infty \left( \frac{1}{2} \int_{A_k} \abs{\nabla v}^2 \vol - \pi \log(2) \right)\CCC,\BBB
	\end{equation*}
	\CCC where we set $B_{r_0} \defas B_{r_0}(x_0)$. \BBB
	Using $\abs{v} = 1$ a.e.~in $B_{r_0}(x_0)$ and $\vort(v) = d \dirac_{\CCC x\BBB_0}$ with $\abs{d} = 1$, \CCC by a standard convolution argument and \BBB (\ref{ener_lowbound_circles}) \CCC we have \BBB that for every $0 < r < R \leq r_0$ and $\eps \in (0, r\NNN)$
  \begin{equation*}
    \frac{1}{2} \int_{A_{r, R}(x_0)} \abs{\nabla v}^2 \vol
    \geq \int_r^R \frac{\pi(1 - C\CCC t \BBB^2)}{\CCC t \BBB + C\eps} \CCC \di t \BBB
    \geq \pi \log\left(\frac{R + C\eps}{r + C\eps}\right)
      - C (R^2 - r^2),
  \end{equation*}
  where $C$ is a constant independent of $\eps$, $r$, and $R$.
  \NNN Passing to the limit $\eps \to 0$ then leads to \BBB
  \begin{equation*}
    \frac{1}{2} \int_{A_{r, R}(x_0)} \abs{\nabla v}^2 \vol
    \geq \pi \log\left(\frac{R}{r}\right)
      - C (R^2 - r^2).
  \end{equation*}
 	\CCC As a consequence\BBB, we \CCC have \BBB that for every $k \in \nn$
	\begin{equation*}
		\frac{1}{2} \int_{A_k} \abs{\nabla v}^2 \vol 
		\geq \pi \log\left( \frac{2^{-k}r_0}{2^{-(k+1)}r_0} \right) - C (2^{-2k} r_0^2 - 2^{-2(k+1)}r_0^2)
		= \pi \log(2) - C 2^{-2(k+1)} r_0^2.
	\end{equation*}
  Hence,
	\begin{equation*}
		\sum_{k = 0}^\infty \CCC\left(\BBB
      \frac{1}{2} \abs{\nabla v}^2 \vol 
      - \pi \log(2)
      + C 2^{-2(k+1)} r_0^2 \CCC\right)\BBB
		\leq \ren(v, B_{r_0}(x_0)) + C r_0^2 < \infty\CCC.\BBB
	\end{equation*}
	\CCC Since \BBB each term of the series is nonnegative we get \BBB (\ref{dyadic_annulus_convergence}).

	\textit{\NNN Step 2 (Finiteness of the Euclidean renormalized energy)\BBB:}
  From the properties of $v$ in the statement and \NNN our choice of $\Psi$ \BBB we derive that $\bar v \in W^{1, 2}_\loc(B_{r_0}(0) \setminus \setof{0}; \sph^1)$ and $\vort(v) = d \dirac_0$.
	By \ref{dyadic_annulus_convergence} \NNN there exists \BBB $K_0 \in \nn$ big enough such that for any $k \geq \NNN K_0$
	\begin{equation*}
		\frac{1}{2} \int_{A_k} \abs{\nabla v}^2 \vol \leq 3 \pi \log(2).
	\end{equation*}
  Setting $\bar A_k \defas A_{2^{-(k+1)}r_0, 2^{-k}r_0}(0)$ and using (\ref{comparison_L2_grad_norms}) we therefore \CCC obtain that \BBB for every $k \geq K_0$
	\begin{align*}
		\frac{1}{2} \int_{\bar A_k} \abs{\nabla \bar v}^2 \di x - \pi \log(2)
		&\leq \frac{1}{2} \int_{A_k} \abs{\nabla v}^2 \vol - \pi \log(2) + C 2^{-k}r_0 \left(1 + \int_{A_k} \abs{\nabla v}^2 \vol\right) \\
		&\leq \frac{1}{2} \int_{A_k} \abs{\nabla v}^2 \vol - \pi \log(2) + C 2^{-k} r_0 (1 + 3 \pi \log(2)).
	\end{align*}
	\CCC As a consequence, we have that \BBB
  \begin{align*}
    \ren(\bar v, B_{r_0}(0))
    &= \frac{1}{2} \int_{A_{2^{-K_0}r_0, r_0}(0)}
        \abs{\nabla \bar v}^2 \di x \NNN - K \pi \log(2) \BBB
      + \sum_{k = K_0}^\infty \left(
        \frac{1}{2} \int_{\bar A_k} \abs{\nabla \bar v}^2 \di x 
        - \pi \log(2) \right) \\
    &\leq \frac{1}{2} \int_{A_{2^{-K_0}r_0, r_0}(0)}
        \abs{\nabla \bar v}^2 \di x \NNN - K \pi \log(2) \BBB
      + \ren(v, B_{2^{-K_0}r_0}(x_0))
      + C r_0 < \infty,
  \end{align*}
  as desired.
\end{proof}

The next lemma will be useful for the construction of a recovery sequence \NNN outside \BBB of vortices.
\begin{lemma}[Approximation outside cores]\label{lem:approx_outside_cores_gengl}
	Let $u \in \lsm(S)$ \CCC with $\ren(u, B_{r_0}(x_0) < \infty$ where \BBB $x_0 \in S$ \CCC is \BBB one of the vortex centers of $u$ and $r_0 \in (0, r^*)$ \CCC is \BBB chosen sufficiently smal\NNN l t\BBB hat $B_{r_0}(x_0) \cap \spt(\vort(u)) = \setof{x_0}$.
  Further\CCC more\BBB, given an orthonormal frame $\setof{\CCC\tau\BBB, i\CCC\tau\BBB}$ with $\CCC \tau \BBB \in C^\infty(B_{r_0}(x_0))$, let $\Psi$ be a local trivialization of $TB_{r_0}(x_0)$ as described in Remark \ref{rem:coordinate_representation_ball}.
  Then, for any $\delta > 0$ we can find $r \in (0, r_0)$, $\lambda_r \in \sph^1$, and $u_* \in SBV^2(A_{\frac{r}{2}, r_0}(x_0); \sph^1)$ such that
	\begin{enumerate}[label=(\roman*)]
    \item $\power(u_*^+) = \power(u_*^-)$ at $\haus^1_g$-a.e.~point on $\mathcal J_{u_*}$;
		\item $u_* = u$ in $A_{r, r_0}(x_0)$;
		\item $\bar u_*^m = \lambda_r \frac{x}{\abs{x}}$ on $\partial B_{\frac{r}{2}}(0)$, where $\bar u_* \defas \Psi^* u_*$ and we identified $\rr^2$ with $\cc$;
		\item $\norm{u - u_*}_{SBV^2(TA_{\frac{r}{2}, r}(x_0))} \leq \delta$.
	\end{enumerate}
\end{lemma}

\begin{proof}
  In $B_{r_0}(x_0)$ we define $v \defas \power(u)$.
  Further\CCC more\BBB, let us set $\bar u \defas \Psi^* u$ and $\bar v \defas \Psi^* \CCC v \BBB$.
  Note that by our choice of coordinates we \CCC have \BBB that $\bar v = \bar u^m \in W^{1, 2}_\loc(B_{r_0}(0) \setminus \setof{0}; \sph^1)$.
  By Lemma \ref{lem:renergy_flat_vs_surface} \NNN it holds that \BBB $\ren(\bar v, B_{r_0}(0)) < \infty$.
  Give\CCC n \BBB $\CCC\delta\BBB' > 0$\CCC, by a standard cut-off argument in the Euclidean setting (see e.g.~the proof of Lemma 3.15 in \cite{goldman_glflat}, \BBB we can find by $r \in (0, r_0)$, $\lambda_r \in \sph^1$, and $\bar u_* \in SBV^2(A_{\frac{r}{2}, r_0}(0); \sph^1)$ such that
	\begin{enumerate}[label=(\alph*)]
		\item $(\bar u_*^+)^m = (\bar u_*^-)^m$ at $\haus^1$-a.e.~point on $\mathcal J_{\bar u_*}$;
		\item $\bar u_* = \bar u$ in $A_{r, r_0}(0)$;
		\item $\bar u_*^m = \lambda_r \frac{x}{\abs{x}}$ on $\partial B_{\frac{r}{2}}(0)$\NNN; \BBB
		\item $\norm{\bar u_* - \bar u}_{SBV^2(A_{\frac{r}{2}, r}(0))} \leq \delta'$\NNN.\BBB
	\end{enumerate}
	Let $u_*$ be the vector field on $B_{r_0}(x_0)$ such that $\Psi^* u_* = \bar u_*$.
  By (b) and (c) this choice of $u_*$ trivially satisfies (ii) and (iii) from the statement, respectively.
  As we chose an orthonormal frame, which preserves angles we also see that (i) follows from (a).
  Lastly, by \CCC Lemma \ref{lem:sbv_in_coords} \BBB
  \begin{equation*}
    \norm{u - u^*}_{SBV^2(TA_{\frac{r}{2}, r}(x_0))} \leq C \delta',
  \end{equation*}
  for a constant only depending on $S$.
  Hence, choosing $\delta'$ small enough (iv) follows.
\end{proof}

The next lemma will be employed in the construction of the recovery sequence in the vicinity \NNN of \BBB vortices.
\begin{lemma}[Approximation inside cores]\label{lem:approx_cores_gengl}
    Let $x_0 \in S$, $r \in (0, r^*)$, $\lambda \in \sph^1$\CCC, \BBB $\delta > 0$, and $\Psi$ be a local trivialization of $TB_r(x_0)$ as described in Remark \ref{rem:coordinate_representation_ball}.
		Then, there exists sequence $(u_\eps) \subset \asm(B_r(x_0))$ such that
	\begin{enumerate}[label=(\roman*)]
		\item $\Psi^* u_\eps = \lambda \frac{x}{\abs{x}}$ on $\partial B_r(0)$;
		\item $\gl(u_\eps, B_r(x_0)) - \frac{\pi}{m^2} \log\CCC\left(\BBB\frac{r}{\eps}\CCC\right)\BBB  \leq \gamma_m + \delta + \littleo_r(1) + \littleo(1)$;
	\end{enumerate}
\end{lemma}

\begin{proof}
	By Corollary 3.11 in \cite{goldman_glflat}, we can find $\eps_0 \in (0, 1)$ and $\bar z \in SBV^2(B_1(0); \sph^1) \CCC \cap L^\infty(B_1(0))\BBB$ such that $(\bar z^+)^m = (\bar z^-)^m$ at $\haus^1$-a.e.~point in $\mathcal J_{\bar z}$, $\bar z^m = \CCC x\BBB$ on $\partial B_1(0)$, and
	\begin{equation}\label{gleps0_barz_unitball}
		\overline{GL}_{\eps_0}^{(m)}(\bar z^m, B_1(0)) \leq \frac{\pi}{m^2} \abs{\log \eps_0} + \gamma_m + \delta.
	\end{equation}
	For any $\eps \in (0, r \eps_0)$ we then define
	\begin{equation*}
		\bar u_\eps(x) \defas \begin{cases}
			\CCC \lambda \BBB \bar z(\frac{\eps_0}{\eps} x) &\text{if } \abs{x} \leq \frac{\eps}{\eps_0} \\
			e^{i \frac{\arg(x) + \alpha}{m}} &\text{if } \frac{\eps}{\eps_0} < \abs{x} \leq r,
		\end{cases}
	\end{equation*}
	where $\alpha$ is such that $\lambda = e^{i \alpha}$.
	Let $u_\eps$ be the vector field on $B_r(x_0)$ such that $\Psi^* u_\eps = \bar u_\eps$\CCC, hence (i) follows\BBB.
	By our choice of \CCC trivialization \BBB it also follows that $u_\eps \in \asm(B_r(x_0)\CCC)\BBB$.
	It remains to show (ii).
	Setting $\bar v_\eps \defas \bar u_\eps^m$ a change of coordinates \CCC gives \BBB for $p \in [1, 2)$ \CCC that \BBB
	\begin{align}
		\int_{B_r(0)} \abs{\nabla \bar v_\eps}^p \di x
		&= \int_{B_{\frac{\eps}{\eps_0}}(0)} 
				\left(\frac{\eps_0}{\eps}\right)^p
				\NNN\abs*{\BBB\nabla \bar z^m\NNN\left(\BBB\frac{\eps_0}{\eps}x\NNN\right)\BBB\NNN}^p\BBB \di x
			+ \int_{A_{\frac{\eps}{\eps_0}\NNN, r\BBB}(0)} \NNN\abs*{\BBB\nabla \frac{x}{\abs{x}}\NNN}^p\BBB \di x \nonumber \\
		&\NNN\leq\BBB \int_{B_1(0)} \abs{\nabla \bar z^m}^p \di x + \frac{2\pi}{2 - p} \NNN r^{2 - p} \BBB < \infty. \label{nabla_bar_veps_bound}
	\end{align}
	\CCC In order to show (ii) we claim that it suffices to prove \BBB
	\begin{equation}\label{better_statement_approx_cores}
		\glm(v_\eps, B_r(x_0)) - \frac{\pi}{m^2} \log\CCC\left(\BBB\frac{r}{\eps}\CCC\right)\BBB  + \haus^1_g(\mathcal J_{u_\eps}) \leq \gamma_m + \delta + \littleo_r(1) + \littleo(1)\CCC, \BBB
	\end{equation}
	\CCC where $v_\eps = \power(u_\eps)$.
	In fact, assume that the claim \eqref{better_statement_approx_cores} is proved, then (ii) follows by \eqref{glm_relation} provided we show that $\int_{B_r} \sprod{\jac(u_\eps)}{\jac(\tau)} \vol = \BigO_r(1)$.
	This last inequality follows by the definition of $\jac(u_\eps)$ and \eqref{nabla_bar_veps_bound}. Let us now prove the claim. \BBB
	By construction \CCC we have \BBB that
	\begin{equation*}
		\mathcal J_{\bar u_\eps} \subset \frac{\eps}{\eps_0} \mathcal J_{\bar z} \cup \partial B_{\frac{\eps}{\eps_0}}(0) \cup \setof*{(t, 0) \colon t \in [-r, -\NNN\eps/\eps_0\BBB]}.
	\end{equation*}
	Therefore,
	\begin{equation*}
		\haus^1(\mathcal J_{\bar u_\eps}) \leq \frac{\eps}{\eps_0} \haus^1(\mathcal J_{\bar z}) + \frac{2\pi \eps}{\eps_0} + r,
	\end{equation*}
	which by the equivalence of norms shows that
	\begin{equation}\label{haus1g_jump_ueps_bound}
		\haus^1_g(\mathcal J_{u_\eps}) \leq \littleo_r(1) + \littleo(1).
	\end{equation}
	Further\CCC more\BBB, by changing coordinates and (\ref{gleps0_barz_unitball}) we derive that
	\begin{align}
		\glmflat(\bar v_\eps, B_r(0))
		&= \frac{1}{2m^2} \int_{B_{\frac{\eps}{\eps_0}}(0)} 
			\abs{\nabla \bar v_\eps}^2
			+ (m^2 - 1) \abs{\nabla \abs{\bar v_\eps}}^2
			+ \frac{m^2}{2\eps^2} (1 - \abs{\bar v_\eps}^2)^2 \di x
		+ \frac{\pi}{m^2} \log\left(\frac{r\eps_0}{\eps}\right) \nonumber \\
		&= \overline{GL}_{\eps_0}^{(m)}(\bar z^m, B_1(0)) - \frac{\pi}{m^2} \abs{\log \eps_0} + \frac{\pi}{m^2} \log\CCC\left(\BBB\frac{r}{\eps}\CCC\right)\BBB
		\leq \frac{\pi}{m^2} \log\CCC\left(\BBB\frac{r}{\eps}\CCC\right)\BBB  + \gamma_m + \delta. \label{glm_barveps_bound}
	\end{align}
	Denoting by $K_0$ \CCC the \BBB largest natural number satisfying $2^{-K_0}r \geq \frac{\eps}{\eps_0}$ we define for any $k \in \setof{0, \dots, K_0}$:
	\begin{equation*}
		\bar A_k \defas \begin{cases}
			A_{2^{-(k+1)}r, 2^{-k}r}(0) &\text{if } k < K_0, \\
			A_{\frac{\eps}{\eps_0}, 2^{-K_0}r}(x_0) &\text{if } k = K_0.
		\end{cases}
	\end{equation*}
	With this notation, (\ref{glm_barveps_bound}), and (\ref{comparison_glm_glmflat}) \NNN it follows that \BBB
	\begin{align*}
		\glm(v_\eps, B_r(x_0))
		&\leq \NNN\left(\BBB 1 + C \frac{\eps}{\eps_0}\NNN\right)\BBB
			\glmflat(\bar v_\eps, B_{\frac{\eps}{\eps_0}}(0)) + C \frac{\eps}{\eps_0}
		+ \sum_{k = 0}^{K_0} (1 + C 2^{-k} r) \glmflat(\bar v_\eps, \bar A_k) + C 2^{-k} r \\
		&= \glmflat(\bar v_\eps, B_r(0)) + C \eps \leps + \sum_{k = 0}^{K_0} C 2^{-k} r \pi \log(2) \\
		&\leq \frac{\pi}{m^2} \log\CCC\left(\BBB\frac{r}{\eps}\CCC\right)\BBB  + \gamma_m + \delta + \littleo_r(1) + \littleo(1)\CCC, \BBB
	\end{align*}
	\CCC which proves the claim once combined with \eqref{haus1g_jump_ueps_bound}.
\end{proof}

The $\Gamma$\CCC-\BBB$\limsup$ is then a straightforward consequence of Lemma \ref{lem:approx_outside_cores_gengl} and Lemma \ref{lem:approx_cores_gengl}.
\begin{proof}[Proof of Theorem \ref{thm:gamma_convergence} \NNN (iii) \BBB]
	Let \NNN $N \defas \abs{\vort(u)}$ and \BBB $\mu \defas \vort(u) = \sum_{k=1}^{Nm} s_k \dirac_{x_k}$.
	Let us further take $\delta > 0$ and $r \in (0, r_0)$, where $r_0 \in (0, r^*)$ is small enough such that the balls $\setof{B_{r_0}(x_k)}_k$ are disjoint.
	Fix \NNN $k$, a\BBB pplying Lemma \ref{lem:approx_outside_cores_gengl} in the ball $B_{r_0}(x_k)$ for $\delta$ and $r$ as above shows the existence of $u_*^{(k)} \in \asm(B_{r_0}(x_k))$ satisfying
	\begin{gather}
		u_*^{(k)} = u \text{ in } A_{r, r_0}(x_k), \label{ustar_equal_u}\\
		(\bar u_*^{(k)})^m = \lambda_r^{(k)} \frac{x}{\abs{x}} \text{ on } \partial B_{\frac{r}{2}}(0), \label{ustar_bdry_values}\\
		\norm{u - u_*^{(k)}}_{SBV^2(TA_{\frac{r}{2}, r}(x_0))} \leq \delta, \nonumber 
	\end{gather}
	where $\lambda_r^{(k)} \in \sph^1$ and $\bar u_*^{(k)} \defas \Psi^* u_*^{(k)}$ where $\Psi$ is the trivialization of $TB_r(x_k)$ from Lemma \ref{lem:approx_outside_cores_gengl}.
	Further\CCC more\BBB, using Lemma \ref{lem:approx_cores_gengl} in $B_r(x_k)$ for $\lambda = \lambda_r^{k}$ we can find a sequence $(u_\eps^{(k)})\NNN_\eps\BBB \subset \asm(B_r(x_k))$ such that $\sup_{\eps > 0} \norm{u_\eps^{(k)}}_{L^\infty(TB_r(x_k))} \leq C$ and
	\begin{align}
		\bar u_\eps^{(k)} &= \lambda_r^{(k)} \frac{x}{\abs{x}} \text{ on } \partial B_r(0), \label{uepsk_bdry_values} \\
		\gl(u_\eps^{(k)}, B_r(x_k)) - \frac{\pi}{m^2} \log\CCC\left(\BBB\frac{r}{\eps}\CCC\right)\BBB  &\leq \gamma_m + \delta + \littleo_r(1) + \littleo(1), \label{uepsk_glmsup}
	\end{align}
	where $\bar u_\eps^{(k)} \defas \Psi^* u_\eps^{(k)}$.
	We then define
	\begin{equation*}
		u_\eps(x) \defas \begin{cases}
			u(x) &\text{if } x \in S_{r_0}(\mu), \\
			u_*^{(k)} &\text{if } x \in A_{\frac{r}{\CCC 2\BBB}, r_0}(x_k), \\
			u_\eps^{(k)} &\text{if } x \in B_{\frac{r}{\CCC 2\BBB}}(x_k).
		\end{cases}
	\end{equation*}
	Note that from (\ref{ustar_equal_u}), (\ref{ustar_bdry_values}), and (\ref{uepsk_bdry_values}) \NNN we derive \BBB that $u_\eps \in \asm(S)$.
	Due to (\ref{ustar_equal_u}) and $\abs{u_\eps} \leq 1$ a.e.~in $\Omega$ we also have that
	\begin{equation*}
		\sup_\eps \norm{u - u_\eps}_{L^1(TS)} \leq \CCC C \BBB \sum_{k=1}^{Nm} \haus^2_g(B_r(x_k)) \leq C r^2.
	\end{equation*}
	Finally, by (\ref{ustar_equal_u}), the definition of $\renm(u)$, and (\ref{uepsk_glmsup}) it holds for any $r$ that
	\begin{align*}
		\glm(u_\eps, S) - \frac{N}{m} \pi \leps
		&= \frac{1}{2} \int_{S_r(\mu)} \abs{\nabla u}^2 \vol - \frac{N}{m} \abs{\log r} + \haus^1_g(\mathcal J_u \cap S_r(\mu)) \NNN + \sum_{k = 1}^{Nm} \haus^1_g(\mathcal J_{u_\eps} \cap \partial B_r(x_k)) \BBB \\
		&\phantom{=}\quad + \gl(u_\eps^{(k)}, B_r(x_k)) - \frac{\pi}{m^2} \log\CCC\left(\BBB\frac{r}{\eps}\CCC\right)\BBB  \\
		&\leq \renm(u) + \haus^1_g(\mathcal J_u) + Nm \gamma_m + Nm \delta + \littleo_r(1) + \littleo(1).
	\end{align*}
	Hence, a desired recovery sequence can be found by a standard diagonal sequence argument.
\end{proof}

\section*{Appendix}
\appendix
\section{Proof of the decomposition theorem}\label{sec:proofdecomp}
In this appendix we provide the missing proofs of the statements in Section \ref{sec:bvsecs} using the same notation provided therein.
\NNN
While the results in Section \ref{sec:bvsecs} are concerned with the special case of the tangent bundle, for the sake of generality, we will deal in this appendix with a general metric vector bundle $E$ over $M$ with rank $m$.
The induced covariant derivative on $E$ will still be denoted by $\nabla$ and its adjoint by $\nabla^*$.
As previously we assume Einstein summation convention where latin indices such as $i, \, j, \, k, \ldots$ and Greek indices such as $\alpha, \, \beta, \, \gamma, \ldots$ that appear multiple times are implicitly summed over $\setof{1, \ldots, n}$ and $\setof{1, \ldots, m}$, respectively.
Notice that an analog of \eqref{ibp_adjoint_grad} remains true in the present setting for $u \in C^\infty_c(E)$ and $v \in C^\infty_c(E \otimes T^* M)$.
Given $u \in L^1_\loc(E)$ and $O \subset M$ open we define the total variation of $u$ in $O$ as follows:
\begin{equation*}
  \var(u, O) \defas \sup\setof*{\int_M \sprod{u}{\nabla^* v} \vol \colon v \in C^\infty_c(E|_O \otimes T^* O)}.
\end{equation*}
A section $u \in L^1(E)$ is then said to have bounded variation, shortly writing $u \in BV(E)$, if and only if $\var(u) \defas \var(u, M) < \infty$.
Riesz representation also holds in the more general setting.
In fact, for a bounded linear functional $T \colon C_c(E) \to \rr$ there exist a unique $E \otimes T^*\!M$-valued Radon measure $\nu$ such that
\begin{equation*}
  T(v) = \int_M \sprod{v}{\sigma_\nu} \di \abs{\nu},
\end{equation*}
where $\abs{\nu}$ and $\sigma_\nu$ are the polar density and total variation of $\nu$, respectively.
\BBB

\NNN
The following theorem is a generalization of Theorem \ref{thm:radon_nikodym} to the case of general vector bundles.
\begin{theorem}[Radon-Nikodym]\label{thm:radon_nikodym_gen}
	For any $\nu \in \radon(E \otimes T^*\!M)$ and $\mu \in \radon_+(M)$ there exist only two measures $\nu^a, \, \nu^s \in \radon(E \otimes T^*\!M)$ such that $\nu^a << \mu, \, \nu^s \perp \mu$ and $\nu = \nu^a + \nu^s$.

	Furthermore, there exists a unique $\sigma^a \in L^1(E \otimes T^*\!M; \mu)$ such that $\nu^a = \sigma^a \mu$.
\end{theorem}
\BBB

\begin{proof}[Proof of Theorem \NNN \ref{thm:radon_nikodym_gen}\BBB]
  \textit{\NNN Step 1 (E\BBB xistence):} We start by showing existence of $\nu^a$ and $\nu^s$.
  Let $\abs{\nu}$ be the total variation of $\nu$ and $\polden{\nu}$ its polar density.
  As $\abs{\nu}$ is a scalar Radon measure we can apply the classical Radon-Nikodym theorem (see for example Theorem 1.28 in \cite{fusco_bible}) to the pair $\abs{\nu}$, $\mu$.
  \NNN Therefore, \BBB we can find positive Radon measures $\abs{\nu}^a$, $\abs{\nu}^s$ such that $\abs{\nu} = \abs{\nu}^a + \abs{\nu}^s$, $\abs{\nu}^a << \mu$, and $\abs{\nu}^s \perp \mu$.
  Furthermore, there exists $f \in L^1(M; \mu)$ such that $\abs{\nu}^a = f \mu$.
  Let us now set $\nu_1 \defas \polden{\nu} \abs{\nu}^a$ \NNN and \BBB $\nu_2 \defas \polden{\nu} \abs{\nu}^s$; then,
  \begin{equation*}
    \nu = \polden{\nu} \abs{\nu} = \polden{\nu} (\abs{\nu}^a + \abs{\nu}^s) = \nu_1 + \nu_2.
  \end{equation*}
  Furthermore, by $\abs{\sigma_\nu} = 1$ \NNN $\abs{\nu}$-a.e.\BBB, it follows that $\abs{\nu_1} = \abs{\nu}^a << \mu$ and $\abs{\nu_2} = \abs{\nu}^s \perp \mu$.
  With the boundedness of $\sigma_\nu$ we also derive that $f\polden{\nu} \in L^1(E; \mu)$.
  Thanks to that, the measures $\nu^1$ and $\nu^2$ are admissible candidates for $\nu^a$ and $\nu^s$, respectively.

  \textit{\NNN Step 2 (U\BBB niqueness):} It remains to prove the uniqueness of $\nu_1$ and $\nu_2$ found in the previous step.
  Let $\tilde \nu_1 \, \tilde \nu_2 \in \radon(E)$ be such that $\nu = \tilde \nu_1 + \tilde \nu_2$, $\tilde \nu_1 << \mu$, and $\tilde \nu_2 \perp \mu$.
  Our first aim is to show $\abs{\nu} = \abs{\tilde \nu_1} + \abs{\tilde \nu_2}$.
  Given an open bounded set $O \subset M$ we consider a sequence $(v_h) \subset C_c^\infty(E|_O)$ with $\norm{v_h}_{L^\infty} \leq 1$ that converges in $L^1(E|_O; \abs{\nu})$ towards $\polden{\nu} \indic{O}$.
  Then, by the dominated convergence theorem \NNN we have that \BBB
  \begin{equation*}
    \abs{\tilde \nu_1}(O) + \abs{\tilde \nu_2}(O)
    \geq \int_M \sprod{v_h}{\polden{\tilde \nu_1}} \di\abs{\tilde\nu_1} + \int_M \sprod{v_h}{\polden{\tilde \nu_2}} \di\abs{\tilde\nu_2}
    = \int_M \sprod{v_h}{\polden{\nu}} \di\abs{\nu} \to \abs{\nu}(O).
  \end{equation*}
  By the arbitrariness of $O$ \NNN it follows that \BBB $\abs{\nu} \leq \abs{\tilde \nu_1} + \abs{\tilde \nu_2}$.
  \NNN Let us \BBB us investigate the reverse inequality.
  As $\abs{\tilde \nu_2} \perp \mu$ \NNN there exists \BBB a Borel-set $B$ such that $\abs{\tilde \nu_2}(M \setminus B) = \mu(B) = 0$ and as $\abs{\tilde \nu_1} << \mu$ we also have \NNN that \BBB $\abs{\tilde \nu_1}(B) = 0$.
  Let $O$ be as before and $(v_h) \subset C_c^\infty(E|_O)$ with $\norm{v_h}_{L^\infty} \leq 1$ converging in $L^1(E|_O; \abs{\nu})$ towards $\polden{\tilde \nu_1} \indic{O \setminus B} + \polden{\tilde \nu_2} \indic{B}$, then
  \begin{align*}
    \abs{\nu}(O)
    \geq \int_M \sprod{v_h}{\polden{\nu}} \di\abs{\nu}
    = \int_M \sprod{v_h}{\polden{\tilde \nu_1}} \di\abs{\tilde\nu_1} + \int_M \sprod{v_h}{\polden{\tilde \nu_2}} \di\abs{\tilde\nu_2}
    &\to \abs{\tilde \nu_1}(O \setminus B) + \abs{\tilde \nu_2}(B) \\
    &= \abs{\tilde \nu_1}(O) + \abs{\tilde \nu_2}(O),
  \end{align*}
  By the arbitrariness of $O$ leads to $\abs{\tilde \nu_1} + \abs{\tilde \nu_2} \leq \abs{\nu}$.
  With the uniqueness of the absolutely continuous and singular part in the classical setting we see that $\abs{\nu_1} = \abs{\tilde \nu_1}$ and $\abs{\nu_2} = \abs{\tilde \nu_2}$ in the sense of measures.
  It remains to show the uniqueness of $\sigma_{\nu_1}$ and $\sigma_{\nu_2}$.
  Let $O$ be an open bounded set and $B$ a Borel set such that $\abs{\nu_1}(B) = \abs{\nu_2}(O \setminus B) = 0$.
  Further\NNN more\BBB, let us take a sequence $(v_h) \subset C_c^\infty(E|_O)$ converging towards $\polden{\tilde \nu_1} \indic{O \setminus B}$ in $L^1(E|_O; \abs{\nu})$, then
  \begin{align*}
    \abs{\nu_1}(O \setminus B)
    \geq \int_{O \setminus B} \sprod{\polden{\tilde \nu_1}}{\polden{\nu_1}} \di \abs{\nu_1}
    &= \lim_{h \to \infty} \int_M \sprod{v_h}{\polden{\nu_1}} \di\abs{\nu_1} + \int_M \sprod{v_h}{\polden{\nu_2}} \di\abs{\nu_2}\\
    &= \lim_{h \to \infty} \int_M \sprod{v_h}{\polden{\tilde \nu_1}} \di\abs{\tilde \nu_1} + \int_M \sprod{v_h}{\polden{\tilde \nu_2}} \di\abs{\tilde \nu_2} \\
    &\geq \abs{\tilde \nu_1}(O \setminus B) - \abs{\tilde \nu_2}(O \setminus B)
    = \abs{\nu_1}(O \setminus B),
  \end{align*}
  where in the last equality we have used the uniqueness of the decomposition of $\abs{\nu}$ proved in the previous step.
  Hence, the first inequality above must be an equality, which can only hold true if $\sigma_{\tilde \nu_1} = \sigma_{\nu_1}$ at $\abs{\nu_1}$-a.e.~point.
  The uniqueness of $\polden{\nu_2}$ follows similarly.
\end{proof}

Our next goal is the proof of Proposition \NNN\ref{prop:approx_cont_coords_gen}\BBB, Proposition \NNN\ref{prop:approx_jump_coords_gen}\BBB, and Proposition \NNN\ref{prop:approx_diff_coords_gen} which are generalizations of Proposition \ref{prop:approx_cont_coords}, Proposition \ref{prop:approx_jump_coords}, and Proposition \ref{prop:approx_diff_coords} to the case of a general vector bundle, respectively\BBB.
\NNN We remark that Definition \ref{def:approximate_limit}, \ref{def:approximate_jump}, and \ref{def:approximate_diff} can be easily generalized to sections using the parallel transport induced by the metric structure on $E$. \BBB
The next lemma provides a Taylor expansion of parallel transport in a coordinate domain, which will turn up to be useful in this task.
\begin{lemma}[First order Taylor expansion of parallel transport]
  Let $\Psi \colon \Omega \times \rr^m \to E|_O$ be a local trivialization, $x \in O$, and $r_0 < r^*$ sufficiently small so that $B_{r_0}(x) \subset O$.
  Then, there exists $T = T_{\Phi^{-1}(x)} \colon \Phi^{-1}(B_{r_0}(x)) \to \rr^{m \times m}$ such that
  \begin{equation*}
    \Psi^*(\transp_x(y, w)) = T(\Phi^{-1}(y)) \Psi^* w.
  \end{equation*}
  Furthermore, $T$ enjoys the following Taylor expansion
  \begin{equation}\label{taylor_parallel_transport}
    T_\beta^\alpha(z) = \delta_\beta^\alpha - X^k\Gamma_{k\beta}^\alpha + O(\dist(x, \Phi(z))^2), \qquad \NNN X \defas \exp_x^{-1}(\Phi(z))\BBB,
  \end{equation}
  where $(\Gamma_{k \beta}^\alpha)$ are the Christoffel symbols at $x$.
\end{lemma}

\begin{proof}
  For a proof we refer to Section 3.3.2 in \cite{nicolaescu_book}.
\end{proof}

Let $O \subset M$ be a bounded open set and let $\Phi \colon \Omega \to O$ and $\setof{\NNN\tau_1\BBB, \dots, \NNN\tau_m\BBB}$ denote a chart and a frame \NNN (smooth up to the boundary)\BBB, respectively.
We define
\begin{align*}
  \Lambda = \Lambda(\NNN \Omega \BBB) &\defas \sup\setof*{\sqrt{\tilde \lambda}: \tilde \lambda \text{ eigenvalue of } (g_{ij}(x)) \text{ or } (\tilde g_{\alpha \beta}(x)), \, x \in \NNN \Omega \BBB}, \\
  \lambda = \lambda(\NNN \Omega \BBB) &\defas \inf\setof*{\sqrt{\tilde \lambda}: \tilde \lambda \text{ eigenvalue of } (g_{ij}(x)) \text{ or } (\tilde g_{\alpha \beta}(x)), \, x \in \NNN \Omega \BBB}.
\end{align*}
As the chart and the frame are smooth in $\Omega$ and $g$ and $\tilde g$ are pointwise positive definite in $\bar \Omega$ it follows that $0 < \lambda \leq \Lambda < \infty$.
Then, for any $x \in \NNN \Omega \BBB$, $v \in \rr^n$, and $w \in \rr^m$ it holds that
\begin{align}\label{lambda_props}
    \lambda \abs{v} &\leq \abs{v}_{g(x)} \leq \Lambda \abs{v}, &
    \lambda \abs{w} &\leq \abs{w}_{\tilde g(x)} \leq \Lambda \abs{w}, &
    \lambda^n &\leq \sqrt{\abs{g(x)}} \leq \Lambda^n,
\end{align}
where $\abs{\cdot}$ denotes the Euclidean norm in both $\rr^n$ and $\rr^m$, $\abs{v}_{g(x)} \defas \sqrt{g_{ij}(x) v^i v^j}$, $\abs{w}_{\tilde g(x)} \defas \sqrt{\tilde g_{ij}(x) w^i w^j}$, and $\abs{g(x)}$ is the determinant of $(g_{ij}(x))$.
In particular, the first estimate in (\ref{lambda_props}) implies the following relation between Euclidean balls an the pre-image of a geodesic ball in coordinates:
\begin{equation}\label{relation_balls}
  B_{\lambda r}(x) \subset \Phi^{-1}(B_r(\Phi(x))) \subset B_{\Lambda r}(x),
\end{equation}
where $x \in \tilde \Omega$ and $r > 0$ is chosen sufficiently small so that $B_r(\Phi(x)) \subset \Phi(\tilde \Omega)$.
Furthermore, changing coordinates and using (\ref{lambda_props}) we can also derive that
\begin{equation}\label{ball_vol_relation}
  \lambda^n \haus^n(B_{\lambda r}(\Phi^{-1}(x))) \leq \haus^n_g(B_r(x)) \leq \Lambda^n \haus^n(B_{\Lambda r}(\Phi^{-1}(x))).
\end{equation}

We are ready to prove the relation of approximate limits on the manifold and in Euclidean space.

\NNN
\begin{proposition}[Approximate limits and coordinates] \label{prop:approx_cont_coords_gen}
  Let $\Psi \colon \Omega \times \rr^m \to E|_O$ be a local trivialization.
  Then, a section $u \in L^1(E|_O)$ has approximate limit $z$ at $x \in O$ if and only if its coordinate representation $\Psi^* u$ has approximate limit $\Psi^* z \in \rr^m$ at $\Phi^{-1}(x)$.
\end{proposition}
\BBB

\begin{proof}[Proof of Proposition \NNN\ref{prop:approx_cont_coords_gen}\BBB]
  Suppose that $\Psi^* w$ is the approximate limit of $\Psi^* u$ at the point $\Phi^{-1}(x)$.
  Then, by changing coordinates $z = \Phi^{-1}(y)$, (\ref{lambda_props}), (\ref{relation_balls}), (\ref{ball_vol_relation}), and (\ref{taylor_parallel_transport})
  \begin{align*}
    \fint_{B_r(x)} \abs{u(y) - \transp(y, w)} \di \haus^n_g
    &= \frac{1}{\haus^n_g(B_r(x))} \int_{\Phi^{-1}(B_r(x))} \abs{\Psi^* u(z) - T(z)\Psi^* w}_{\tilde g(z)} \sqrt{\abs{g(z)}} \di z \\
    &\leq \frac{1}{\lambda^n \haus^n(B_{\lambda r}(\Phi^{-1}(x)))} \int_{B_{\Lambda r}(\Phi^{-1}(x))} \Lambda \abs{\Psi^* u(z) - T(z) \Psi^* w} \Lambda^n \di z \\
    &\leq \frac{\Lambda^{2n+1}}{\lambda^n} \fint_{B_{\Lambda r}(\Phi^{-1}(x))} \abs{\Psi^* u(z) - \Psi^* w} + \abs{T(z) \Psi^* w - \Psi^* w} \di z \\
    &\leq C \fint_{B_{\Lambda r}(\Phi^{-1}(x))} \abs{\Psi^* u(z) - \Psi^* w} \di z + C r
  \end{align*}
  for $C > 0$ independent of $r$ and $T$ is as in (\ref{taylor_parallel_transport}).
  \NNN L\BBB etting $r \to 0$ in the inequality above shows that $z$ is an approximate limit of $u$ at $x$.
  
  The reverse implication in the statement follows similarly.
\end{proof}

Before coming to the proof of Proposition \NNN\ref{prop:approx_jump_coords_gen} \BBB we introduce the following helpful result:
\begin{lemma}\label{lem:rel_nu_barnu}
  Let $N \subset M$ be an $(n-1)$-dimensional $C^1$-submanifold of $M$ and $\nu$ a unit normal on $N$ at some point $x \in N$.
  Further, let $\Phi \colon \Omega \to O$ \NNN be \BBB a chart of $M$ with $x \in O$.
  Then, the following relation holds true between $\nu$ and the Euclidean normal $\bar \nu$ on $\Phi^{-1}(N \cap O)$ at $\Phi^{-1}(x)$ with orientation induced by $\nu$:
  \begin{equation}\label{rel_nu_barnu}
    \nu^i = \frac{1}{\sqrt{g^{kl} \bar \nu^k \bar \nu^l}} g^{ij} \bar \nu^j \qquad \NNN \text{for } k \in \setof{1, \dots, n}.
  \end{equation}
\end{lemma}

\begin{proof}
  \NNN N\BBB ote that the coordinate representation $(X^k)$ for any $X \in T_x N$ is tangential to $\Phi^{-1}(N \cap O)$ at $\Phi^{-1}(x)$ in the Euclidean sense.
  Let us define $\mu^i \defas g^{ij} \bar \nu^j$ for $i \in \setof{1, \dots, n}$.
  Then, as $\bar \nu$ is orthogonal to the coordinate representation of any $X \in T_xN$ we derive that
  \begin{equation*}
    g_{ij} \mu^i X^j = g_{ij} g^{ik} \bar \nu^k X^j = \delta_j^k \bar \nu^k X^j = \bar \nu^j X^j = 0.
  \end{equation*}
  Consequently, $\mu$ must be parallel to the coordinate representation of $\nu$.
  Furthermore, by positive definiteness of $(g^{ij})$ it follows that
  \begin{equation*}
    \mu^i \bar \nu^i = g^{ik} \bar \nu^k \bar \nu^i = \bar \nu^i \nu^i > 0,
  \end{equation*}
  where \NNN we have used \BBB in the last inequality that $\bar \nu$ and $\nu$ have the same orientation.
  Eventually the coordinate representation of $\nu$ follows by renormalizing $\mu$, that is
  \begin{equation*}
    \nu^i = \frac{1}{\abs{\mu}_g} \mu^i = \frac{1}{\sqrt{g_{k'l'} g^{k' l} \bar \nu^{l} g^{l' k} \bar \nu^{k}}} \mu^i = \frac{1}{\sqrt{\delta_{l'}^l g^{l'k} \bar \nu^k \bar \nu^{l}}} \mu^i = \frac{1}{\sqrt{g^{kl} \bar \nu^k \bar \nu^l}} g^{ij} \bar \nu^j.
  \end{equation*}
\end{proof}

\NNN
\begin{proposition}[Approximate jumps and coordinates] \label{prop:approx_jump_coords_gen}
	Let $\Psi \colon \Omega \times \rr^m \to E|_O$ be a local trivialization.
	Then, a section $u \in L^1_\loc(E|_O)$ has an approximate jump at $x \in O$ with triplet $(a, b, \nu) \in E_x \times E_x \times T_x^*\! M$ if and only if $\Psi^* u$ has an approximate jump at $\Phi^{-1}(x)$ in the usual Euclidean sense with triplet $(\Psi^*a, \Psi^*b, \bar \nu)$, such that
  \begin{equation*}
    \nu^k = \frac{1}{\sqrt{g^{ij} \bar \nu^i \bar \nu^j}} g^{kl} \bar \nu^l \qquad \text{for } k \in \setof{1, \dots, n}
  \end{equation*}
  and $(g^{ij})$ denotes the inverse of the metric tensor $(g_{ij})$.
\end{proposition}
\BBB

\begin{proof}[Proof of Proposition \NNN\ref{prop:approx_jump_coords_gen}\BBB]
  Let $r_0$ be chosen sufficiently small so that $B_{2r_0}(x) \subset O \defas \Phi(\Omega)$.
  As we will be interested in limits where $r \to 0$, without loss of generality, we assume that $r < r_0$.
  Furthermore, we will only prove the statement for the approximate upper limit as the arguments for the lower limit are the same.

  Let us first assume that $\Psi^* a$ is the approximate upper limit and $\Psi^* b$ is the approximate lower limit of $\Psi^* u$ at $\Phi^{-1}(x)$ in direction $\bar \nu$.
  Further, let $\nu = \nu^i \del{i} \in T_x M$ be given by
  \begin{equation*}
    \nu^i = \frac{1}{\sqrt{g^{kl} \bar \nu^k \bar \nu^l}} g^{ij} \bar \nu^j \qquad \NNN \text{for } i \in \setof{1, \dots, n}\BBB.
  \end{equation*}
  Similarly to the proof of Proposition \NNN\ref{prop:approx_cont_coords_gen}, \BBB to show that $a$ is an approximate upper limit of $u$ at $x$ with respect to the unit normal $\nu$ it suffices to prove that 
  \begin{equation*}
    \lim_{r \to 0} \frac{1}{\haus^n(\Phi^{-1}(B_r^+(x, \nu)))} \int_{\Phi^{-1}(B_r^+(x, \nu))} \abs{\Psi^* u(y) - \Psi^* a} \di y = 0.
  \end{equation*}
  We \NNN define \BBB the $(n-1)$-dimensional geodesic disk centered at the point $x$ with radius $r$ orthogonal to $\nu$ as follows:
  \begin{equation*}
    D_r(x, \nu) \defas \exp_x(\setof{X \in T_xM \colon \abs{X} < r, \, \sprod{X}{\nu} = 0}).
  \end{equation*}
  Note that $D_r(x, \nu)$ is a smooth $(n-1)$-dimensional submanifold of $M$ and that the vector $\nu$ is orthogonal to $D_r(x, \nu)$ at $x$.
  By Lemma \ref{lem:rel_nu_barnu} and our choice of $\nu$, the Euclidean unit normal onto $\Phi^{-1}(D_r(x, \nu))$ at $\Phi^{-1}(x)$ is given by $\bar \nu$.
  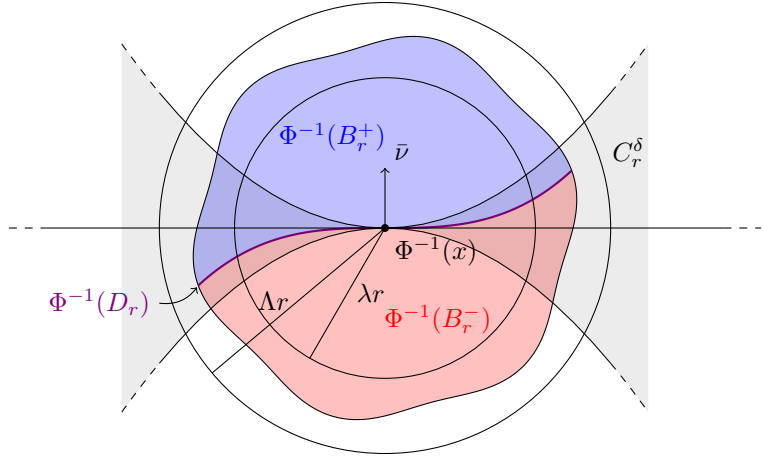
\begin{figure}
  \begin{tikzpicture}
    \begin{scope}
      \fill[domain=0:360, samples=500, blue, opacity=0.25] plot (\x:{2.5+0.1*sin(6*\x)});
      \node[blue] at (120:1.4) {$\Phi^{-1}(B_r^+)$};
      \clip[domain=0:360, samples=500] plot (\x:{2.5+0.1*sin(6*\x)});
      \clip[domain=-3:3, variable=\x] 
        (-3, -1.35)
        -- plot (\x, {0.05*(\x)^3})
        -- (3, 1.35) 
        -- (3, -3)
        -- (-3, -3)
        -- cycle;
      \fill[white] (0, 0) circle (3);
      \fill[domain=0:360, samples=500, red, opacity=0.25] plot (\x:{2.5+0.1*sin(6*\x)});
      \node[red] at (-60:1.4) {$\Phi^{-1}(B_r^-)$};
    \end{scope}
    \draw[domain=0:360, samples=500, black] plot (\x:{2.5+0.1*sin(6*\x)});
    \begin{scope}
      \clip[domain=0:360, samples=500] plot (\x:{2.5+0.1*sin(6*\x)});
      \draw[domain=-3:3, samples=500, violet, thick] plot (\x, {0.05*(\x)^3});
    \end{scope}
    \draw (-4.5, 0) -- (4.5, 0);
    \draw[dashed] (-5.0, 0) -- (-4.5, 0);
    \draw[dashed] (4.5, 0) -- (5, 0);
    \draw[domain=-3:3] plot (\x, {0.2*(\x)^2});
    \draw[domain=-3.5:-3, dashed] plot (\x, {0.2*(\x)^2});
    \draw[domain=3:3.5, dashed] plot (\x, {0.2*(\x)^2});
    \draw[domain=-3:3] plot (\x, {-0.2*(\x)^2});
    \draw[domain=-3.5:-3, dashed] plot (\x, {-0.2*(\x)^2});
    \draw[domain=3:3.5, dashed] plot (\x, {-0.2*(\x)^2});
    \fill[domain=-3.5:3.5, variable=\x, opacity=0.075] 
      (-3.5, 0)
      -- plot (\x, {0.2*(\x)^2})
      -- (3.5, 0)
      -- cycle;
    \fill[domain=-3.5:3.5, variable=\x, opacity=0.075] 
      (-3.5, 0)
      -- plot (\x, {-0.2*(\x)^2})
      -- (3.5, 0)
      -- cycle;
    \node[right] at (2.9, 1) {$C_r^\delta$};
    \draw (0, 0) circle (2);
    \draw (0, 0) circle (3);
    \draw (0, 0) -- (-120:2) node[midway, right] {$\lambda r$};
    \draw (0, 0) -- (-140:3) node[midway, left] {$\Lambda r$};
    \draw[->] (0, 0) -- (0, 0.8);
    \node[right] at (0, 1) {$\bar \nu$};
    \fill[black] (0, 0) circle (1.5 pt);
    \node[below right] at (0, 0) {$\Phi^{-1}(x)$};
    \draw[->] (-3, -1) to [bend right] (-2.5, -0.79);
    \node[left, violet] at (-3, -1) {$\Phi^{-1}(D_r)$};
  \end{tikzpicture}
  \caption{Geodesic half-balls in coordinates}
  \label{fig:geodesic_ball}
  \end{figure}
  Therefore, there exists a $\delta > 0$ such that $\Phi^{-1}(D_r(x, \nu))$ is contained in the parabolic cone 
  \begin{equation*}
    \setof*{y \in \rr^2 \colon \abs{\Pi_{\bar \nu} (y - x)} \leq \delta \abs{\Pi_{\bar \nu^\perp} (y - x)}^2},
  \end{equation*}
  where $\Pi_{\bar \nu} y \defas y^i \bar \nu^i$ and $\Pi_{\bar \nu^\perp} y \defas y - \Pi_{\bar \nu} y$ (see Figure \ref{fig:geodesic_ball}).
  Setting $H^+ \defas \setof{y \in \rr^n \colon (y^i - x^i) \bar \nu^i \geq 0}$, we can estimate
  \begin{equation*}
    \abs{\Phi^{-1}(B^+_r(x, \nu)) \setminus H^+} + \abs{\Phi^{-1}(B^-_r(x, \nu)) \cap H^+} \leq C_1 r^{n + 1},
  \end{equation*}
  where $C_1 > 0$ is independent of $r$.
  As
  \begin{equation*}
    \Phi^{-1}(B^+_r(x, \nu)) \supset B^+_{\lambda r}(\Phi^{-1}(x), \bar \nu) \setminus (\Phi^{-1}(B_r^-(x, \nu)) \cap H^+),
  \end{equation*}
  by possibly decreasing the value of $r_0$, the following bound holds:
  \begin{equation*}
    \haus^n(\Phi^{-1}(B^+_r(x, \nu))) \geq C_2 r^n
  \end{equation*}
  for some constant $C_2 > 0$ independent of $r$.
  As a consequence it follows for $r$ small enough that
  \begin{align*}
    &\frac{1}{\haus^n(\Phi^{-1}(B^+_r(x, \nu)))} \int_{\Phi^{-1}(B_r^+(x, \nu))} \abs{\Psi^* u(y) - \Psi^* a} \di y \\
    &\quad\leq \frac{1}{C_2} \left( \omega_n \Lambda^n \fint_{B^+_{\Lambda r}(\Phi^{-1}(x), \bar \nu)} \hspace{-1cm}\abs{\Psi^* u(y) - \Psi^* a} \di y + \frac{1}{r^n} \int_{\Phi^{-1}(B_r^+(x, \nu)) \setminus H^+} \hspace{-1cm}\abs{\Psi^* u(y) - \Psi^* b} \di y \right)
    + \frac{C_1 r^{n + 1} \abs{\Psi^* a - \Psi^* b}}{C_2 r^n} \\
    &\quad\leq C \left(\fint_{B^+_{\Lambda r}(\Phi^{-1}(x), \bar \nu)} \hspace{-1cm}\abs{\Psi^* u(y) - \Psi^* a} \di y + \fint_{B^-_{\Lambda r}(\Phi^{-1}(x), \bar \nu)} \hspace{-1cm}\abs{\Psi^* u(y) - \Psi^* b} \di y + r\right),
  \end{align*}
  where $\omega_n$ the volume of the unit ball in $\rr^n$.
  Taking the limit $r \to 0$ \NNN leads to the desired result\BBB.

  We omit the proof of the reverse implication as it follows by similar arguments.
\end{proof}

Lastly we prove the relation between approximate differentiability points on the manifold and in Euclidean space.

\NNN
\begin{proposition}\label{prop:approx_diff_coords_gen}
	Let $\Psi \colon \Omega \times \rr^m \to E|_O$ be a local trivialization with induced frame $\setof{\tau_1, \ldots, \tau_m}$.
	Then, any section $u \in L^1(E|_O)$ is approximately differentiable at $x \in O$ with approximate gradient $L \in E_x \otimes T_x^* \! M$ if and only if $\Psi^*u$ is approximately differentiable at $\Phi^{-1}(x)$ in the usual Euclidean sense with approximate gradient $\bar L \in \rr^{m \times n}$ and approximate limit $\bar z \in \rr^m$ such that
	\begin{equation}\label{approx_grad_coords_section}
		L = (\bar L_i^\alpha + \Gamma_{i \beta}^\alpha \bar z^\beta) \, \tau_\alpha \otimes \di x^i,
	\end{equation}
	where $(\Gamma_{i \beta}^\alpha)$ denotes the Christoffel symbols at $x$.
\end{proposition}
\BBB

\begin{proof}[Proof of Proposition \NNN\ref{prop:approx_diff_coords_gen}\BBB]
  Let $r_0$ be chosen sufficiently small so that $B_{2r_0}(x) \subset O \defas \Phi(\Omega)$.
  As we will be interested in limits where $r \to 0$, without loss of generality, we assume that $r < r_0$.
  Let us assume that $\Phi^{-1}(x)$ is an approximate differentiability point of $\Psi^* u$ with approximate gradient $\bar L$.
  By definition of approximate differentiability in the Euclidean setting, $\Psi^* u$ has an approximate limit $\Psi^* z$ at $\Phi^{-1}(x)$.
  With Proposition \NNN\ref{prop:approx_cont_coords_gen} \BBB this implies that $u$ has the approximate limit $z$ at $x$.
  Let $L$ be defined as in the statement.
  Following the same lines of the proof of Proposition \NNN\ref{prop:approx_cont_coords_gen}, \BBB in order to prove that $u$ is approximately differentiable at $x$ with approximate gradient $L$\NNN, \BBB it suffices to show that
  \begin{equation}\label{approx_diff_limit_coords}
    \lim_{r \to 0} \fint_{B_r(\bar x)} \frac{1}{r} \abs{\Psi^* u(\bar y) - T(\bar y)\Psi^* z - T(\bar y) \Psi^* L(X)} \di \bar y = 0,
  \end{equation}
  where $X \defas \exp_x^{-1}(y)$ and $\bar x \defas \Phi^{-1}(x)$.
  For each $\bar y \in B_r(\bar x)$ let $y \NNN \defas \BBB \Phi(\bar y)$ and $\gamma_y \colon [0, \dist(x, y)] \to M$ be the unique unit-speed geodesic such that $\gamma_y(0) = x$ and $\gamma_y(\dist(x, y)) = y$.
  By the smoothness of the map $(y, t) \mapsto \gamma_y(t)$ the following Taylor expansion \NNN holds true\BBB:
  \begin{equation}\label{taylor_X}
    X^i \defas \dist(x, y) \dot \gamma_y^i(0) = y^i - x^i + O(r^2).
  \end{equation}
  Thanks to (\ref{taylor_parallel_transport}) we have for any $\alpha \in \setof{1, \dots, m}$ \NNN that \BBB
  \begin{equation*}
    T(y)_\beta^\alpha L^\beta(X)
    = L^\alpha(X) - X^k \Gamma_{k \beta}^\alpha L^\beta(X) + O(r^2) = L^\alpha(X) + O(r^2),
  \end{equation*}
  where we \NNN have \BBB used $\sqrt{X^i X^i} = O(r)$.
  Hence using the choice of $L$ and (\ref{taylor_X}) we further derive for any $\alpha \in \setof{1, \dots, m}$ that
  \begin{equation*}
    T(y)_\beta^\alpha L^\beta(X)
    = \bar L_k^\alpha X^k + X^k \Gamma_{k \beta}^\alpha z^\beta + O(r^2)
    = \bar L_k^\alpha (y^k - x^k) + (\Gamma_0)_\beta^\alpha z^\beta + O(r^2),
  \end{equation*}
  \NNN where \BBB $(\Gamma_0)_\beta^\alpha \defas X^k \Gamma_{k \beta}^\alpha$.
  Consequently, the integrand in (\ref{approx_diff_limit_coords}) can be written as
  \begin{align*}
    \frac{1}{r} \abs{\Psi^* u(\bar y) - T(\bar y)\Psi^* z - T(\bar y) \Psi^* L(X)}
    & = \frac{1}{r} \abs{\Psi^* u(y) - (\Psi^* z - \Gamma_0 \Psi^* z) - (\bar L (\bar y - \bar x) + \Gamma_0 \Psi^* z)} + O(r) \\
    & = \frac{1}{r} \abs{\Psi^* u(y) - \Psi^* z - \bar L (\bar y - \bar x)} + O(r).
  \end{align*}
  The desired limit in (\ref{approx_diff_limit_coords}) then follows by the approximate differentiability of $\Psi^* u$ at $\Phi^{-1}(x)$ with approximate limit $\Psi^* z$ and approximate gradient $\bar L$.
  
  The reverse implication in the statement follows similarly.
\end{proof}

We are ready to investigate the decomposition of a section $u \in BV(E)$.
We first provide several helpful results.
The next lemma derives a formula for integration on submanifolds in coordinates.
Without further mention we will assume that the metric tensor of an oriented $(n-1)$-dimensional $C^1$-submanifold $N \subset M$ is given by the restriction of $g$ to $TN$.
The corresponding volume form on $N$ will be written as $\vol_N$.
Note that the orientation of $M$ guarantees the existence of a unit-length $C^1$-vector field $\nu$ such that $\iprod{\nu}{\vol} = \vol_N$, where $\iprod{\nu}{(\cdot)}$ is the interior product with $\nu$.
We will call $\nu$ the \textit{normal vector-field} of $N$.

\begin{lemma}[Integration on a submanifold and coordinates]\label{lem:integration_submanifold_coords}
  Let $N \subset M$ be an $(n-1)$-dimensional oriented $C^1$-submanifold of $M$.
  Furthermore, consider an open set $\Omega \subset \rr^n$ and an orientation preserving chart $\Phi \colon \Omega \to O \subset M$.
  Then, for any $f \in C^\infty_c(O)$ we have
  \begin{equation*}
    \int_N f \vol_N = \int_{\bar N} f \circ \Phi \sqrt{g^{ij} \bar \nu^i \bar \nu^j} \afac \vol_{\bar N},
  \end{equation*}
  $\bar N \defas \Phi^{-1}(N \cap O)$, $\bar \nu$ is the Euclidean unit normal onto $\bar N$ with orientation induced by the normal field $\nu$ on $N$, and $\vol_{\bar N}$ is the volume form on $\bar N$ induced by the restriction of the Euclidean metric to $\bar N$.
\end{lemma}

\begin{remark}
  Note that by a standard approximation argument the result above holds for locally integrable functions on $N$.
\end{remark}

Before coming to the proof of Lemma \ref{lem:integration_submanifold_coords} we state the following classical result (see Proposition 4.1.54 in \cite{nicolaescu_book}).
For any $l \in \setof{0, \dots, n}$ we denote by $\Omega^l(M)$ the space of smooth $l$-forms on $M$.
\begin{lemma}
	For $k \in \setof{0, \dots, n}$ let $\alpha \in \Omega^{k -1}(M)$ and $\beta \in \Omega^k(M)$.
  Further\NNN more\BBB, let $N$ be an $(n - 1)$-dimensional oriented $C^1$-submanifold of $M$ with unit normal field $\nu$.
  \NNN Then\BBB,
	\begin{equation}
		\label{magic_formula}
		(\alpha \wedge \hodge \beta)|_N = \alpha|_N \wedge \hodge_N (\iprod{\nu} \beta)|_N,
	\end{equation}
	where $(\cdot)|_N$ is the restriction to $N$ and $\hodge_N$ denotes the Hodge star on $N$.
\end{lemma}

\begin{proof}[Proof of Lemma \ref{lem:integration_submanifold_coords}]
  Let $\hodge_N$ denote the Hodge star on $N$ induced by the restriction of $g$ to $N$.
  Note that $\hodge_N 1 = \vol_N$, the volume form on $N$.
	Using (\ref{magic_formula}) with $\alpha = f$ and $\beta = \nu^\flat \defas \sprod{\nu}{\cdot}$, the linearity of $\hodge$, and \NNN a \BBB chang\NNN e of \BBB coordinates \NNN shows \BBB
	\begin{equation}\label{prop_pullback_fnuflat}
		\int_{N \cap O} f \vol_N
			= \int_{N \cap O} f \wedge \hodge_N (\nu^\flat(\nu))
			= \int_{N \cap O} f \wedge \hodge_N (\iprod{\nu} \nu^\flat)
			= \int_{N \cap O} \hodge (f \nu^\flat)
      = \int_{\bar N} \Phi^* (\hodge f \nu^\flat),
	\end{equation}
  where $\Phi^*$ is the pull-back operator induced by $\Phi$ and $\bar N \defas \Phi^{-1}(N \cap O)$.
  By the definition of $\hodge$ we have for an arbitrary $\alpha \in \Omega^1(O)$ that
  \begin{equation*}
    \alpha \wedge \hodge(f \nu^\flat) = f \sprod{\alpha}{\nu^\flat} \vol.
  \end{equation*}
  Taking the pull-back of both sides we \NNN have that \BBB
  \begin{equation*}
    \Phi^*(\alpha) \wedge \Phi^*(f \nu^\flat)
    = f g^{ij} \alpha_i \nu^\flat_j \afac \di x^1 \wedge \dots \wedge \di x^n
    = f g^{ij} \alpha_i g_{jk} \nu^k \afac \di x
    = f \alpha_i \nu^i \afac \di x.
  \end{equation*}
  Let $\beta \defas \nu^i \di x^i$ and $\alpha = \alpha_i \di x^i$ be an arbitrary 1-form on $\Omega$.
  By the definition of the Euclidean Hodge star $\hodge_{\rr^n}$ it holds that
  \begin{equation*}
    \alpha \wedge \hodge_{\rr^n} (f \afac \beta) = f \alpha_i \nu^i \afac \di x.
  \end{equation*}
  By (\ref{prop_pullback_fnuflat}) and the arbitrariness of $\alpha$ we see that
  \begin{equation*}
    \hodge_{\rr^n} (f \afac \beta) = \Phi^*(f \nu^\flat).
  \end{equation*}
	With this face and (\ref{magic_formula}) applied for the Euclidean Hodge star
  we \NNN arrive \BBB that
	\begin{equation*}
		\int_{N \cap U} f \vol_N
    = \int_{\bar N} f \afac \wedge \hodge_{\rr^n} \beta
    = \int_{\bar N} f \afac \wedge \hodge_{\bar N} (\iprod{\bar \nu} \beta)
    = \int_{\bar N} f \bar \nu^i \nu^i \vol_{\bar N},
	\end{equation*}
  where $\bar \nu$ is as in the statement.
  The desired result then ready follows from (\ref{rel_nu_barnu}) as
  \begin{equation*}
    \bar \nu^i \nu^i = \frac{\bar \nu^i g^{ij} \bar \nu^j}{\sqrt{g^{kl} \bar \nu^k \bar \nu^l}} = \sqrt{g^{ij} \bar \nu^i \bar \nu^j}.
  \end{equation*}
\end{proof}

In the next lemma we will relate the total variation of a section defined in (\ref{total_variation_local}) to the Euclidean total variation of its coordinate representations.
We will assume that $\Psi \colon \Omega \times \rr^m \to E|_O$ is a local trivialization such that its induced frame $\setof{\NNN\tau_1\BBB, \dots, \NNN\tau_m\BBB}$ is orthonormal.
As usual, we will denote by $\Phi$ the induced chart.
\begin{lemma}[Coordinate representation of the total variation]\label{lem:tot_var_coords}
  Let $\Psi \colon \Omega \times \rr^m \to E|_O$ be as above and let $u \in \NNN L^1\BBB(E|_O)$.
  Then, there exists a constant $C$ independent of $u$ such that
  \begin{equation}\label{variation_man_vs_coords}
  \begin{aligned}
    \var(u, \NNN O \BBB) &\leq C( \var(\Psi^* u, \NNN\Omega\BBB) + \norm{u}_{L^1(E|_{\NNN O \BBB})}), \\
    \var(\Psi^* u, \NNN\Omega\BBB) &\leq C (\var(u, \NNN O \BBB) + \norm{\Psi^* u}_{L^1(\NNN \Omega \BBB; \rr^d)}),
  \end{aligned}
  \end{equation} 
  where $\var(\Psi^* u, \NNN\Omega\BBB)$ stands for the Euclidean total variation of $\Psi^* u$ in $\NNN\Omega\BBB$.

  Moreover, for any $v \in C^\infty_c(E|_O \otimes T^* O)$ it holds that
  \begin{equation}\label{test_coords}
    -\int_O \sprod{u}{\nabla^* v} \di \haus^n_g = \int_\Omega u^\alpha \del{k}(\afac g^{ki} v_i^\alpha) \di x + \int_\Omega g^{ki} u^\alpha \Gamma_{k \beta}^\alpha v_i^\beta \afac \di x.
  \end{equation}
\end{lemma}

\begin{proof}
  We start by proving (\ref{test_coords}).
  In a coordinate domain with orthonormal frame, $\nabla^*$ has the following representation (see \NNN e.g.~\BBB (10.1.8) in \cite{nicolaescu_book}):
  \begin{equation*}
    \nabla^* = - \left[
      \afac^{-1} \del{k} (\afac g^{ki}) + g^{ki} \nabla_{\del{k}}
    \right] \iprod{\del{i}},
  \end{equation*}
  where for any $w \in C^\infty_c(E|_O)$ and $\alpha \in \Omega^1(O)$ we set $\iprod{\del{i}} (w \otimes \alpha) \defas \alpha\NNN\left(\BBB\del{i}\NNN\right)\BBB w$.
  Let $v = v_i^\alpha \NNN\tau_\alpha\BBB \otimes \di x^i \in C^\infty_c(E|_O \otimes T^* O)$ with $\norm{v}_{L^\infty} \leq 1$, passing to coordinates and using the representation of $\nabla^*$ from above we derive
  \begin{align*}
    -\int_M \sprod{u}{\nabla^* v} \vol
    &= \int_\Omega u^\alpha \del{k} (\afac g^{ki}) v_i^\alpha \di x + 
    \int_\Omega u^\alpha g^{ki} \afac \frac{\partial v_i^\alpha}{\partial x^k} \di x
    + \int_\Omega g^{ki} u^\alpha \Gamma_{k \beta}^\alpha v_i^\beta \afac \di x \\
    &= \int_\Omega u^\alpha \del{k}(\afac g^{ki} v_i^\alpha) \di x + \int_\Omega g^{ki} u^\alpha \Gamma_{k \beta}^\alpha v_i^\beta \afac \di x,
  \end{align*}
  which is (\ref{test_coords}).

  We will now prove (\ref{variation_man_vs_coords}).
  Let $v = v_i^\alpha e_\alpha \otimes \di x^i \in C^\infty_c(E|_{\NNN O \BBB} \otimes T^* \NNN O \BBB)$.
  Using the smoothness of $g$, and the fact that $\norm{v}_{L^\infty} \leq 1$
  \begin{equation*}
    \sum_{k, \alpha} (\afac g^{ki} v_i^\alpha)^2
    \leq C \norm{v}_{L^\infty} \leq C.
  \end{equation*}
  Hence, by the definition of Euclidean total variation we have
  \begin{equation*}
    \int_{\NNN \Omega \BBB} u^\alpha \del{k}(\afac g^{ki} v_i^\alpha) \di x \leq C \var(\Psi^* u, \NNN \Omega \BBB).
  \end{equation*}
  Using the smoothness of the Christoffel symbols and the metric we can similarly estimate
  \begin{equation*}
    \int_\Omega g^{ki} u^\alpha \Gamma_{k \beta}^\alpha v_i^\beta \afac \di x \leq C \int_{\NNN O \BBB} \abs{u} \di \haus^n_g.
  \end{equation*}
  for some constant $C$ independent of $u$ and $v$.
  This completes the proof of the first inequality in \ref{variation_man_vs_coords}.
  The proof of the second inequality follows similarly.
\end{proof}

In the following we provide a definition for the push-forward of a vector-valued Radon measure in the Euclidean space to the manifold.

\begin{definition}[Push-forward of vector measures]\label{def:push_forward}
  Let $F$ be \NNN a \BBB vector-bundle over $M$ \NNN of rank $m$ \BBB, $\Psi \colon \Omega \times \rr^m \to F|_O$ a local trivialization, and $\bar \nu \in \radon(\Omega; \rr^m)$ an $\rr^m$-valued Radon measure on $\Omega$.
  We denote by $\Psi \# \bar \nu \in \radon(F|_O)$ the unique generalized vector-measure such that
  \begin{equation}\label{def_push_forward_man}
    \sprod{\Psi \# \bar \nu}{v} = \int_\Omega v^\alpha(\Phi(x)) \bar \sigma^\alpha(x) \di\abs{\bar \nu}(x)
  \end{equation}
  for all $v \in C_c(F|_O)$, where $\bar \sigma$ and $\abs{\bar \nu}$ are the polar density and total variation of $\bar \nu$, respectively.
\end{definition}

Given $\nu \in \radon(F|_O)$ and $A \in \borel(O)$ we denote by $\nu \mrestr A$ the restriction of $\nu$ to $A$ which is defined through
\begin{equation*}
  \sigma_{\nu \mrestr A} = \sigma_\nu, \quad \abs{\nu \mrestr A} = \abs{\nu} \mrestr A.
\end{equation*}

Furthermore, using the definition of $\Phi \# \abs{\bar \nu}$ we see that 
Let $F, \, \Psi, \, \bar \nu$ be as in Definition \ref{def:push_forward}.
Using the definition of $\Phi \# \bar \nu$, where $\Phi$ is the chart associated to $\Psi$, we see that
\begin{equation*}
  \sigma_{\Psi \# \bar \nu}^\alpha  = \frac{\tilde g^{\alpha \beta} \bar \sigma^\beta}{\sqrt{\tilde g^{\gamma \delta} \bar \sigma^\gamma \bar \sigma^\delta}}, \quad \abs{\Psi \# \bar \nu} = \sqrt{\tilde g^{\alpha \beta} \bar \sigma^\alpha \bar \sigma^\beta} \Phi \# \abs{\bar \nu}.
\end{equation*}
In the upcoming proof we will use the following relation
\begin{equation}\label{meas_restr_property}
  (\Psi \# \bar \nu) \mrestr A = \Psi \# (\bar \nu \mrestr \Phi^{-1}(A)).
\end{equation}

We will now relate all components of the distributional derivative of a $BV$ section $u$ to the corresponding components of its coordinate representations.
\begin{lemma}\label{lem:comp_der_man_vs_coords}
  Let $\Psi \colon \Omega \times \rr^m \to E|_O$ be a local trivialization of $E$, $\tilde \Psi \colon \Omega \times \rr^{m \times n} \to E|_O \otimes T^*O$ be a local trivialization of $E \otimes T^*\!M$ (both with the same induced coordinate chart), and $u \in BV_\loc(E|_O)$.
  Then, the following relations hold true:
  \begin{align}
    D^a u &= \tilde \Psi \# (\afac g^{-1} D^a (\Psi^* u)) + \Gamma u \haus^n_g, \label{rel_abs_cont}\\
    D^j u &= \tilde \Psi \# (\afac g^{-1} D^j (\Psi^* u)), \label{rel_jump}\\
    D^c u &= \tilde \Psi \# (\afac g^{-1} D^c (\Psi^* u)), \label{rel_cantor}
  \end{align}
  where $\Gamma u \defas \Gamma_{i \beta}^\alpha u^\beta e_\alpha \otimes \di x^i$ and $g^{-1} \defas (g^{ij})$.
\end{lemma}
\begin{remark}
  Note that the right hand-sides of (\ref{rel_abs_cont}), (\ref{rel_jump}), and (\ref{rel_cantor}) are well-defined since by Lemma \ref{lem:tot_var_coords} and $u \in BV_\loc(E|_O)$ implies that $\Psi^* u$ has locally bounded variation.
\end{remark}

\begin{proof}
  Let $v \in C^\infty_c(E|_O \otimes T^* O)$, integrating by parts in (\ref{test_coords}), using $\Gamma_{i \alpha}^\beta = - \Gamma_{i \beta}^\alpha$ and the Euclidean Radon-Nikodym theorem, we have
  \begin{align*}
    -\int_O \sprod{u}{\nabla^* v} \vol
    &= -\int_\Omega g^{ij} v_j^\alpha \bar \sigma_i^\alpha \afac \di\abs{D (\Psi^* u)} + \int_\Omega g^{ij} u^\alpha \Gamma_{i \beta}^\alpha v_j^\beta \afac \di x \\
    &= -\int_\Omega g^{ij} v_j^\alpha ((\bar \sigma^a)_i^\alpha + \Gamma_{i \beta}^\alpha u^\beta) \afac \di x - \int_\Omega g^{ij} v_j^\alpha (\bar \sigma^s)_i^\alpha \afac \di\abs{D^s(\Psi^* u)},
  \end{align*}
  where $\bar \sigma$, $\bar \sigma^a$, and $\bar \sigma^s$ are the polar densities of $D \Psi^* u$, $D^a \Psi^* u$, and $D^s \Psi^* u$, respectively.
  By the arbitrariness of $v$, the definition of the push-forward in (\ref{def_push_forward_man}), and the uniqueness of Radon-Nikodym decomposition of $D u$ with respect to $\haus^n_g$, we obtain equality (\ref{rel_abs_cont}) together with
  \begin{equation} \label{rel_singular}
    D^s u = \Psi \# (\afac g^{-1}  D^s (\Psi^* u)).
  \end{equation}
  Note that by Proposition \NNN \ref{prop:approx_cont_coords_gen} \BBB and Proposition \NNN \ref{prop:approx_jump_coords_gen} \BBB we have that $\mathcal S_u = \Phi(\mathcal S_{\Psi^* u})$ and $\mathcal J_u = \Phi(\mathcal J_{\Psi^* u})$.
  Consequently, (\ref{rel_jump}) and (\ref{rel_cantor}) follow from (\ref{rel_singular}) and (\ref{meas_restr_property}).
\end{proof}

We are ready to prove the decomposition theorem for sections of bounded variation \NNN which generalizes Theorem \ref{thm:decomp_secs}\BBB.

\NNN
\begin{theorem}[Decomposition of sections of bounded variation]\label{thm:decomp_secs_gen}
	Let $u \in BV(E)$, then the discontinuity set $\mathcal S_u$ is $\haus_g^{n-1}$-rectifiable, $\haus_g^{n-1}(\mathcal S_u \setminus \mathcal J_u) = 0$, and
	the restriction $D^j u \defas D^s u \mrestr {\mathcal J_u}$ of the singular part of $Du$ to $\mathcal J_u$ can be represented as
	\begin{equation*}
		D^j u = (u^+ - u^-) \otimes \nu^\flat \haus_g^{n-1} \mrestr {\mathcal J_u},
	\end{equation*}
	where the triplet $(u^+, \, u^-, \, \nu)$ is as in Definition \ref{def:approximate_jump} adapted to the setting of vector bundles and $\nu^\flat$ is the 1-form given by $\nu^\flat(X) = \sprod{\nu}{X}$ for any $X \in T_x M$.

	Furthermore, $u$ is approximately differentiable at a.e.~point of $M$ and the absolutely continuous part of $Du$ can be written as
	\begin{equation*}
		D^a u = \nabla u \, \haus_g^n,
	\end{equation*}
	$\nabla u$ being the approximate gradient of $u$.	
\end{theorem}
\BBB

\begin{proof}[Proof of Theorem \ref{thm:decomp_secs_gen}]
  In what follows we will assume that $\Psi \colon \NNN\Omega\BBB \times \rr^m \to E|_{\NNN O\BBB}$ is an arbitrary local trivialization with indued chart $\Phi$.
  Note that by Lemma (\ref{lem:tot_var_coords}) $\Psi^* u \in BV(\Omega; \rr^m)$.

  \textit{\NNN Step 1 (R\BBB ectifiability of $\mathcal S_u$):}
  By Proposition \NNN \ref{prop:approx_cont_coords_gen} \BBB and Proposition \NNN \ref{prop:approx_jump_coords_gen} \BBB we have that $\mathcal S_u \cap O = \Phi(\mathcal S_{\Psi^* u})$ and $\mathcal J_u \cap O = \Phi(\mathcal J_{\Psi^* u})$.
  By Theorem 3.78 in \cite{fusco_bible} $\mathcal S_{\Psi^* u}$ is $\haus^{n-1}$-rectifiable and $\haus^{n-1}(\mathcal S_{\Psi^* u} \setminus \mathcal J_{\Psi^* u}) = 0$.
  By the smoothness of $\Phi$ and (\ref{lambda_props}) we derive that $S_u \cap O$ is $\haus^{n-1}_g$-rectifiable and $\haus^{n-1}_g((S_u \setminus \mathcal J_u) \cap O) = 0$.
  The result on $M$ then follows by the arbitrariness of $\Psi$.

  \textit{\NNN Step 2 (C\BBB haracterization of the absolutely continuous part):}
  Following the same argument in the first step but using Proposition \NNN \ref{prop:approx_diff_coords_gen} \BBB and Theorem 3.83 in \cite{fusco_bible} one has that $u$ is approximately differentiable at $\haus^n_g$-a.e.~points in $M$.
  For any test-function $v \in C^\infty_c(E|_O \otimes T^* O)$ we derive from (\ref{rel_abs_cont}) and that (\ref{approx_grad_coords_section})
  \begin{equation*}
    \int_O \sprod{v}{\sigma^a} \di\abs{D^a u}
    = \int_\Omega g^{ij} v_j^\alpha (\nabla (\Psi^* u))_i^\alpha \sqrt{\abs{g}} \di x + \int_O \sprod{v}{\Gamma u} \di\haus^n_g
    = \int_O \sprod{v}{\nabla u} \di \haus^n_g,
  \end{equation*}
  where $\sigma^a$ is the polar density of $D^a u$, and according to Theorem 3.83 in \cite{fusco_bible}, $D^a(\Psi^*u) = \nabla (\Psi^*u) \haus^n$.
  By the arbitrariness of $v$ we have shown that $D^a u \mrestr O = \nabla u \haus^n_g \mrestr O$.
  The result on $M$ follows by a standard partition of unity argument.

  \textit{\NNN Step 3 (C\BBB haracterization of the jump set)}
  As in the second step it is enough to prove the representation in $O$.
  With (\ref{rel_jump}) and the representation of $D^j (\Psi^* u)$ in the Euclidean setting (see Theorem 3.78 in \cite{fusco_bible}) we have
  \begin{equation*}
    \int_O \sprod{v}{\sigma^j} \di\abs{D^j u}
    = \int_{\mathcal J_{\Psi^* u}} g^{ij} v_j^\alpha (((\Psi^* u)^+)^\alpha - ((\Psi^* u)^-)^\alpha) \bar \nu^i \sqrt{\abs{g}} \di \haus^{n-1},
  \end{equation*}
  where $\sigma^a$ is the polar density of $D^j u$ and $\bar \nu$ is approximate normal to $\mathcal J_{\Psi^* u}$.
  Using Proposition \NNN (\ref{prop:approx_jump_coords_gen}) and \BBB $\nu^j = g^{jk} (\nu^\flat)_k$ it follows that
  \begin{align*}
    (((\Psi^* u)^+)^\alpha - ((\Psi^* u)^-)^\alpha) g^{ij} \bar \nu^i
    &= ((\Psi^* u^+)^\alpha - (\Psi^* u^-)^\alpha) \nu^j \sqrt{g^{ll'} \bar \nu^l \bar \nu^{l'}} \\
    &= ((\Psi^* u^+)^\alpha - (\Psi^* u^-)^\alpha) g^{jk} (\nu^\flat)_k \sqrt{g^{ll'} \bar \nu^l \bar \nu^{l'}}.
  \end{align*}
  By the rectifiability of the jump set we can assume without loss of generality that $\mathcal J_{\Psi^* u}$ is contained in a $C^1$-submanifold $\bar N \subset \Omega$ such that $\bar \nu$ coincides with the normal to $\bar N$.
  Hence, \NNN by \BBB Lemma \ref{lem:integration_submanifold_coords} it follows that
  \begin{align*}
    \int_O \sprod{v}{\sigma^j} \di\abs{D^j u}
    &= \int_{\mathcal J_{\Psi^* u}} g^{jk} v_j^\alpha ((\Psi^* u^+)^\alpha - (\Psi^* u^-)^\alpha) (\nu^\flat)_k \sqrt{g^{ll'} \bar \nu^l \bar \nu^{l'}} \sqrt{\abs{g}} \di \haus^{n-1} \\
    &= \int_{\mathcal J_u \cap O} \sprod{v}{(u^+ - u^-) \otimes \nu^\flat} \di \haus^{n-1}_g.
  \end{align*}
  By the arbitrariness of $v$ \NNN we derive \BBB that $D^j u \mrestr O = (u^+ - u^-) \otimes \nu^\flat \haus^{n-1}_g \mrestr (\mathcal J_u \cap O)$, as desired.
\end{proof}

\NNN
The next lemma is a extension of Lemma \ref{lem:sbv_in_coords} to the case of a general vector bundle $E$.
\begin{lemma}\label{lem:sbv_in_coords_gen}
	A section $u \in L^1_\loc(E)$ is in $BV_\loc(E)$ ($SBV_\loc(E)$, $L^p_\loc(E) \cap SBV_\loc^p(E)$) if and only if for any local trivialization $\Psi \colon \Omega \times \rr^m \to E|_O$ the pull-back $\Psi^* u$ is in $BV_\loc(\Omega; \rr^m)$ ($SBV_\loc(\Omega; \rr^m)$, $L^p_\loc(\Omega; \rr^m) \cap SBV^p_\loc(\Omega; \rr^m))$ in the usual Euclidean sense.
\end{lemma}
\BBB

\begin{proof}[Proof of Lemma \NNN\ref{lem:sbv_in_coords_gen}\BBB]
  \NNN Let $\Psi \colon \Omega \times \rr^m \to E|_O$ be local trivialization of $E$ for some $O \subset \subset M$ open.
  Suppose that $u \in BV(E|_{O})$. \BBB
  Then, $\var(u, \NNN O\BBB) < \infty$ and $\norm{u}_{L^1(E|_{\NNN O \BBB})} < \infty$ by the definition of $BV(E|_{\NNN O \BBB})$.
  Consequently, the second inequality in (\ref{variation_man_vs_coords}) implies $\var(\Psi^* u, \NNN \Omega \BBB) < \infty$.
  With (\ref{lambda_props}) we also \NNN have \BBB that $\norm{\Psi^* u}_{L^1(\NNN \Omega \BBB; \rr^m)} < C \norm{u}_{L^1(E|_{\NNN O \BBB})} < \infty$ for some constant $C$ independent of $u$.
  Hence, we have shown that $\Psi^* u \in BV(\NNN\Omega\BBB; \rr^m)$.

  Suppose now that $u \in SBV(E|_{\NNN O\BBB})$; then, by the reasoning above we already know that $\Psi^* u \in BV(\NNN\Omega\BBB; \rr^m)$.
  By (\ref{rel_cantor}) and (\ref{lambda_props}) it holds that $D^c u \mrestr \NNN O \BBB = 0$ if and only if $D^c (\Psi^* u) \mrestr \NNN\Omega\BBB = 0$.
  Hence, $\Psi^* u \in SBV(\NNN\Omega\BBB; \rr^m)$ follows.

  We have shown the forward implication of the lemma due to the arbitrariness of $\NNN O\BBB$.
  The reverse implication can be shown in similar manner.
\end{proof}

In the next lemma we will investigate a similar relationship of $SBV^p$ on the manifold and in Euclidean space.
\begin{lemma}\label{lem:sbvp_in_coords}
  Let $p \in (1, \infty)$, $q \in [p, \infty]$, and $O \subset \subset M$ be open set \NNN such that \BBB there exists a local trivialization $\Psi \colon \Omega \to E|_O$.
  Then, \NNN there \BBB exists a constant $C$ such that for all $u \in SBV^p(E|_{\NNN O \BBB}) \cap L^q(E|_{\NNN O \BBB})$ the following estimates hold true:
  \begin{align}
    \frac{1}{C} \norm{\Psi^* u}_{L^q(\NNN \Omega \BBB; \rr^m)}
      &\leq \norm{\Psi^* u}_{L^q(E|_{\NNN O\BBB})}
      \leq C \norm{\Psi^* u}_{L^q(\NNN \Omega \BBB; \rr^m)}, \label{sbvp_u}\\
    \frac{1}{C} \haus^{n-1}(\mathcal J_{\Psi^* u} \cap \NNN \Omega\BBB)
      &\leq \haus_g^{n-1}(\mathcal J_u \cap \NNN O\BBB)
      \leq C \haus^{n-1}(\mathcal J_{\Psi^* u} \cap \NNN\Omega\BBB), \label{sbvp_jump}\\
    \norm{\nabla u}_{L^p(E|_{\NNN O \BBB})} &\leq C \left(
      \norm{\nabla \Psi^* u}_{L^p(\NNN\Omega\BBB; \rr^{m \times n})} + \norm{\Psi^* u}_{L^p(\NNN\Omega\BBB; \rr^m)}
    \right) \label{sbvp_grad_1} \\
    \norm{\nabla \Psi^* u}_{L^p(\NNN\Omega\BBB; \rr^{m \times n})} &\leq C \left(
      \norm{\nabla u}_{L^p(E|_{\NNN O\BBB})} + \norm{u}_{L^p(E|_{\NNN O\BBB})}
    \right) \label{sbvp_grad_2}
  \end{align}
\end{lemma}

\begin{proof}
  Both estimates in (\ref{sbvp_u}) directly follow from (\ref{lambda_props}).
  Now\NNN, \BBB note that by Lemma \NNN \ref{lem:sbv_in_coords_gen} \BBB we \NNN have \BBB that $\Psi^* u \in SBV(E|_{\NNN O \BBB})$.
  By (\ref{lambda_props}) and the fact that $\mathcal J_u \cap \NNN O \BBB = \Phi(\mathcal J_{\Psi^* u} \cap \NNN \Omega \BBB)$ the inequalities in (\ref{sbvp_jump}) follow.

  Let us shortly write $\bar u$ for $\Psi^* u$.
  Using (\ref{approx_grad_coords_section}), a change of \NNN coordinates\BBB, (\ref{lambda_props}), and the smoothness of the Christoffel symbols we derive that
  \begin{align*}
    \int_{\NNN O \BBB} \abs{\nabla u}^p \NNN \vol \BBB
    &= \int_{\NNN \Omega \BBB} \tilde g^{\alpha \beta} g_{ij} \left[
      \Big((\nabla \bar u)_i^\alpha + \Gamma_{i\gamma}^\alpha \bar u^\gamma\Big)
      \Big((\nabla \bar u)_j^\beta + \Gamma_{j\delta}^\beta \bar u^\delta\Big)
    \right]^{\frac{p}{2}} \afac \di x \\
    &\leq C \int_{\NNN \Omega \BBB} \Big[
      \sum_{i, \alpha} \Big((\nabla \bar u)_i^\alpha + \Gamma_{i \beta}^\alpha \bar u^\beta\Big)^2
    \Big]^{\frac{p}{2}} \di x \\
    &\leq C \norm{\nabla \Psi^*u}_{L^p(\NNN \Omega \BBB; \rr^{m \times n})}^p + C \int_{\NNN \Omega \BBB} \Big[\sum_{\beta} \Big(\sum_{i, \alpha} (\Gamma_{i, \beta}^\alpha)^2 \Big) (\bar u^\beta)^2 \di x\Big]^\frac{p}{2} \di x \\
    &\leq C \left(
      \norm{\nabla \Psi^*u}_{L^p(\NNN \Omega \BBB; \rr^{m \times n})}^p
      + \norm{\Psi^*u}_{L^p(\NNN \Omega \BBB; \rr^m)}^p 
    \right).
  \end{align*}
  Taking the $p$-th root \NNN leads to \BBB (\ref{sbvp_grad_1}).
  We can show (\ref{sbvp_grad_2}) in similar fashion.
\end{proof}

\begin{proof}[Proof of Theorem \ref{thm:compactness_spec_bvsecs}]
  The result follows \NNN by \BBB a partition of unity argument as in the proof of the decomposition theorem, employing Lemma \ref{lem:sbvp_in_coords}.
\end{proof}

\section{Proof of the ball construction in open subsets}\label{sec:proofballcon}
Let $S$ be a closed, oriented, 2-dimensional Riemannian manifold.
In this appendix, we will prove the localized version of the ball-construction stated in Theorem \ref{thm:ballcon}.
As the argument closely follows the one presented in \cite{jerrard_glman} we will only sketch the necessary modifications.

The first such modification is a localized version Lemma A.3 from \cite{jerrard_glman} to an open subset of $S$\NNN.\BBB
\begin{lemma}\label{lem:hausdorff_cover}
  Let $v \in C^\infty(TO)$ for an open subset $O \subset S$ with Lipschitz boundary such that the energy upper bound from (\ref{ballcon_upper_bound}) is satisfied for some $\eps > 0$.
  Then, there exist $\eps_0 > 0$ and a constant $C > 0$ independent of $v$ or $\eps$ such that whenever $\eps \in (0, \eps_0)$ we can find a finite collection of closed, pairwise disjoint balls $\balls = \setof{B_j}$ whose union covers $\setof{x \in O \colon \abs{v(x)} \leq \frac{1}{2}}$, and such that
  \begin{equation*}
    \sum_j r_j \leq C \eps \leps,
  \end{equation*}
  where \NNN $r_j$ is \BBB the radius of $B_j$.
\end{lemma}
\begin{proof}
  As was done in the proof of Lemma A.3 in \cite{jerrard_glman} using the coarea formula we can find a regular value $\alpha \in [\tfrac{1}{2}, \tfrac{3}{4}]$ of $\abs{v}$ such that
  \begin{equation*}
    \haus^1_g(\setof{x \in O \colon \abs{v(x)} = \alpha}) \leq C \eps \leps,
  \end{equation*}
  for some constant $C > 0$ independent of $v$ and $\eps$.
  Further\NNN more\BBB, using the Lipschitz-regularity of $\partial O$, we can find a constant $\tilde C > 0$ only depending on $\partial O$ such that
  \begin{equation*}
    \haus^1_g(\partial \setof{x \in O \colon \abs{v(x)} \leq \alpha}) \leq \tilde C \haus^1_g(\setof{x \in O \colon \abs{v(x)} = \alpha}).
  \end{equation*}
  Combining both estimates we discover by the very definition of $\haus^1_g$ that there exists a countable cover of $\partial \setof{x \in O \colon \abs{v} \leq \alpha}$ whose radii sum up to at most $C \eps \leps$.
  By compactness we can reduce ourselves to a finite cover and using a standard merging procedure to a disjoint cover without increasing the sum of all radii.
\end{proof}

We continue by showing that (123) from \cite{jerrard_glman} is still satisfied if we replace $S$ by an open subset of $S$.
\begin{lemma}\label{lem:lower_bound_circle}
  There exists $r_0 = r_0(O) < r^*$ depending on the geometry of $\partial O$ and $C > 0$ such that for any $v \in C^\infty(TO)$, $x \in O$, and $r \in (\eps, r_0)$ we have
  \begin{equation}\label{lower_bound_circles}
    \frac{1}{2} \int_{\partial B_r(x) \cap O} \abs{\di\abs{v}}^2 + \frac{1}{2\eps^2} \NNN(1-\abs{v}^2)^2\BBB \di\!\haus^1_g \geq \frac{C}{\eps} \norm{1 - \abs{v}}_{L^\infty(\partial B_r(x))}.
  \end{equation}
\end{lemma}
\begin{proof}
  By the Lipschitz regularity and compactness of $\partial O$ we can find $\alpha \in (0, \pi)$ and $\tilde r_0 \in (0, r^*)$ such that for all $y \in \partial O$ there exist a unit-length vector $\nu \in T_y S$ satisfying
  \begin{equation*}
    \partial O \cap B_r(y) \subset C_{\alpha, r}(y, \nu)
  \end{equation*}
  for all $r \in (0, \tilde r_0)$, where
  \begin{equation*}
    C_{\alpha, r}(y, \nu) \defas \exp_y(\setof{X \in T_y S \colon \abs{X} < r, \, \abs{\sprod{X}{\nu}} \leq \cos(\tfrac{\alpha}{2}) \abs{\sprod{X}{\nu^\perp}} }).
  \end{equation*}
  Let us set $r_0 \defas \frac{\tilde r_0}{2}$ and consider $x \in O$ and $r \in (0, r_0)$.
  Suppose that $B_r(x) \setminus O \neq \emptyset$.
  Then, we can find $y \in \partial O \cap B_r(x)$.
  Note that $B_r(x)$ is contained in $B_{2r}(y)$.
  As $2r < 2 r_0 \leq \tilde r_0$ we can find by our choice of $\tilde r_0$ a unit-length vector $\nu \in T_y S$ such that the set $A \defas (\partial B_r(x) \cap O) \setminus C_{\alpha, 2r}(y, \nu)$ is connected.
  Hence, by the compactness of $S$, we can find a constant $C > 0$ only depending on $\alpha$ and $S$ such that
  \begin{equation}\label{min_portion_inside}
    \haus^1_g(A) \geq C r.
  \end{equation}

  Let us set $\zeta \defas (1 - \abs{v})^2$ on $A$.
  Using Young's inequality \NNN we \BBB derive that
  \begin{equation*}
    \abs{\zeta'} + \frac{1}{\eps} \abs{\zeta} \leq C \eps\left(\abs{\di\abs{v}}^2 + \frac{1}{2\eps^2} \NNN(1-\abs{v}^2)^2\BBB\right).
  \end{equation*}
  Here, $(\cdot)'$ denotes the differential in tangential direction of $\partial B_r(x)$.
  Let us now select a point $z \in A$ such that $\zeta(z) = \fint_{A} \zeta \di\!\haus^1_g$.
  Using the connectedness of $A$, the fundamental theorem of calculus, and (\ref{min_portion_inside}) it then follows that
  \begin{align*}
    \norm{\zeta}_{L^\infty}
    \leq \abs{\zeta(z)} + \int_{A} \abs{\zeta'} \di\!\haus^1_g 
    \leq \int_{A} \frac{1}{Cr} \abs{\zeta} + \abs{\zeta'} \di\!\haus^1_g 
    &\leq C \int_{A} \frac{1}{\eps} \abs{\zeta} + \abs{\zeta'} \di\!\haus^1_g \\
    &\leq C \eps \int_{A} \abs{\di\abs{v}}^2 + \frac{1}{2\eps^2} \NNN(1-\abs{v}^2)^2\BBB \di\!\haus^1_g,
  \end{align*}
  which is (\ref{lower_bound_circles}).

  The argument for the remaining case $B_r(x) \subset O$ follows as in (\cite{jerrard_glman}).
\end{proof}

In the same way as was done in \cite{jerrard_glman} we define $\Lambda_\eps \colon [0, \infty) \to \rr$ by
\begin{equation*}
  \Lambda_\eps(\sigma) \defas \int_0^\sigma \lambda_\eps(r) \NNN \di r, \BBB \NNN\qquad \text{where } \lambda_\eps(r) \defas \min_{0 < s \leq 1} \left[ \frac{c_2}{4\eps} (1 - s)^2 + s^2 \frac{\pi}{r} (1 - c_2 r^2) \right],
\end{equation*}
for constants $c_2, \, c_3$ as in (125) from Lemma A.1 in \cite{jerrard_glman}.

\begin{proof}[Sketch of the proof of Theorem \ref{thm:ballcon}]
  As our argument mostly coincides with the one provided in the proof of Proposition 8.2 in \cite{jerrard_glman} we only briefly sketch the main differences.
  Employing Lemma \ref{lem:lower_bound_circle} and the same Besicovitch covering argument as in \cite{jerrard_glman} (see the proof of Proposition 8.2) we can find an initial cover of $Z_E$ with a finite family $\balls = \setof{B_k}_k^K$ of pairwise disjoint, closed, geodesic balls each with radius denoted by $r_k$ such that $\sum_{k = 0}^K r_k \leq C \eps \leps$ for some universal constant $C > 0$ and
  \begin{equation*}
    \frac{1}{2} \int_{B_k} \abs{\nabla v}^2 + \frac{1}{2\eps^2} F(\abs{v}) \vol \geq \Lambda_\eps(r_{k, 0})
  \end{equation*}
  for any $B_k \subset O$.

  Given $k = 1, \dots, K$ we set $d_k \defas \dg(v, B_k; O)$ (see (\ref{def_local_degree})) and define
  \begin{equation*}
    \sigma^* \defas \min_{d_k \neq 0} \NNN\frac{r_k}{\abs{d_k}}\BBB.
  \end{equation*}
  Starting from the family $\balls^{(\sigma^*)} \defas \balls$ we grow and merge every ball that does not intersect $\partial O$ according to the standard ball-construction algorithm, while we leave unchanged all the remaining balls.
  For every $\sigma \in (\sigma^*, \sigma_0)$, where $\sigma_0$ is as in the statement of Lemma \ref{lem:lower_bound_circle}, this produces a finite family of pairwise disjoint, geodesic balls $\balls^{(\sigma)} = \setof{B_k^{(\sigma)}}$ each with radius $r_k^{(\sigma)}$ and degree $d_k^{(\sigma)} \defas \dg(v, \partial B; O)$ such that $r_k^{(\sigma)} \geq \sigma \abs{d_k^{(\sigma)}}$ for all $k$ and
  \begin{equation*}
    \frac{1}{2} \int_{B_k^{(\sigma)}} \abs{\nabla v}^2 + \frac{1}{2 \eps^2} \NNN(1-\abs{v}^2)^2\BBB \vol \geq \frac{r_k^{(\sigma)}}{\sigma} \Lambda_\eps(\sigma),
  \end{equation*}
  as long as $B_k^{(\sigma)} \subset O$.

  We conclude by following exactly the same lines of the proof of Proposition 8.2 in \cite{jerrard_glman}, since Lemma \ref{lem:hausdorff_cover} provides the necessary extension of Lemma A.3 in \cite{jerrard_glman}.
\end{proof}

\noindent \textbf{Acknowledgements} R.~Badal acknowledges support by the DFG project FR 4083/5-1. The work of M.~Cicalese was supported by the DFG Collaborative Research Center TRR 109, “Discretization in Geometry and Dynamics”.

\typeout{References}

\end{document}